\pdfoutput=1
\documentclass[11pt,a4paper,reqno]{amsart}

\usepackage{times}
\usepackage{xspace}
\usepackage[inline]{enumitem}
\usepackage[numbers,sort&compress,longnamesfirst]{natbib}
\usepackage{geometry}
\usepackage[colorlinks=true, linkcolor=blue ,citecolor=blue]{hyperref}
\usepackage[capitalize]{cleveref}
\usepackage{dsfont}
\usepackage{enumitem}

\usepackage{amsfonts,comment}
\usepackage{amsmath,mathtools}
\usepackage{amssymb,amsmath,latexsym}
\numberwithin{equation}{section}
\usepackage{ color, graphicx}
\usepackage{todonotes}
\usepackage{stmaryrd}
\usepackage[mathscr]{euscript}

\def\<{\langle}
\def\>{\rangle}

\newcommand{\be}{\begin{equation}}
\newcommand{\ee}{\end{equation}}
\newcommand{\bea}{\begin{eqnarray}}
\newcommand{\eea}{\end{eqnarray}}
\newcommand{\beas}{\begin{eqnarray*}}
\newcommand{\eeas}{\end{eqnarray*}}
\theoremstyle{plain}
\newtheorem{theorem}{Theorem}[section]
\newtheorem{lemma}[theorem]{Lemma}
\newtheorem{proposition}[theorem]{Proposition}
\newtheorem{corollary}[theorem]{Corollary}

\newtheorem{definition}[theorem]{Definition}

\theoremstyle{definition}
\newtheorem{remark}[theorem]{Remark}

\newtheorem*{notation}{Notation}
 
\newcommand{\tnorm}[1]{\left\vert\kern-0.25ex\left\vert\kern-0.25ex\left\vert #1 
    \right\vert\kern-0.25ex\right\vert\kern-0.25ex\right\vert}
\newcommand{\numberthis}{\addtocounter{equation}{1}\tag{\theequation}}

\def\\ud sum{\displaystyle\sum}
\def\Prod{\displaystyle\prod}

\numberwithin{equation}{section}
\numberwithin{table}{section}
\numberwithin{figure}{section}

\def\ue{\mathrm e}
\def\ud{\mathrm d}

\setlist[enumerate]{label=(\roman*),align=left, leftmargin=*}
\crefname{subsection}{subsection}{subsections}

\allowdisplaybreaks

\title{Stability of backward propagation of chaos}

\author[A. Papapantoleon]{Antonis Papapantoleon}
\author[A. Saplaouras]{Alexandros Saplaouras}
\author[S. Theodorakopoulos]{Stefanos Theodorakopoulos}

\address{Delft Institute of Applied Mathematics, EEMCS, TU Delft, 2628 Delft, The Netherlands \& Department of Mathematics, School of Applied Mathematical and Physical Sciences, National Technical University of Athens, 15780 Zografou, Greece \& Institute of Applied and Computational Mathematics, FORTH, 70013 Heraklion, Greece}
\email{a.papapantoleon@tudelft.nl}

\address{Department of Mathematics, ETH Zurich, 8092 Zurich, Switzerland}
\email{alexandros.saplaouras@math.ethz.ch}

\address{Institut f\"{u}r Mathematik, Technische Universit\"{a}t Berlin, 10623 Berlin, Germany}
\email{stefanos.theodorakopoulos@tu-berlin.de}

\thanks{AS gratefully acknowledges the financial support from the Hellenic Foundation for Research and Innovation Grant No. 235 (2nd Call for H.F.R.I. Research Projects to support Post-Doctoral Researchers).
ST gratefully acknowledges the financial support from the Hellenic Foundation for Research and Innovation Grant No. 05724 (3rd Call for H.F.R.I. Scholarships for PhD candidates).}



\date{}

\begin{document}

\begin{abstract}
The purpose of the present paper is to introduce and establish a notion of stability for the backward propagation of chaos with respect to (initial) data sets. 
Consider, for example, a sequence of discrete-time martingales converging to a continuous-time limit, and a system of mean-field BSDEs that satisfies the backward propagation of chaos, \textit{i.e.} converges to a sequence of i.i.d. McKean--Vlasov BSDEs.
Then, we say that the backward propagation of chaos is stable if the system of mean-field BSDEs driven by the discrete-time martingales converges to the sequence of McKean--Vlasov BSDEs driven by the continuous-time limit.
We consider the convergence scheme of the backward propagation of chaos as the image of the corresponding data set under which this scheme is established. 
Then, using an appropriate notion of convergence for data sets, we are able to show a variety of continuity properties for this functional point of view.
Along the way, we also provide stability results for mean-field and McKean--Vlasov BSDEs, which are of interest in their own right, for numerical approximations of these equations.
\end{abstract}

\maketitle\frenchspacing



\section{Introduction}

The backward propagation of chaos states that, under appropriate conditions, the solution of a mean-field system of backward stochastic differential equations (BSDEs) with $N$ players (or particles) converges to the solutions of $N$ independent and identically distributed McKean--Vlasov BSDEs, as the number of players $N$ goes to infinity. 
The propagation of chaos property has been extensively studied for forward stochastic differential equations (SDEs), see \textit{e.g.} the review articles of \citet{chaintron2022propagation,Chaintron_Diez_2022}, while only a handful of papers have been published so far on the backward propagation of chaos; see \citet{Buckdahn_Djehiche_Li_Peng_2009,Hu_Ren_Yang_2023,djehiche2021propagation,buckdahn2009mean, LiDu2024, BayraktarWuZhang2023,lauriere2022backward} and also the recent paper \citet{PST_BackPropagChaos} that contains a more detailed review of this literature.
Every such phenomenon is, of course, associated with a set of data $\mathscr{D}$ that provides the basis for its mathematical description. 
Consider a well-posed equation and assume that the sequence of data $\{\mathscr{D}^k\}_{k \in \mathbb{N}}$ converges to the data $\mathscr{D}^{\infty}$; then, if the corresponding solutions also converge, we say that the equation is stable. 
Naturally, the framework to answer this kind of questions should determine a lot of the technical details; for example, in what sense do the data converge, how do we measure the distance of the solutions, and so on. 
In the theory of BSDEs, frameworks for the stability of BSDEs are provided by \citet{hu1997stability} in the special case of a constant filtration, by \citet{briand2001donsker,briand2002robustness} for Brownian drivers, and more recently by \citet{papapantoleon2023stability} where a very general framework is established.
The aim of the present paper is to introduce and establish a stability result for the backward propagation of chaos. 
Along the way, we will also introduce and establish stability results for systems of mean-field BSDEs as well as for McKean--Vlasov BSDEs.

In order to stimulate the interest of the reader with a simple example, let us consider the framework for Donsker's theorem.
More specifically, we associate with each player an independent sequence of symmetric random walks converging to a Brownian motion. 
When these symmetric random walks are sufficiently close to their corresponding Brownian motion, and when the number of players is sufficiently large, the stability of backward propagation of chaos ensures that the solution of the mean-field system of BSDEs driven by the symmetric random walks approximates the solution of the McKean--Vlasov BSDEs driven by the Brownian motions.

The main motivation for this work comes from the numerical solution of mean-field control problems or mean-field equilibria, which can often be recast as mean-field or McKean--Vlasov (F)BSDEs; see, for example, \citet{Chassagneux2022,10.1214/18-AAP1429} or \citet{Warin_etal_2022} for numerical schemes and also \citet{Bensoussan2013,Carmona2012} and \citet{Carmona_Delarue_2018_I,Carmona_Delarue_2018_II} for the theoretical background.
In these situations, our results offer general convergence results for discretization schemes at three levels: for the mean-field systems of BSDEs, for the associated McKean--Vlasov BSDEs and, in a unified step, from the discretized mean-field BSDEs to the associated limiting McKean--Vlasov BSDEs.
Hence, a variety of numerical schemes can be immediately applied on the basis of these results, while the study of convergence rates for these schemes is a logical next step; see also the recent article of \citet{djehiche2025convergenceraterandomwalk}.



In this work, we will combine the frameworks of \citet{PST_BackPropagChaos} and \citet{papapantoleon2023stability} in order to study the stability of the backward propagation of chaos, in a general setting that includes BSDEs driven by either discrete-time or continuous-time martingales with independent increments. 
Let us describe what we mean by that. 
Set $\overline{\mathbb{N}} := \mathbb{N} \cup \{\infty\}$ and assume a sequence of standard data 
$\{\mathscr{D}^k\}_{k \in \overline{\mathbb{N}}}$, where $\mathscr{D}^k := \Big(\{\overline{X}^{k,i}\}_{i \in \mathbb{N}},T^k,\left\{\{\xi^{k,i,N}\}_{i \in \mathscr{N}}\right\}_{N \in \mathbb{N}},\{\xi^{k,i}\}_{i \in \mathbb{N}},\Theta^k,\Gamma,f^k \Big)$ as defined in \cite[Definition~4.2]{PST_BackPropagChaos}, \emph{i.e.}, the set of Assumptions \ref{H1}--\ref{H:prop_contraction} (presented in \cref{sec:Setting}) are set in force.   
Let $(k,N) \in \overline{\mathbb{N}} \times \mathbb{N}$ and $i = 1,\dots,N,$ then the following system of mean-field BSDEs under the standard data $\mathscr{D}^k$,
for $t\in[0,T^k]$,
\begin{align}
\begin{multlined}[0.9\textwidth]
Y^{k,i,N}_t = \xi^{k,i,N} + \int^{T^k}_{t}f^k \left(s,Y^{k,i,N}_s,Z^{k,i,N}_s  c^k_s,\Gamma^{(\mathbb{F}^{k,(1,\dots,N)},\overline{X}^{k,i},\Theta^k)}(U^{k,i,N})_s,L^N(\textbf{Y}^{k,N}_s) \right) \, \ud C^{k}_s\\ 
- \int^{T^k}_{t}Z^{k,i,N}_s \,  \ud X^{k,i,\circ}_s - \int^{T^k}_{t}\int_{\mathbb{R}^n}U^{k,i,N}_s(x) \, \widetilde{\mu}^{(\mathbb{F}^{k,(1,\dots,N)},X^{k,i,\natural})}(\ud s,\ud x) - \int^{T^k}_{t} \,\ud M^{k,i,N}_s,  
\label{mfBSDE}
\end{multlined}
\end{align}
is well-defined and has the unique solution $\textbf{S}^{k,N} : = \left(\textbf{Y}^{k,N},\textbf{Z}^{k,N},\textbf{U}^{k,N},\textbf{M}^{k,N}\right)$, where in particular
$\textbf{Y}^{k,N} := \left(Y^{k,1,N},\dots,Y^{k,N,N}\right)$; we use analogous notation for $\textbf{Z}^{k,N}$, $\textbf{U}^{k,N}$ and $\textbf{M}^{k,N}$.
Here, $N$ denotes the total number of players in the mean-field system and $i$ denotes one of these players, while $L^N$ denotes the empirical measure associated to $\textbf{Y}^{k,N}_t$, \textit{i.e.}
\[
    L^N(\textbf{Y}^{k,N}_t):= \frac1N \sum_{i=1}^N \delta_{Y^{k,i,N}_t}.
\]
The notation will be clarified in the next section, however let us quickly state the following:
$\xi^{k,i,N}$ is the terminal condition of the BSDE for the $i$-th player in the $N$-th game for the $k$-th data set, 
$\overline{X}^{k,i}=(X^{k,i,\circ},X^{k,i,\natural})$ is a tuple of square-integrable martingales (where one could think of them as the continuous and purely discontinuous part of a martingale) with quadratic covariation $C^{\overline{X}^{k,i}}$, 
$M^{k,i,N}$ is the orthogonal martingale with respect to $\overline{X}^{k,i}$,
$Z^{k,i,N}$ is the integrand of the It\=o integral with respect to $X^{k,i,\circ}$,
$U^{k,i,N}$ is the integrand of the stochastic integral with respect to the random measure associated with $X^{k,i,\natural}$,
and $\Gamma$ is a suitable function of $U^{k,i,N}$ chosen from a large class.
In particular, the $\Gamma$ function can be chosen as the Radon--Nikodym derivative of the stochastic integral of $U^{k,i,N}$ with respect to the jump measure relative to the measure induced by $C^{\overline{X}^{k,i}}$, for a range of free parameters $\Theta$; see \Cref{sec:gamma}.

In addition, for $(k,i) \in \overline{\mathbb{N}} \times \mathbb{N}$, the following McKean--Vlasov BSDE under the standard data $\mathscr{D}^k$, for $t\in[0,T^k]$,
\begin{align}
\begin{multlined}[0.9\textwidth]
Y^{k,i}_t = \xi^{k,i} + \int^{T^k}_{t}f^k\left(s,Y^{k,i}_s,Z^{k,i}_s  c^{k}_s,\Gamma^{(\mathbb{F}^{k,i},\overline{X}^{k,i},\Theta^k)}(U^{k,i})_s,\mathcal{L}(Y^{k,i}_s)\right) \, \ud C^{k}_s\\
- \int^{T^k}_{t}Z^{k,i}_s \,  \ud X^{k,i,\circ} - \int^{T^k}_{t}\int_{\mathbb{R}^n}U^{k,i}_s \, \widetilde{\mu}^{(\mathbb{F}^{k,i},X^{k,i,\natural})}(\ud s,\ud x) - \int^{T^k}_{t} \,\ud M^{k,i}_s,
\label{MVBSDE}
\end{multlined}
\end{align}
is well-defined and has the unique solution $(Y^{k,i},Z^{k,i},U^{k,i},M^{k,i})$, where $\mathcal{L}(Y^{k,i}_t)$ denotes the law of $Y^{k,i}_t$.

We have shown in \citet[Theorem 6.3]{PST_BackPropagChaos} that the system of mean-field BSDEs satisfies a backward propagation of chaos property, \textit{i.e.} we know that for every $(k,i) \in \overline{\mathbb{N}} \times \mathbb{N}$ we have 
\begin{align}
   \big\|\big(Y^{k,i,N} - Y^{k,i},Z^{k,i,N} - Z^{k,i},U^{k,i,N} - U^{k,i},M^{k,i,N} - M^{k,i}\big) \big\|
   \xrightarrow[N \rightarrow \infty]{} 0,
\end{align}
under some suitable norm.
Our goal now is to show that the property of backward propagation of chaos is also stable, namely, we want to show that for every $i \in \mathbb{N}$ holds
\begin{multline*}
    \left(Y^{k,i,N}, Z^{k,i,N} \cdot X^{k,i,\circ} + U^{k,i,N} \star \widetilde{\mu}^{X^{k,i,\natural}},M^{k,i,N}\right)\\  \xrightarrow[(k,N)
      \rightarrow (\infty,\infty)]{
      } \left(Y^{\infty,i}, Z^{\infty,i} \cdot X^{\infty,i,\circ} + U^{\infty,i} \star \widetilde{\mu}^{X^{\infty,i,\natural}},0\right),
\end{multline*}
under some appropriate metric.
To the best of our knowledge, this will be the first result of this kind. 
In addition, we will also provide stability results for the systems of mean-field BSDEs in \eqref{mfBSDE} and for McKean--Vlasov BSDEs in \eqref{MVBSDE}.


In order to accommodate the reader, we provide in \cref{table:propagation scheme} a roadmap of our approach for the proof of the stability of the backward propagation of chaos. 
Analogous roadmaps for the other stability results are provided in the respective sections with the proofs; see \cref{table:McKean--Vlasovscheme,table:mean-fieldscheme}.
%
Using this table, we will describe the implications of the setting inherited from \citet{PST_BackPropagChaos}, \emph{i.e.}, the set of Assumptions \ref{H1}--\ref{H:prop_contraction} presented in \cref{sec:Setting}, as well as the forthcoming results we intend to prove.
The validity, for every $k\in\overline{\mathbb{N}}$, of the backward propagation of chaos is denoted by a solid, horizontal right arrow in \cref{table:propagation scheme}. 
The set of Assumptions \ref{S1}--\ref{S10}, which is presented in \cref{sec:Setting} and amounts to the convergence of the sequence of standard data, is denoted by the solid, vertical down arrow in the first column of the table.
Now, we claim that the framework we have already set allows for the validity, on the one hand, of the stability of McKean--Vlasov BSDEs and, on the other hand, of the uniform (over $k\in\overline{\mathbb{N}}$) backward propagation of chaos.
The former is denoted by the dashed, vertical down arrow in the last column of the table.
The latter was chosen not to be represented in \cref{table:propagation scheme} in order to keep the table as simple as possible.
The conjunction of the two aforementioned results yields the convergence of the doubly indexed sequence of solutions of mean-field BSDEs $\{\mathscr{S}^{k,i,N}\}_{(k,N)\in\mathbb{N}\times \mathbb{N}}$ to the solution of McKean--Vlasov BSDEs $\mathscr{S}^{\infty,i}$, for every $i\in\mathbb{N}$; this convergence is denoted by the wiggly, diagonal arrow.
%
%
%
%
%
\begin{table}[ht]
{\centering 
\begin{tabular}{c|ccccccc}
$\mathscr{D}^1$ 	&{$\mathscr{S}^{1,i,i}$} 			&{$\mathscr{S}^{1,i,i + 1}$} 			&{$\mathscr{S}^{1,i,i + 2}$}		
 					&{$\cdots$}							&{$\mathscr{S}^{1,i,N}$}			&{$\xrightarrow{\ N\to\infty\ }$} 	&{$\mathscr{S}^{1,i}$}\\[0.3cm]
$\mathscr{D}^2$ 	&{$\mathscr{S}^{2,i,i}$} 			&{$\mathscr{S}^{2,i,i+1}$} 			&{$\mathscr{S}^{2,i,i+2}$}
 					&{$\cdots$}							&{$\mathscr{S}^{2,i,N}$}			&{$\xrightarrow{\ N\to\infty\ }$}	&{$\mathscr{S}^{2,i}$}\\[0.3cm]
$\mathscr{D}^3$ 	&{$\mathscr{S}^{3,i,i}$} 			&{$\mathscr{S}^{3,i,i+1}$} 			&{$\mathscr{S}^{3,i,i+2}$}
					&{$\cdots$}							&{$\mathscr{S}^{3,i,N}$}			&{$\xrightarrow{\ N\to\infty\ }$}	&{$\mathscr{S}^{3,i}$}\\[0.1cm]
$\vdots$			&{$\vdots$} 						&{$\vdots$} 						&{$\vdots$} 						& 
					&{$\vdots$}							& 												&{$\vdots$} \\[0.2cm]
$\mathscr{D}^k$ 	&{$\mathscr{S}^{k,i,i}$} 			&{$\mathscr{S}^{k,i,i+1}$} 			&{$\mathscr{S}^{k,i,i+2}$}			&{$\cdots$}
					&{$\mathscr{S}^{k,i,N}$}			&{$\xrightarrow{\ N\to\infty\ }$}	&{$\mathscr{S}^{k,i}$}\\
$\big\downarrow$	
&{$ $} 	&{$ $}  	&{$ $}  	&  		
					&{$ $}  	&	\rotatebox[origin=c]{-45}{\huge${\rightsquigarrow}$}	& \rotatebox[origin=c]{270}{$\dashrightarrow{}$} \\[0.2cm]
$\mathscr{D}^\infty$&{$\mathscr{S}^{\infty,i,i}$} 		&{$\mathscr{S}^{\infty,i,i+1}$} 		&{$\mathscr{S}^{\infty,i,i+2}$}		&{$\cdots$}
					&{$\mathscr{S}^{\infty,i,N}$}		&{$\xrightarrow{\ N\to\infty\ }$}	&{$\mathscr{S}^{\infty,i}$}\\[0.2cm]
\end{tabular}\\}
\caption{The doubly-indexed scheme for the stability of backward propagation of chaos.}
\label{table:propagation scheme}
\end{table}

In order to prove the stability of backward propagation of chaos, we first show in \cref{subsec:aux_PropagChaos_Stability} that, based on results from \citet{PST_BackPropagChaos}, the propagation of chaos property holds uniformly with respect to the date $\mathscr{D}^k$. 
Next, by extending the stability results of \citet{papapantoleon2023stability} to the McKean--Vlasov BSDEs, we can immediately deduce the desired result. 
At this point it is important to note that the stability of the mean-field systems of BSDEs was neither needed nor is implied from the stability of the propagation of chaos. 
That is the reason why there are no vertical arrows for the main body of the table.
The stability result for mean-field BSDEs is proved in addition, under slightly different conditions to the other two stability results.

This paper is organized as follows:
In \cref{sec:preliminaries}, we provide notation and preliminary results that are required for the remainder of this work.
In \cref{sec:Setting}, we provide the setting so that we can prove the main results on stability for McKean--Vlasov BSDEs and the backward propagation of chaos.
In \cref{sec:StabilityProperties}, we provide the exact statements as well as the outline of their proofs.
Finally, in \cref{sec:proofs} we provide the remaining proofs and other auxiliary results.


\section{Preliminaries}
\label{sec:preliminaries}

In this section, we are going to introduce the notation as well as to give a brief overview of known results that will be used in the current work. Although we will provide explanations and references for these results, one can also consult \citet{PST_BackPropagChaos} for more details.
Let $(\Omega,\mathcal{G},\mathbb{G},\mathbb{P})$ denote a complete stochastic basis in the sense of \citet[Chapter I, Definition 1.3]{jacod2013limit}. 
Once there is no ambiguity about the reference filtration, we will conceal the dependence on $\mathbb{G}$.
The letters $p,n$ and $d$ will denote arbitrary natural numbers. 
Let $q \in \mathbb{N}^*$, then we will denote by $|\cdot|$ the Euclidean norm on $\mathbb{R}^q$. 
Moreover, for every $(m,q) \in \mathbb{N}^*\times \mathbb{N}^*$ we will denote by $\mathbb{R}^{m\times q}$ the $m\times q-$matrices with real entries. 
We will also denote by $\|\cdot\|_2$ the Euclidean norm on $\mathbb{R}^{m\times q}$, that is, for $z \in \mathbb{R}^{m\times q}$ we have $\|z\|_2^2 := \text{Tr}[z^\mathsf{T}z]$. 
Notice that we identify $\mathbb{R}^q$ as $\mathbb{R}^{q \times 1}$. In addition, we will use abusively the notation $\mathscr{N} := \{1,\dots,N\}$, for $N\in \mathbb{N}$.


\subsection{Martingales}

Let us denote by $\mathcal{H}^2(\mathbb{G};\mathbb{R}^p)$ the set of square integrable $\mathbb{G}-$martingales, \emph{i.e.}, 
\begin{align*}
\mathcal{H}^2(\mathbb{G};\mathbb{R}^p)
    :=\big\{ X:\Omega\times\mathbb{R}_+ {}\longrightarrow \mathbb{R}^p, X \text{ is a $\mathbb{G}-$martingale with }\sup_{t\in\mathbb{R}_+}\mathbb E[\vert X_t\vert^2]<\infty \big\}, 
\end{align*}
equipped with its usual norm 
\begin{align*}
\Vert X\Vert^2_{\mathcal{H}^2(\mathbb{G};\mathbb{R}^p)}
    := \mathbb{E}\big[ \vert X_\infty \vert^2\big]
     = \mathbb{E}\left[{{\rm Tr}[\langle X\rangle_\infty^{\mathbb{G}}]}\right],
\end{align*}
where $\langle X \rangle^{\mathbb{G}}$ denotes the $\mathbb{G}-$predictable quadratic variation of $X$, which is the $\mathbb{G}-$predictable compensator of the $\mathbb{G}-$optional quadratic variation $[X]$.

Let us also define a notion of orthogonality between two square--integrable martingales. 
More precisely, we say that $X\in \mathcal{H}^2(\mathbb{G};\mathbb{R}^p)$ and $Y \in \mathcal{H}^2(\mathbb{G};\mathbb{R}^q)$ are (mutually) orthogonal if and only if $\langle X,Y\rangle^{\mathbb{G}}=0$, and we denote this relation by $X \perp_{\mathbb{G}} Y$. 
Moreover, for a subset $\mathcal{X}$ of $\mathcal{H}^2(\mathbb{G};\mathbb{R}^p)$, we denote the space of martingales orthogonal to each component of every element of $\mathcal{X}$ by $\mathcal{X}^{\perp_{\mathbb{G}}}$, \emph{i.e.},
\begin{align*}
\mathcal{X}^{\perp_{\mathbb{G}}} 			
  &:=\big\{ Y\in\mathcal{H}^2(\mathbb{G};\mathbb{R}),\   \langle Y,X\rangle^{\mathbb{G}}=0\text{ for every } X\in\mathcal{X} \big\}.
\end{align*}
A martingale $X\in\mathcal{H}^2(\mathbb{G};\mathbb{R}^p)$ will be called a \emph{purely discontinuous} martingale if $X_0=0$ and if each of its components is orthogonal to all continuous martingales of $\mathcal{H}^2(\mathbb{G};\mathbb{R}).$ 
Using \cite[Chapter I, Corollary 4.16]{jacod2013limit} we can decompose $\mathcal{H}^2(\mathbb{G};\mathbb{R}^p)$ as follows
\begin{align}\label{SqIntMartDecomp}
\mathcal{H}^{2}(\mathbb{G};\mathbb{R}^p)
=\mathcal{H}^{2,c}(\mathbb{G};\mathbb{R}^p)\oplus\mathcal{H}^{2,d}(\mathbb{G};\mathbb{R}^p),
\end{align}
where $\mathcal{H}^{2,c}(\mathbb{G};\mathbb{R}^p)$ denotes the subspace of $\mathcal{H}^2(\mathbb{G};\mathbb{R}^p)$ consisting of all continuous square--integrable martingales and $\mathcal{H}^{2,d}(\mathbb{G};\mathbb{R}^p)$ denotes the subspace of $\mathcal{H}^2(\mathbb{G};\mathbb{R}^p)$ consisting of all purely discontinuous square--integrable martingales. 



\subsection{It\=o stochastic integrals}

Using \cite[Chapter III, Section 6.a]{jacod2013limit}, in order to define the stochastic integral with respect to a square--integrable martingale $X \in \mathcal{H}^2(\mathbb{G};\mathbb{R}^p)$, we need to select a $\mathbb{G}-$predictable, non-decreasing and right continuous process $C^{\mathbb{G}}$ with the property that 
\begin{align}\label{comp_property}
\langle X\rangle^{\mathbb{G}} 
    = \int_{(0,\cdot]} \frac{\ud \langle X\rangle_s^{\mathbb{G}}}{\ud C^{\mathbb{G}}_s}\ud C_s^{\mathbb{G}},
\end{align} 
where the equality is understood component-wise.
That is to say, $\frac{\ud \langle X\rangle^{\mathbb{G}}}{\ud C^{\mathbb{G}}}$ is a predictable process with values in the set of symmetric non--negative definite $p\times p$ matrices, that $\frac{\ud \langle X\rangle^{\mathbb{G}}}{\ud C^{\mathbb{G}}}$ is well defined can be easily deduced from the predictable section theorem; a short and elementary proof of the former can be found in \citet{THEODORAKOPOULOS2025110324}.
Then, we define the set of integrable processes as
\begin{align*}
\mathbb H^2(\mathbb{G},X;\mathbb{R}^{d \times p}) 
    &:= \left\{ Z:(\Omega\times\mathbb{R}_+,\mathcal{P}^{\mathbb{G}}) 
    {}\longrightarrow (\mathbb{R}^{d\times p},\mathcal{B}(\mathbb{R}^{d\times p})),\  
    \mathbb{E}\left[{\int_0^\infty {\rm Tr}\left[Z_t \frac{\ud \langle X\rangle_s^{\mathbb{G}}}{\ud C_s^{\mathbb{G}}}Z_t^{\top}\right] \ud C_t^{\mathbb{G}}} \right] <\infty \right\},
\end{align*}
where $\mathcal{P}^{\mathbb{G}}$ denotes the $\mathbb{G}-$predictable $\sigma-$field on $\Omega\times\mathbb{R}_+$; see \cite[Chapter I, Definition 2.1]{jacod2013limit}. 
The associated stochastic integrals will be denoted by $Z\cdot X$ or by $\int Z_s \textup{d} X_s$.
In case we need to underline the filtration under which the It\=o stochastic integral is defined, we will write either $(Z\cdot X)^{\mathbb{G}}$ or $(\int Z_s \ud X_s)^{\mathbb{G}}$.
We will denote the space of It\=o stochastic integrals of processes in $\mathbb H^2(\mathbb{G},X)$ with respect to $X$ by $\mathcal{L}^2(\mathbb{G},X)$.
In particular, for $X^c\in\mathcal{H}^{2,c}(\mathbb{G};\mathbb{R}^d)$ we remind the reader that, by \cite[Chapter III, Theorem 4.5]{jacod2013limit}, $Z\cdot X^c\in\mathcal{H}^{2,c}(\mathbb{G};\mathbb{R}^d)$ for every $Z\in\mathbb H^2(X^c,\mathbb{G})$, \emph{i.e.}, $\mathcal{L}^2(X^c,\mathbb{G})\subset\mathcal{H}^{2,c}(\mathbb{G};\mathbb{R}^d)$.


\subsection{Integrals with respect to an integer--valued random measure}

Let us now expand the space, and accordingly the predictable $\sigma$--algebra, in order to construct measures that depend also on the height of the jumps of a stochastic process, that is
\begin{align*}
\big(\widetilde{\Omega},\widetilde{\mathcal{P}}^{\mathbb{G}}\big):=
    \big(\Omega\times\mathbb{R}_{+}\times \mathbb{R}^{n},\mathcal{P}^{\mathbb{G}}\otimes \mathcal{B}({\mathbb{R}^{n}})\big).
\end{align*} 
A measurable function 
$U:\big(\widetilde{\Omega},\widetilde{\mathcal{P}}^{\mathbb{G}}\big)\longrightarrow \left(\mathbb{R}^{d},\mathcal{B}({\mathbb{R}^{d}}) \right)$ 
is called \emph{$\widetilde{\mathcal{P}}^{\mathbb{G}}-$measurable function} and, abusing notation, the space of these functions will also be denoted by $\widetilde{\mathcal{P}}^{\mathbb{G}}$.
Let us also denote by $\widetilde{\mathcal{P}}^{\mathbb{G}}_+$ the space of non--negative $\widetilde{\mathcal{P}}^{\mathbb{G}}-$measurable functions.\footnote{We will adopt analogous notation for any non--negative measurable function; for example, for a $\sigma-$algebra $\mathcal{A}$, the set $\mathcal{A}_+$ denotes the set of non--negative $\mathcal{A}-$measurable functions.}

 
Let $\mu$ be a random measure on $\left(\mathbb{R}_{+}\times\mathbb{R}^{n},\mathcal{B}(\mathbb{R}_{+})\otimes\mathcal{B}(\mathbb{R}^{n})\right)$.
Consider a function $U \in \widetilde{\mathcal{P}}^{\mathbb{G}}$, then we define the process
\begin{align*}
U * \mu_\cdot(\omega) :=
			\begin{cases}
			\displaystyle \int_{(0,\cdot]\times\mathbb{R}^{n}} 
U\left(\omega,s,x\right) \mu(\omega; \ud s, \ud x),\textrm{ if } \int_{(0,\cdot]\times\mathbb{R}^{n}} 
|U\left(\omega,s,x\right)| \mu(\omega,\ud s, \ud x)<\infty,\\
		\displaystyle	\infty,\textrm{ otherwise}.
			\end{cases}    
\end{align*}

Let $X\in\mathcal{H}^{2}(\mathbb{G};\mathbb{R}^n)$, and denote by $\mu^{X}$ the $\mathbb{G}$--optional integer--valued random measure associated to the jumps of $X$.
In order to define the stochastic integral of a function $U \in \widetilde{\mathcal{P}}^{\mathbb{G}}$ with respect to the \emph{$\mathbb{G}-$compensated integer--valued random measure 
${\widetilde{\mu}}^{(\mathbb{G},X)}\!:=\mu^X-\nu^{(\mathbb{G},X)}$}, we will consider the following class
\begin{align*}
G_2(\mathbb{G},\mu^X):=\bigg\{
U:\big(\widetilde{\Omega},\widetilde{\mathcal{P}}^{\mathbb{G}}\big){}\longrightarrow
\big(\mathbb{R}^{d},\mathcal{B}(\mathbb{R}^{d})\big),\, 
	\mathbb{E}\bigg[\sum_{t>0} \left|U(t,\Delta X_t)\mathds{1}_{\{\Delta 
X_t\neq0 \}}-\widehat{U}_t^{(\mathbb{G},X)}\right|^2\bigg]
	<\infty
\bigg\}.
\end{align*}
Any element of $G_2(\mathbb{G},\mu^X)$ can be associated to an element of $\mathcal{H}^{2,d}(\mathbb{G};\mathbb{R}^d)$, uniquely up to $\mathbb{P}-$indistinguisha\-bi\-lity, via 
\begin{align*}
G_2(\mathbb{G},\mu^X)\ni U{}\longmapsto U\star{\widetilde{\mu}}^{(\mathbb{G},X)}\in \mathcal{H}^{2,d}(\mathbb{G};\mathbb{R}^d),
\end{align*}
see \cite[Chapter II, Definition 1.27 and Proposition 1.33.a]{jacod2013limit} and \citet[Theorem XI.11.21]{he2019semimartingale}.
We call $U\star{\widetilde{\mu}}^{(\mathbb{G},X)}$ the \emph{stochastic integral of $U$ with respect to ${\widetilde{\mu}}^{(\mathbb{G},X)}$}. 
Let us point out that for an arbitrary function of $G_2(\mathbb{G},\mu^X)$ the two processes $U*(\mu^X - \nu^{(\mathbb{G},X)})$ and $U\star \widetilde{\mu}^{(\mathbb{G},X)}$ are not equal.
We will make use of the following notation for the space of stochastic 
integrals with respect to ${\widetilde{\mu}}^{X}$ which are square integrable martingales
\begin{align*}
\mathcal{K}^2(\mathbb{G},\mu^X)
  := \left\{U\star{\widetilde{\mu}}^{(\mathbb{G},X)},\  U\in G_2(\mathbb{G},\mu^X)\right\}. 
\end{align*}
Moreover, by \cite[Chapter II, Theorem 1.33]{jacod2013limit} or \cite[Theorem 11.21]{he2019semimartingale}, we have that 
$\mathbb{E}\left[\langle U\star{\widetilde{\mu}}^{(\mathbb{G},X)}\rangle^{\mathbb{G}}_{\infty}\right]<\infty$ if and only if $U\in G_2(\mathbb{G},\mu^X)$, which enables us to define the following more convenient space
\begin{align*}
\mathbb{H}^2(\mathbb{G},X)
 :=\left\{ U:\big(\widetilde{\Omega},\widetilde{\mathcal{P}}^{\mathbb{G}}\big) 
  {}\longrightarrow \big(\mathbb{R}^{d},\mathcal{B}(\mathbb{R}^{d})\big),\ 
  \mathbb{E}\left[
  \text{Tr} 
  \Big[\langle U\star{\widetilde{\mu}}^{(\mathbb{G},X)}\rangle^{\mathbb{G}}_t\Big]\right]<\infty \right\},
\end{align*}
and we emphasize that we have the direct identification $\mathbb{H}^2(\mathbb{G},X) = G_2(\mathbb{G},\mu^X).$

\subsection{Dol\'{e}ans-Dade measure and disintegration} 

Assume that we are given a square--integrable $\mathbb{G}$--martingale $X\in\mathcal{H}^{2}(\mathbb{G};\mathbb{R}^n)$ along with its associated random measures $\mu^{X}$ and $\nu^{(\mathbb{G},X)}$.
In $\big(\widetilde{\Omega},\mathcal{G}_{\infty} \otimes \mathcal{B}(\mathbb{R}_+) \otimes \mathcal{B}(\mathbb{R}^n)\big)$ we can define the Dol\'{e}ans-Dade measures of $\mu^{X}$, resp. of $\nu^{(\mathbb{G},X)}$, as follows
\begin{align*}
M_{\mu^{X}}(A) &:= \mathbb{E}\left[\mathds{1}_A * \mu^{X}_{\infty}\right],
    \quad \text{resp. } \quad 
M_{\nu^{(\mathbb{G},X)}}(A) := \mathbb{E}\left[\mathds{1}_A * \nu^{(\mathbb{G},X)}_{\infty}\right],
\end{align*}
for every $A \in \mathcal{G}_{\infty} \otimes \mathcal{B}(\mathbb{R}_+) \otimes \mathcal{B}(\mathbb{R}^n)$.
Because $M_{\mu^{X}}$ is $\sigma$--integrable with respect to $\widetilde{\mathcal{P}}^{\mathbb{G}}$, we can define, for every non--negative $\mathcal{G}_{\infty} \otimes \mathcal{B}(\mathbb{R}_+) \otimes \mathcal{B}(\mathbb{R}^n)$--measurable function $W$, its conditional expectation with respect to $\widetilde{\mathcal{P}}^{\mathbb{G}}$ using $M_{\mu^{X}}$, which we denote by $M_{\mu^{X}}[W |\widetilde{\mathcal{P}}]$. 


Consider a pair of martingales $\overline{X} := (X^\circ,X^{\natural}) \in \mathcal{H}^2(\mathbb{G};\mathbb{R}^p) \times \mathcal{H}^{2}(\mathbb{G};\mathbb{R}^n)$, and $|I|$ the map where $|I|(x) := |x| + \mathds{1}_{\{0\}}(x)$, for $x\in\mathbb{R}^n$, then we define 
\begin{align}\label{def_C}
    C^{(\mathbb{G},\overline{X})} 
    := \text{Tr}\left[\langle X^{\circ} \rangle^{\mathbb{G}}\right] + |I|^2 * \nu^{(\mathbb{G},X^{\natural})}.
\end{align}
%
%
Let us now make some immediate, yet crucial, observations for the current work.
\begin{remark}\label{rem:crucial_remarks}
\begin{enumerate}
    \item \label{rem:C_properties}
The fact that $\overline{X}$ is a $\mathbb{G}-$martingale allows one to alternatively write \eqref{def_C} as follows
\begin{align*}
    C^{(\mathbb{G},\overline{X})} 
    = \textup{Tr}\big[\langle \overline{X}\rangle^{\mathbb{G}} \big]
    =       \textup{Tr}\big[\langle X^{\circ}\rangle^{\mathbb{G}} \big] 
        +   \textup{Tr}\big[\langle X^{\natural} \rangle^{\mathbb{G}} \big].
\end{align*}
In other words, we need to argue about the identity $\textup{Tr}\big[\langle X^{\natural} \rangle^{\mathbb{G}} \big] = |I|^2 * \nu^{(\mathbb{G},X^{\natural})}$.
Indeed, the martingale property yields 
\begin{align*}
    \widehat{I}^{(\mathbb{G},X^{\natural})}_t 
    =
    \mathbb{E}\Big[ \int_{\mathbb{R}^n} x \mu^{X^\natural} (\omega;\{t\}\times \textup{d}x) \Big| \mathcal{G}_{t-}\Big]
    =\mathbb{E} \big[ \Delta X^{\natural}_t \big| \mathcal{G}_{t-}\big]
    =0;
\end{align*}
for the last identity see \cite[Chapter I, Lemma 2.27]{jacod2013limit}.
Hence, from \cite[Chapter II, Lemma 1.33]{jacod2013limit} we conclude.
We stress that the above are true because the integer--valued measure associates to the jumps of the martingale $X^{\natural}$. 
%
%
    \item \label{rem:immersion_no_change_in_comp}
Assuming that $\mathbb{G}$ is immersed in $\mathbb{H}$, \emph{i.e.}, $\mathbb{H}$ is a filtration such that $\mathcal{G}_t\subseteq \mathcal{H}_t$ for every $t\in\mathbb{R}_+$ and, additionally, it possesses the property that every $\mathbb{G}-$martingale is an $\mathbb{H}-$martingale, then $C^{(\mathbb{G},\overline{X})}$ and $C^{(\mathbb{H},\overline{X})}$ are indistinguishable.
Indeed, one immediately has that 
\begin{align*}
    \langle X^{\circ}\rangle^{\mathbb{G}} - 
    \langle X^{\circ}\rangle^{\mathbb{H}}
    = (\langle X^{\circ}\rangle^{\mathbb{G}} - X^{\circ}\cdot\left(X^{\circ}\right)^{\top})
    +(X^{\circ}\cdot\left(X^{\circ}\right)^{\top} - \langle X^{\circ}\rangle^{\mathbb{H}})
\end{align*}
is an $\mathbb{H}-$martingale, $\mathbb{H}-$predictable and of finite variation.
Hence, since its initial value is $0$, it is the zero martingale, which proves that $\langle X^{\circ}\rangle^{\mathbb{G}}$ and $\langle X^{\circ}\rangle^{\mathbb{H}}$ are indistinguishable.
One can argue analogously for the processes $|I|^2*\nu^{(\mathbb{G},X^{\natural})}$ and $|I|^2*\nu^{(\mathbb{H},X^{\natural})}$.
In other words, we are allowed to interchange the filtration symbol in the notation of \eqref{def_C}, or even omit it.
More general results associated to the immersion of filtrations are presented in \cite[Appendix B.2]{PST_BackPropagChaos}, which we will make use of hereinafter.
\end{enumerate}
\end{remark}

Returning to \eqref{def_C}, one notices that we can disintegrate  
$\nu^{(\mathbb{G},X^{\natural})}$, \emph{i.e.}, 
we can determine kernels 
$K^{(\mathbb{G},\overline{X})} :\left(\Omega \times \mathbb{R}_+, \mathcal{P}^{\mathbb{G}}\right) {}\longrightarrow \mathcal{R}\left(\mathbb{R}^n,\mathcal{B}(\mathbb{R}^n)\right)$, where $\mathcal{R}\left(\mathbb{R}^n,\mathcal{B}(\mathbb{R}^n)\right)$ are the Radon measures on $\mathbb{R}^n$, such that 
\begin{align}\label{def:Kernels}
  \nu^{(\mathbb{G},X^{\natural})}(\omega,\ud t,\ud x) = K^{(\mathbb{G},\overline{X})}(\omega,t,\ud x)\ud C^{(\mathbb{G},\overline{X})}_t(\omega).
\end{align}
The kernels $K^{(\mathbb{G},\overline{X})}$ are $\mathbb{P} \otimes C^{(\mathbb{G},\overline{X})}-$unique, as can be deduced by a straightforward Dynkin class argument.

Moreover, let us define 
\begin{equation}\label{def_c}
    c^{(\mathbb{G},\overline{X})} := \bigg(\frac{\ud \langle X^{\circ} \rangle^{\mathbb{G}}}{\ud C^{(\mathbb{G},\overline{X})}}\bigg)^{\frac{1}{2}}.
\end{equation}
The reader may observe that $\frac{\ud \langle X^{\circ}\rangle^{\mathbb{G}}}{\ud C^{(\mathbb{G},\overline{X})}}$ is a $\mathbb{G}-$predictable process with values in the set of symmetric, non--negative definite $p\times p$ matrices. 
Using the diagonalization property of these matrices and results of \citet[Corollary 2]{azoff1974borel}, one can easily show that $c^{(\mathbb{G},\overline{X})}$ will also be a $\mathbb{G}-$predictable process with values in the set of symmetric, non--negative definite $p\times p$ matrices.


\subsection{Orthogonal decompositions}

Let us now state the decomposition results that will be used to solve the BSDEs of interest; for more details, we refer to \citet[Section 2.2.1]{papapantoleon2018existence}.

Let $\overline{X}:=(X^\circ,X^\natural)\in\mathcal{H}^2(\mathbb{G};\mathbb{R}^p)\times\mathcal{H}^{2,d}(\mathbb{G};\mathbb{R}^n)$ with $M_{\mu^{X^\natural}}[\Delta X^\circ|\widetilde{\mathcal{P}}^{\mathbb{G}}]=0$. 
Using this assumption, we get that for $Y^1 \in \mathcal{L}^2(X^\circ, \mathbb{G})$ and $Y^2 \in \mathcal{K}^2(\mu^{X^\natural}, \mathbb{G})$ it holds that $\langle Y^1,Y^2 \rangle = 0$; see \emph{e.g.} \citet[Theorem 13.3.16]{cohen2015stochastic}.
Then we define
\begin{align*}
\mathcal{H}^2(\overline{X}^{\perp_{\mathbb{G}}})  :=\big(\mathcal{L}^2(X^\circ, \mathbb{G})\oplus\mathcal{K}^2(\mu^{X^\natural}, \mathbb{G})\big)^{\perp_{\mathbb{G}}}.
\end{align*}  
Subsequently, we have the following description for $\mathcal{H}^2(\overline{X}^{\perp_{\mathbb{G}}})$, which is \cite[Proposition 2.6]{papapantoleon2018existence}.
Let $\overline{X}:=(X^\circ,X^\natural)\in\mathcal{H}^2(\mathbb{G};\mathbb{R}^p)\times\mathcal{H}^{2,d}(\mathbb{G};\mathbb{R}^n)$ with $M_{\mu^{X^\natural}}[\Delta X^\circ|\widetilde{\mathcal{P}}^{\mathbb{G}}]=0$.
Then,
\begin{align*}
\mathcal{H}^2(\overline{X}^{\perp_{\mathbb{G}}}) 
    = \big\{ L\in\mathcal{H}^2(\mathbb{G};\mathbb{R}^d),\ \langle X^{\circ}, L\rangle^{\mathbb{G}}=0 \text{ and } M_{\mu^{X^\natural}}[\Delta L|\widetilde{\mathcal{P}}^{\mathbb{G}}]=0\big\}.
\end{align*}
Moreover, the space $\big(\mathcal{H}^2(\overline{X}^{\perp_{\mathbb{G}}}), \Vert \cdot\Vert_{\mathcal{H}^2(\mathbb{R}^d)}\big)$ is closed.

Summing up the previous results, we arrive at the decomposition result that is going to dictate the structure of the BSDEs in our setting
\begin{align*}
\mathcal{H}^2(\mathbb{G};\mathbb{R}^p) 
= \mathcal{L}^2(X^\circ,\mathbb{G}) \oplus \mathcal{K}^2(\mu^{X^\natural},\mathbb{G})\oplus \mathcal{H}^2(\overline{X}^{\perp_{\mathbb{G}}}),
\end{align*}
where each of the spaces appearing in the identity above is closed.


%


\subsection{Norms and spaces}\label{Norms and spaces}

We will largely follow the notation of \cite[Section 2.3]{papapantoleon2018existence} with regard to the norms and spaces of stochastic processes.
However, we will need to additionally keep track of the filtration under which we are working, since later many filtrations will appear in our framework.

Let $\overline{X}:=(X^\circ,X^\natural)\in \mathcal{H}^2(\mathbb{G};\mathbb{R}^p)\times\mathcal{H}^{2,d}(\mathbb{G};\mathbb{R}^n)$ with 
$M_{\mu^{X^\natural}}[\Delta X^\circ|\widetilde{\mathcal{P}}^{\mathbb{G}}]=0$,
 and $A,C:(\Omega \times \mathbb{R}_{+},\mathcal{P}^{\mathbb{G}}) {}\longrightarrow (\mathbb{R}_{+},\mathcal{B}(\mathbb{R}_+))$ be c\`adl\`ag and increasing.
The following spaces will appear in the analysis throughout this article, for $\beta \geq 0$ and $T$ a $\mathbb{G}-$stopping time:
\begin{align*}
\mathbb{L}^{2}_{\beta}(\mathcal{G}_{T}, A;\mathbb{R}^d)
&:= \left\{\xi, \hspace{0.2cm} \mathbb{R}^{d} \text{--valued, } \mathcal{G}_{T}  \text{--measurable} , \|\xi\|^{2}_{\mathbb{L}^{2}_{\beta}(\mathcal{G}_{T}, A;\mathbb{R}^d)}: = \mathbb{E}\left[\mathcal{E}(\beta A)_{T-} |\xi|^2 \right] < \infty\right\},\\
\mathcal{H}^{2}_{\beta}(\mathbb{G},A;\mathbb{R}^d) 
&:= \left\{M \in \mathcal{H}^{2}(\mathbb{G},A;\mathbb{R}^d) , \|M\|^{2}_{\mathcal{H}^{2}_{\beta}(\mathbb{G},A;\mathbb{R}^d)}:= \mathbb{E} \left[\int_{0}^{T}\mathcal{E}(\beta A)_{t-} \, \ud\text{Tr}\big[\langle M \rangle^{\mathbb{G}}\big]_t \right] < \infty \right\},\\
\mathbb{H}^{2}_{\beta}(\mathbb{G},A,C;\mathbb{R}^d) 
&:= \begin{multlined}[t][0.75\textwidth]
 \bigg\{\phi, \hspace{0.2 cm} \mathbb{R}^{d} \text{--valued, } \mathbb{G}\text{--optional},\\  
 \|\phi\|^{2}_{\mathbb{H}^{2}_{\beta}(\mathbb{G},A,C;\mathbb{R}^d)}:= \mathbb{E} \left[\int_{0}^{T}\mathcal{E}(\beta A)_{t-}|\phi|_{t}^2 \, \ud C_t \right] < \infty  \bigg\},
\end{multlined}\\
\mathcal{S}^{2}_{\beta}(\mathbb{G},A;\mathbb{R}^d) 
&:=
\begin{multlined}[t][0.75\textwidth]
\bigg\{\phi, \hspace{0.2 cm} \mathbb{R}^{d} \text{--valued,} \hspace{0.1 cm} \mathbb{G}\text{--optional}, \\ \|\phi\|^{2}_{\mathcal{S}^{2}_{\beta}(\mathbb{G},A;\mathbb{R}^d)}:= 
\mathbb{E} \Big[\sup_{t \in [0,T]}\big\{\mathcal{E}(\beta A)_{t-}|\phi|_{t}^2\big\} \Big] < \infty  \bigg\},
\end{multlined}
\\
\mathbb{H}^{2}_{\beta}(\mathbb{G},A,X^\circ;\mathbb{R}^{d \times p}) 
&:=
\begin{multlined}[t][0.75\textwidth]
    \bigg\{Z \in \mathbb{H}^{2}(\mathbb{G},X^{\circ};\mathbb{R}^{d \times p}),  \\
    \|Z\|^{2}_{\mathbb{H}^{2}_{\beta}((\mathbb{G},A,X^\circ;\mathbb{R}^{d \times p}))}:= \mathbb{E} \left[\int_{0}^{T}\mathcal{E}(\beta A)_{t-} \, \ud\text{Tr}\big[\langle Z \cdot X^\circ \rangle^{\mathbb{G}}\big]_t \right] < \infty  \bigg\},
\end{multlined}
    \\
\mathbb{H}^{2}_{\beta}(\mathbb{G},A,X^\natural;\mathbb{R}^d)
&:= 
\begin{multlined}[t][0.75\textwidth]
    \bigg\{U \in \mathbb{H}^{2}(\mathbb{G},X^\natural;\mathbb{R}^d),  \\ \|U\|^{2}_{\mathbb{H}^{2}_{\beta}(\mathbb{G},A,X^\natural;\mathbb{R}^d)}:= \mathbb{E} \left[\int_{0}^{T}\mathcal{E}(\beta A)_{t-} \, \ud\text{Tr}\big[\langle U  \star \tilde{\mu}^{X^{\natural}} \rangle^{\mathbb{G}}\big]_t \right] < \infty  \bigg\},
\end{multlined}
\shortintertext{and}
\mathcal{H}^{2}_{\beta}(\mathbb{G},A,\overline{X}^{\perp_{\mathbb{G}}};\mathbb{R}^{d}) 
&:= \left\{M \in \mathcal{H}^{2}(\overline{X}^{\perp_{\mathbb{G}}}), \|M\|^{2}_{\mathcal{H}^{2}_{\beta}(\mathbb{G}, A,\overline{X}^{\perp_{\mathbb{G}}};\mathbb{R}^d)}:= \mathbb{E} \left[\int_{0}^{T}\mathcal{E}(\beta A)_{t-} \, \ud\text{Tr}\big[\langle M \rangle^{\mathbb{G}}\big]_t \right] < \infty \right\}.
\end{align*}

\begin{remark}
Let $A$ be a finite variation process, then we have defined and used above the process 
\begin{align}\label{stocheq}
    \mathcal{E}(A)_\cdot := \ue^{A_\cdot} \prod_{s \leq \cdot} \frac{1 + \Delta A_s}{\ue^{\Delta A_s}},
\end{align}
which is called the stochastic exponential of $A$. 
Using the trivial inequalities $ 0\leq 1 + x \leq \ue^x$, for all $x \geq -1$, and the usual properties of the jumps of finite variation processes, we can easily see that the above process is well defined, adapted, c\`adl\`ag and locally bounded. 
Moreover, it is known that, if $\Delta A \geq - 1$,  then 
    \begin{align}
        0 \leq \mathcal{E}(A)_\cdot \leq \ue^{A_\cdot}.
    \end{align}
\end{remark}

Now, let $\alpha$ be a $\mathbb{P}\otimes C$--a.e. non--negative, predictable with respect to $\mathbb{G}$, $\mathbb{R}-$valued process such that
$A_{\cdot} = \int_{0}^{\cdot}\alpha_s^2\ud C_s$. Then we define as 
\begin{align*}
    \mathscr{S}^{2}_{\hat{\beta}}(\mathbb{G},\alpha,C;\mathbb{R}^d) :=
    \left\{y\in \mathcal{S}^2_{\hat{\beta}}(\mathbb{G},A;\mathbb{R}^d) : \| \alpha y \|_{\mathbb{H}_{\hat{\beta}}(\mathbb{G},A,C;\mathbb{R}^d)}<\infty\right\},
\end{align*}
with corresponding norm defined by
\begin{align}\label{eq 4.5}
    \| \cdot \|^2_{\mathscr{S}^{2}_{\hat{\beta}}(\mathbb{G},\alpha,C;\mathbb{R}^d)}
    :=
    \| \cdot\|^2_{\mathcal{S}^2_{\hat{\beta}}(\mathbb{G},A;\mathbb{R}^d)}
    +
    \|\alpha \cdot\|^2_{\mathbb{H}_{\hat{\beta}}(\mathbb{G},A,C;\mathbb{R}^d)}.
\end{align}
Hence, for
\begin{align*}
    \left(Y,Z,U,M\right) \in  
\mathscr{S}^{2}_{\beta}(\mathbb{G},\alpha,C;\mathbb{R}^d) \times 
\mathbb{H}^{2}_{\beta}(\mathbb{G},A,X^\circ;\mathbb{R}^{d\times p}) \times \mathbb{H}^{2}_{\beta}(\mathbb{G},A,X^\natural;\mathbb{R}^d) \times  \mathcal{H}^{2}_{\beta}(\mathbb{G},A,\overline{X}^{\perp_{\mathbb{G}}};\mathbb{R}^d)
\end{align*}
we define
\begin{multline*}  
    \|\left(Y,Z,U,M\right) \|^{2}_{\star,\beta,\mathbb{G},\alpha,C,\overline{X}} \\
    :=  \|Y\|^{2}_{\mathscr{S}^{2}_{\beta}(\mathbb{G},\alpha,C;\mathbb{R}^d)} + \|Z\|^{2}_{\mathbb{H}^{2}_{\beta}(\mathbb{G},A,X^\circ;\mathbb{R}^{d\times p})} + 
    \|U\|^{2}_{\mathbb{H}^{2}_{\beta}(\mathbb{G},A,X^\natural;\mathbb{R}^d)} +  \|M\|^{2}_{\mathcal{H}^{2}_{\beta}(\mathbb{G},A,\overline{X}^{\perp_{\mathbb{G}}};\mathbb{R}^d)}.
\end{multline*}
Later on, we will need to rewrite the norms associated to the spaces $\mathbb{H}^{2}_{\beta}(\mathbb{G},A,X^\circ;\mathbb{R}^{d \times p})$ and $\mathbb{H}^{2}_{\beta}(\mathbb{G},A,X^\natural;\mathbb{R}^{d}) $ in terms of Lebesgue--Stieltjes integrals with respect to $C^{(\mathbb{G},\overline{X} )}$, for $C^{(\mathbb{G},\overline{X})}$ as defined in \eqref{def_C}; one may consult \cite[Lemma 2.13]{papapantoleon2018existence} for the details.
Then, for $(Z,U) \in 
\mathbb{H}^2_{\beta}(\mathbb{G}, A, X^\circ;\mathbb{R}^d) \times \mathbb{H}^2_{\beta}(\mathbb{G}, A, X^\natural;\mathbb{R}^d)$ and $C^{(\mathbb{G},\overline{X})}$ as defined in \eqref{def_C}, we have
 \begin{align}
     \|Z\|^2_{\mathbb{H}^2_{\beta}(\mathbb{G}, A, X^\circ;\mathbb{R}^d)} &= \mathbb{E}\left[\int_{0}^{T}\mathcal{E}({\beta} A)_{s-}\|Z_s c^{(\mathbb{G},\overline{X})}_s\|^2\ud C^{(\mathbb{G},\overline{X})}_s\right] \label{comp_norm_altern_circ}
     \shortintertext{and}
     \|U\|^2_{\mathbb{H}^2_{\beta}(\mathbb{G}, A, X^\natural;\mathbb{R}^d)} &= \mathbb{E}\left[\int_{0}^{T}\mathcal{E}({\beta} A)_{s-}\left(\tnorm{U_s(\cdot)}^{(\mathbb{G},\overline{X})}_s\right)^2\ud C^{(\mathbb{G},\overline{X})}_s\right],\label{comp_norm_altern_natural}
 \end{align}
 where 
\begin{align}
\begin{multlined}[0.9\textwidth]
\Big(\tnorm{U_t(\omega;\cdot)}^{(\mathbb{G},\overline{X})}_t(\omega)\Big)^2 
 :=
 \int_{\mathbb{R}^n}\big|U(\omega,s,x) - \widehat{U}_s^{(\mathbb{G},\overline{X})}(\omega)\big|^2 K^{(\mathbb{G},\overline{X})}_s(\omega,\ud x) \label{def:tnorm}\\
 + \big(1 - {\zeta}^{(\mathbb{G},X^{\natural})}_s(\omega)\big) \Delta C^{(\mathbb{G},\overline{X})}_s(\omega) \left|\int_{\mathbb{R}^n}U(\omega,s,x)\,K^{(\mathbb{G},\overline{X})}_s(\omega,\ud x)\right|^2,
\end{multlined}
\end{align}
with $K^{(\mathbb{G},\overline{X})}$ satisfying \eqref{def:Kernels}, \emph{i.e.},
 \begin{align*}
     \nu^{(\mathbb{G}, X^{\natural})}(\omega,\ud t,\ud x) 
     = K^{(\mathbb{G}, \overline{X})}(\omega,t,\ud x) \ud C^{(\mathbb{G},\overline{X})}_t(\omega).
 \end{align*} 
Finally, using the assumption $M_{\mu^{X^\natural}}[\Delta X^\circ | \widetilde{\mathcal{P}}^{\mathbb{G}}]=0$, in conjunction with \cite[Theorem 13.3.16]{cohen2015stochastic}, we have  
 \begin{align}\label{identity:norm_ZU}
     \|Z \cdot X^\circ + U \star \widetilde{\mu}^{X^\natural}\|^2_{\mathcal{H}^2_{\beta}(\mathbb{G},A;\mathbb{R}^d)} = \|Z\|^2_{\mathbb{H}^2_{\beta}(\mathbb{G},A,X^\circ;\mathbb{R}^{d \times p})} + \|U\|^2_{\mathbb{H}^2_{\beta}(\mathbb{G},A,X^\natural;\mathbb{R}^d)}.
 \end{align}

\begin{remark}\label{rem:notation_beta_zero}
In order to simplify the notation whenever possible, if we consider one of the aforementioned spaces for $\beta=0$, then we will omit $0$. 
As a result, the dependence on the process $A$ is redundant, hence we will also omit the process $A$ from the notation of the respective space.
As an example, $\mathbb{L}^2(\mathcal{G}_T;\mathbb{R}^d)$ denotes the space $\mathbb{L}^2_0(\mathcal{G}_T,A;\mathbb{R}^d)$, which is the classical Lebesgue space, and so forth.  
\end{remark}

In case we have to deal with a system of $N\in\mathbb{N}$ couples of martingales, we will need to introduce norms associated to the respective product space. 
To this end, let
$ \{\overline{X}^i\}_{i \in \mathscr{N}}$ be a family of couples of martingales, \emph{i.e.}, $\overline{X}^i:= (X^{i,\circ},X^{i,\natural}) \in \mathcal{H}^2(\mathbb{G};\mathbb{R}^p) \times \mathcal{H}^{2,d}(\mathbb{G}; \mathbb{R}^n)$, such that $M_{\mu^{X^{i,\natural}}}[\Delta X^{i,\circ}|\widetilde{\mathcal{P}}^{\mathbb{G}}]=0,$ for all $i \in \mathscr{N}$. 
Moreover, let $A^i:(\Omega \times \mathbb{R}_{+},\mathcal{P}^{\mathbb{G}}) {}\longrightarrow (\mathbb{R}_{+},\mathcal{B}(\mathbb{R}_+))$ be c\`adl\`ag and increasing, for every $i \in \mathscr{N}$. 
Let us denote 
\[
    (\textbf{Y}^N,\textbf{Z}^N,\textbf{U}^N,\textbf{M}^N) := ({Y}^{i,N}, {Z}^{i,N}, {U}^{i,N}, {M}^{i,N})_{i \in \mathscr{N}}.
\]
Then, for 
\begin{multline*}
(\textbf{Y}^N,\textbf{Z}^N,\textbf{U}^N,\textbf{M}^N) \in \\ \Prod_{i = 1}^{N}
\mathscr{S}^{2}_{\hat{\beta}}(\mathbb{G},\alpha,C^i;\mathbb{R}^d)
\times \mathbb{H}^{2}_{\hat{\beta}}(\mathbb{G},A^i,X^{i,\circ};\mathbb{R}^{d \times p}) 
\times \mathbb{H}^{2}_{\hat{\beta}}(\mathbb{G},A^i,X^{i,\natural};\mathbb{R}^d) 
\times \mathcal{H}^{2}_{\hat{\beta}}(\mathbb{G},A^i,{\overline{X}^i}^{\perp_{\mathbb{G}}};\mathbb{R}^{d})
\end{multline*}
we define
\begin{align*}
\big\|(\textbf{Y}^N,\textbf{Z}^N,\textbf{U}^N,\textbf{M}^N )\big\|^{2}_{\star,N,\beta,\mathbb{G},\alpha,\{C^i\}_{i \in \mathscr{N}}, \{\overline{X}^i\}_{i \in \mathscr{N}}}
    := \sum_{i = 1}^{N}\|(Y^i,Z^i,U^i,M^i)\|^{2}_{\star,\beta,\mathbb{G},\alpha,C^i,\overline{X}^i}.
\end{align*}

Finally, let us conclude this subsection with the following important definitions. 
Let $E := \mathbb{R}_+ \times \mathbb{R}^n$ and $(F,\|\cdot\|_2)$, $(G,\|\cdot\|_2)$ be two finite-dimensional Euclidean spaces.
\begin{itemize}
    \item Set $E_0 := (\mathbb{R}_+ \times \{0\}) \cup (\{0\} \times \mathbb{R}^n)$ and $\widetilde{E} := E \setminus E_0$.
\vspace{0.2cm}
    \item Let $f :(F,\|\cdot\|_2) \longrightarrow (G,\|\cdot\|_2)$. We will call the support of $f$ the set $\text{supp}(f) := \overline{\{f \neq 0\}}^{\|\cdot\|_2}$, where for a set $A$ we denote by $\overline{A}^{\|\cdot\|_2}$ its closure under the metric corresponding to the norm $\|\cdot\|_2$.
\vspace{0.2cm}
    \item $C_c(F;G) :=\{f : (F,\|\cdot\|_2) \longrightarrow (G,\|\cdot\|_2) : f \hspace{0.1cm}\text{continuous with compact support}\}$.
    \vspace{0.2cm}
    \item $C_{c|\widetilde{E}}(E;\mathbb{R}^d) := \{ f \in C_c(E;\mathbb{R}^d) : \text{supp}(f)\subseteq \widetilde{E}\}$.
    \vspace{0.2cm}
    \item $D^{\circ,d\times p} \subseteq C_c(\mathbb{R}_+;\mathbb{R}^{d\times p})$ is a fixed, countable set, dense in the space $\mathbb{L}^2(\mathbb{R}_+,\mathcal{B}(\mathbb{R}_+),\langle X^{\circ}\rangle (\omega))$ for $\mathbb{P}-a.s. \hspace{0.1cm} \omega \in \Omega$. 
    Existence results for such sets appear in \cite[Lemma A.14.]{papapantoleon2023stability}.
\vspace{0.2cm}
\item $D^{\natural} \subseteq C_{c|\widetilde{E}}(E ;\mathbb{R}^d)$ is a fixed, countable set, which is dense in the space $\mathbb{L}^2(E,\mathcal{B}(E),\nu^{(\mathbb{G},X^\natural)}(\omega))$ for $\mathbb{P}-a.s. \hspace{0.1cm} \omega \in \Omega$. 
Existence results for such sets appear in \cite[Lemma A.15.]{papapantoleon2023stability}. Also, keep in mind that $\nu^{(\mathbb{G},X^\natural)}(\omega)(E_0) = 0, \hspace{0.1cm} \mathbb{P}-a.s.$ and $\int_{E}\|x\|^2_2\hspace{0.1cm} \nu^{(\mathbb{G},X^\natural)}(\omega)(\ud s, \ud x) < \infty, \hspace{0.1cm} \mathbb{P}-a.s.$
\end{itemize}


\subsection{The \texorpdfstring{\(\Gamma\)}{Gamma} function}
\label{sec:gamma}

The integrand in the stochastic integral with respect to the purely-discontinuous martingale $\widetilde\mu^{X^\natural}$ is a suitable process $U$, and the driver of the BSDE $f$ will, obviously, depend on this process. 
However, in the present work, we cannot simply require that $f$ is Lipscitz in this argument with respect to the $\tnorm{\cdot}$-norm defined in \eqref{def:tnorm}, as done, for example, in \citet{papapantoleon2018existence,papapantoleon2023stability}, because this norm depends on the filtration $\mathbb{G}$.
Hence, motivated by applications and connections to PDEs, see the remark below, we define a composition of functions, called the $\Gamma$ function, with a free parameter $\Theta$.


\begin{definition}\label{def_Gamma_function}
Let $\overline{X}:=(X^\circ,X^\natural)\in \mathcal{H}^2(\mathbb{G};\mathbb{R}^p)\times\mathcal{H}^{2,d}(\mathbb{G};\mathbb{R}^n)$, 
let $C^{(\mathbb{G},\overline{X})}$ be as defined in \eqref{def_C} 
and $K^{(\mathbb{G},\overline{X})}$ that satisfies \eqref{def:Kernels}.
Additionally, let $\Theta $ be an $\mathbb{R}$--valued, $\widetilde{\mathcal{P}}^{\mathbb{G}}-$measurable function such that $|\Theta|\leq |I|$, for $|I|(x) := |x| + \mathds{1}_{\{0\}}(x)$.
Define the process
$\Gamma^{(\mathbb{G}, \overline{X},\Theta)} : \mathbb{H}^{2}(\mathbb{G},X^{\natural};\mathbb{R}^d) {}\longrightarrow \mathcal{P}^{\mathbb{G}}(\mathbb{R}^d)$ such that, for every $s\in \mathbb{R}_+$, holds
\begin{multline*}
\Gamma^{(\mathbb{G},\overline{X},\Theta)}(U)_s(\omega) :=
\int_{\mathbb{R}^n}\left(U(\omega,s,x) - \widehat{U}_s^{(\mathbb{G},X^{\natural})}(\omega)\right)\left(\Theta(\omega,s,x) - \widehat{\Theta}_s^{(\mathbb{G},X^{\natural})}(\omega)\right)\,K^{(\mathbb{G},\overline{X})}_s(\omega,\ud x)\\
+ (1 - \zeta^{(\mathbb{G},X^{\natural})}_s(\omega))\Delta C^{(\mathbb{G},\overline{X})}_s(\omega) \int_{\mathbb{R}^n}U(\omega,s,x)\,K^{(\mathbb{G},\overline{X})}_s(\omega,\ud x) \int_{\mathbb{R}^n}\Theta(\omega,s,x)\,K^{(\mathbb{G},\overline{X})}_s(\omega,\ud x).
\end{multline*}
\end{definition}

\begin{remark}
\begin{enumerate}
    \item 
    Given the square--integrability of the martingale $X^\natural$, it is immediate that the process $\Theta$ as defined above lies in $\mathbb{H}^{2}(\mathbb{G},X^{\natural};\mathbb{R})$. 
    Therefore, for any $U\in\mathbb{H}^2(\mathbb{G},X^{\natural};\mathbb{R}^d)$, the process $\Gamma$ is well-defined $\mathbb{P}\otimes C^{(\mathbb{G},\overline{X})}-$a.e.
    \item 
    One would expect in $\Gamma$ a notational dependence on the kernel. 
    However, for $K_1^{(\mathbb{G},\overline{X})},K_2^{(\mathbb{G},\overline{X})}$ satisfying \eqref{def:Kernels} and $U \in \mathbb{H}^{2}(\mathbb{G},X^{\natural};\mathbb{R}^d)$ we have $\Gamma^{(\mathbb{G},\overline{X},\Theta)}(U)_s = \Gamma^{(\mathbb{G},\overline{X},\Theta)}(U)_s, \hspace{0.2cm} \mathbb{P} \otimes C^{(\mathbb{G},\overline{X})}-$a.e; here, implicitly, the left-hand side is defined with respect to $K_1$ and the right-hand side with respect to $K_2$.
    In that way, we have the uniqueness of the kernels that satisfy \eqref{def:Kernels}.
    Since in the respective computations all the appearing equalities will be taken under $\mathbb{P} \otimes C^{(\mathbb{G},\overline{X})}$, we have suppressed the notational dependence on the kernels.
    \item 
    The choice of $\Gamma$ was based on, and inspired by, applications.
    The reader may recall, for example, the connection between BSDEs and partial integro-differential equations, and the special structure that is required for the generator, see \emph{e.g.}  \citet{barles1997backward} or \citet[Section 4.2]{delong2013backward}. 
    Moreover, one can easily verify that $\Gamma$ is equal to
    \begin{align*}
        \frac{\ud \langle U \star \widetilde{\mu}^{(\mathbb{G},X^\natural)},\Theta \star \widetilde{\mu}^{(\mathbb{G},X^\natural)} \rangle^{\mathbb{G}}}{\ud C^{(\mathbb{G},\overline{X})}}.
    \end{align*}
\end{enumerate}
\end{remark}

The next lemma justifies the definition of the process $\Gamma$.
In fact, let us point out that, if one can prove that results analogous to \cref{lem:Gamma_is_Lipschitz} and \cite[Lemma B.6]{PST_BackPropagChaos} hold for a process $\Gamma$, then the results of the present article remain valid, with the appropriate modifications.

\begin{lemma}\label{lem:Gamma_is_Lipschitz}
Let $\overline{X} \in \mathcal{H}^2(\mathbb{G};\mathbb{R}^p)\times\mathcal{H}^{2,d}(\mathbb{G};\mathbb{R}^n)$ and 
$\Theta \in \widetilde{\mathcal{P}}^{\mathbb{G}}$ be an $\mathbb{R}$--valued function such that $|\Theta|\leq |I|$. 
Then, for every $U^1,U^2 \in \mathbb{H}^{2}(\mathbb{G},X^{\natural};\mathbb{R}^d)$, we have 
\begin{align*}
\big|\Gamma^{(\mathbb{G},\overline{X},\Theta)}(U^1)_t(\omega) - \Gamma^{(\mathbb{G},\overline{X},\Theta)}(U^2)_t(\omega)\big|^2
\leq 2 \left(\tnorm{U^1_t(\omega ; \cdot) - U^2_t(\omega; \cdot)}^{(\mathbb{G},\overline{X})}_t(\omega)\right)^2,\hspace{0.2cm}\mathbb{P} \otimes C^{(\mathbb{G},\overline{X})}-\text{a.e.}
\end{align*}
\end{lemma}


\subsection{Wasserstein distance}\label{sec:Wasserstein_dist}

Let $\mathbb{X}$ be a Polish space endowed with a metric $\rho$, then we denote by $\mathscr{P}(\mathbb{X})$ the space of probability measures on $(\mathbb{X},\rho)$. 
Moreover, for every real $q \in [1,\infty)$, we define the probability measures on $\mathbb{X}$ with finite $q$ moment to be 
\begin{align*}
    \mathscr{P}_q(\mathbb{X}) :=\left\{\mu : \int_{\mathbb{X}}\rho(x_0,x)^q \,\mu(\ud x) < \infty \right\},\text{ for some } x_0 \in \mathbb{X}.
\end{align*}
By the triangle inequality and the fact that we consider probability (\textit{i.e.} finite) measures, it is immediate that the space $\mathscr{P}_q(\mathbb{X})$ is independent of the choice of $x_0$. 
On $\mathscr{P}_q(\mathbb{X})$ we can define a mode of convergence, which we will simply call weak convergence in $\mathscr{P}_q(\mathbb{X})$, that says
\begin{gather*}
    \text{a sequence of probability measures} \hspace{0.2cm} \{\mu_m\}_{m \in \mathbb{N}}\text{ converges weakly in }\mathscr{P}_q(\mathbb{X})\text{ to a probability measure } \mu\\
\text{if and only if for every continuous function } f \text{ such that } |f(x)| \leq C \hspace{0.1cm}(1 + \rho(x_0,x)^q), \text{ where }C := C(f) \in \mathbb{R}_+, \\
\numberthis\label{fW}
\text{we have } \int_{\mathbb{X}}f(x)\,\mu_m(\ud x) \xrightarrow[m\to\infty]{} \int_{\mathbb{X}}f(x)\,\mu(\ud x).
\end{gather*}
Of course, as before, it is immediate that in the above definition it does not matter which $x_0$ we choose. 
Now, the topology induced by this stronger mode of convergence is metrizable from a metric with nice properties; this metric is called the Wasserstein distance of order $q$.
More precisely, given two probability measures $\mu, \nu \in \mathscr{P}_q(\mathbb{X})$ we define the Wasserstein distance of order $q$ between them to be 
\begin{equation}\label{def:wasserstein}
    \mathcal{W}_{q,\rho}^q(\mu,\nu) := \inf_{\pi \in \Pi(\mu,\nu)}\left\{\int_{\mathbb{X} \times \mathbb{X}}\rho(x,y)^q\,\pi(\ud x,\ud y)\right\},
\end{equation}
where $\Pi(\mu,\nu)$ are the probability measures on $\mathbb{X} \times \mathbb{X}$ with marginals $\pi_1 = \mu$ and $\pi_2 = \nu$.
The interested reader may consult \citet[Theorem 6.9]{villani2009optimal} for the fact that $\mathcal{W}_{q,\rho}$ metrizes $\mathscr{P}_q(\mathbb{X})$.
Moreover, $\mathscr{P}_q(\mathbb{X})$ with this mode of convergence is a Polish space, as can be seen from \cite[Theorem 6.18]{villani2009optimal}.

Given an $N \in \mathbb{N}$ and $\textbf{x}^N:=(x_1,\dots,x_N),\textbf{y}^N:=(y_1,\dots,y_N) \in \mathbb{X}^{N}$, we have the empirical measures 
\begin{align*}
    L^N(\textbf{x}^N) := \frac{1}{N}\sum_{i = 1}^N \delta_{x_i} \quad\text{ and }\quad L^N(\textbf{y}^N) := \frac{1}{N}\sum_{i = 1}^N \delta_{y_i},
\end{align*}
where $\delta_{\cdot}$ is the Dirac measure on $\mathbb{X}$. Then we have 
\begin{align}\label{empiricalineq}
    \mathcal{W}_{q,\rho}^q\big(L^N(\textbf{x}^N),L^N(\textbf{y}^N)\big) \leq \frac{1}{N} \sum_{i = 1}^{N}\rho(x_i,y_i)^q.
\end{align}
The above inequality is immediate if in the definition of the Wasserstein distance \eqref{def:wasserstein} we choose the probability measure $$\pi := \frac{1}{N}\sum_{i =1}^{N}\delta_{(x_i,y_i)},$$
where $\delta_{(\cdot,\cdot)}$ is the Dirac measure over $\mathbb{X} \times \mathbb{X}$. 
One can immediately see that $\pi_1 = L^N(\textbf{x}^N)$ and $\pi_2 = L^N(\textbf{y}^N)$.


\subsection{Skorokhod space}

Let $(m,q) \in \mathbb{N}\times \mathbb{N}$, and denote by $\mathbb{D}^{m \times q} := \{f:[0,\infty) {}\longrightarrow \mathbb{R}^{m \times q}: f\text{ c\`adl\`ag}\}$ the space of functions which are right-continuous on $[0,\infty)$ and admit a left limit for every $t\in(0,\infty)$.
We will simply write $\mathbb{D}^{m}$ instead of $\mathbb{D}^{m \times 1}$ or $\mathbb{D}^{1 \times m}$.
We supply $\mathbb{D}^{m\times q}$ with its usual $\textup{J}_1$--metric, which we denote by $\rho_{\textup{J}_1^{m \times q}}$. 
Endowed with this metric, $\mathbb{D}^{m \times q}$ becomes a Polish space. 
We are not going to get into the specifics of $\rho_{\textup{J}_1^{m \times q}}$, as we will only need a couple of its basic properties. 
Firstly, for every $x,y \in \mathbb{D}^{m\times q}$, we have 
\begin{align}\label{skorokineq}
\rho_{\textup{J}_1^{m \times q}}(x,y) \leq \sup_{s \in [0,\infty)}\{\|x_s - y_s\|_2\} \wedge 1.
\end{align}
Secondly, the Borel $\sigma-$algebra that $\rho_{\textup{J}_1^{m \times q}}$ generates coincides with the usual product $\sigma-$algebra on ${(\mathbb{R}^{m \times q})}^{[0,\infty)}$ that the projections generate. 
To be more precise, we have 
\begin{align}\label{skoalgebra}
    \mathcal{B}_{\rho_{\textup{J}_1^{m\times q}}}(\mathbb{D}^{m\times q}) = \sigma\Big(\textrm{Proj}_s^{-1}(A): A \in \mathcal{B}(\mathbb{R}^{m\times q}), s \in [0,\infty)\Big) \bigcap \mathbb{D}^{m\times q},
\end{align}
where ${(\mathbb{R}^{m\times q})}^{[0,\infty)} \ni x \overset{\textrm{Proj}_s}{\longmapsto} x(s)\in \mathbb{R}^{m\times q}$, for every $s \in [0,\infty)$. 
Additional results on the Skorokhod space are available from \cite[Chapter 15]{he2019semimartingale} or \cite[Chapter VI]{jacod2013limit}. 

Using that the $\sigma-$algebra on $\Omega$ is $\mathcal{G}$, it is obvious from \eqref{skoalgebra} that every $\mathcal{G} \otimes \mathcal{B}([0,\infty))-$jointly measurable c\`adl\`ag process $X$ can be seen as a function with domain $\Omega$ and taking values in $\mathbb{D}^{m\times q}$ such that $X$ is $\big(\mathcal{G}/\mathcal{B}_{\rho_{\textup{J}_1^d}}(\mathbb{D}^{m\times q})\big)-$measurable and \emph{vice versa}: every $\big(\mathcal{G}/\mathcal{B}_{\rho_{\textup{J}_1^d}}(\mathbb{D}^{m\times q})\big)-$measurable random variable $X$ can be seen as a $\mathcal{G} \otimes \mathcal{B}([0,\infty))-$jointly measurable c\`adl\`ag process.

\begin{remark}\label{rem:Sko}
Later on, when we say that a collection of c\`adl\`ag processes is independent or is identically distributed or is exchangeable, they will be understood as $\Big(\mathcal{G}/\mathcal{B}_{\rho_{\textup{J}_1^{m\times q}}}(\mathbb{D}^{m\times q})\Big)-$measurable random variables.
\end{remark}


\subsection{Weak convergence of filtrations}

Finally, let us collect here some useful results about the weak convergence of filtrations.

\begin{definition}
 Let $(m,q) \in \mathbb{N}\times \mathbb{N}$ and assume that $\{a^k\}_{k \in \overline{\mathbb{N}}}$ is a sequence of $\Big(\mathcal{G}/\mathcal{B}_{\rho_{\textup{J}_1^{m \times q}}}(\mathbb{D}^{m\times q})\Big)-$measurable random variables. 
 \begin{enumerate}
     \item We say that the sequence $\{a^k\}_{k \in\mathbb{N}}$ converges in probability under the $\textup{J}_1^{m\times q}-$metric to $a^{\infty}$, denoted by $a^k \xrightarrow[k\rightarrow \infty]{\left(\textup{J}_1(\mathbb{R}^{m\times q}),\mathbb{P}\right)} a^{\infty}$, if and only if, for every $\varepsilon >0$ we have
     \begin{align*}
         \mathbb{P}\left(\rho_{\textup{J}_1^{m\times q}}(a^k,a^{\infty}) > \varepsilon \right) \xrightarrow[k \rightarrow \infty]{|\cdot|} 0.
     \end{align*}
     
     \item For every $\vartheta \in [1,\infty)$, we say that the sequence $\{a^k\}_{k \in\mathbb{N}}$ converges in $\mathbb{L}^{\vartheta}-$mean under the $\textup{J}_1^{m\times q}-$metric to $a^{\infty}$, denoted by $a^k \xrightarrow[k\rightarrow \infty]{\left(\textup{J}_1(\mathbb{R}^{m\times q}),\mathbb{L}^{\vartheta}\right)} a^{\infty}$, if and only if, we have
     \begin{align*}
         \mathbb{E}\left[\left(\rho_{\textup{J}_1^{m\times q}}(a^k,a^{\infty})\right)^{\vartheta} \right] \xrightarrow[k \rightarrow \infty]{|\cdot|} 0.
     \end{align*}
 \end{enumerate}
\end{definition}

Now, let $\{\mathbb{G}^k:=(\mathcal{G}^k_t)_{t\in \mathbb{R}_+}\}_{k \in \mathbb{N}}$ and $\mathbb{G}^{\infty} := (\mathcal{G}^{\infty}_t)_{t \in \mathbb{R}_+}$ be subfiltrations of $\mathbb{G}$ that satisfy the usual conditions. 
In addition, define $\mathcal{G}^{\infty}_{\infty} := \bigvee_{t \in \mathbb{R}_+}\mathcal{G}^{\infty}_t$.

\begin{definition}
The sequence $\{\mathbb{G}^k\}_{k \in \mathbb{N}}$ converges weakly to $\mathbb{G}^{\infty}$, denoted by $\mathbb{G}^k\xrightarrow[k \rightarrow \infty]{\textup{w}} \mathbb{G}^{\infty}$, if and only if, for every set $S \in \mathcal{G}^{\infty}_{\infty}$, we have
\begin{align*}
    \mathbb{E}\left[\mathds{1}_S\Big|\mathcal{G}^k_{\cdot}\right] \xrightarrow[k \rightarrow \infty]{\left(\textup{J}_1(R),\mathbb{P}\right)} \mathbb{E}\left[\mathds{1}_S\Big|\mathcal{G}^{\infty}_{\cdot}\right].
\end{align*}
\end{definition}


\section{Setting}
\label{sec:Setting}

In this section, we introduce the framework under which we are going to work. 
In order to ease the abundant relevant notation, we will mainly use only the respective indices, \emph{i.e.}, we will abstain from explicitly denoting the datum (or data) upon which a quantity depends.
To this end, we will reserve the letters $k,i$ and $N$ to denote natural numbers, where $k$ may also take the value infinity, and assign to them a specific role, which is described by the forthcoming convention:
the letter $k$ denotes the set of standard data,
the letter $i$ denotes the label of the player, and 
the letter $N$ denotes the total number of players participating in the mean-field system.
Moreover, we will respect the order $k,i,N$.
Additionally, the letter $j$ will be reserved to denote a natural number and serve as a substitution of $i$ whenever two indices are required to be used for two distinct player labels.
Hence, after these explanatory comments, whenever three indices appear, the first one is associated to the standard data, the second one denotes the label of the player, and the third one denotes the total number of players. 
In case less than three of the aforementioned indices appear in the notation, given the reserved role for the indices $k,i$ and $N$, the interpretation will be still clear.

Let us fix for the remainder of the article a probability space $(\Omega,\mathcal{G},\mathbb{P})$ and a sequence of standard data $\{\mathscr{D}^k\}_{k\in\overline{\mathbb{N}}}$, where, 
for $k \in \overline{\mathbb{N}}$,
\begin{align}\label{def:seq_standard_data}
\mathscr{D}^k := 
    \big(\{\overline{X}^{k,i}\}_{i \in \mathbb{N}}, \{\mathbb{F}^{k,i}\}_{i\in\mathbb{N}},T^k,
    \{\xi^{k,i,N}\}_{i,N \in \mathbb{N}},
    \{\xi^{k,i}\}_{i \in \mathbb{N}},\Theta^k,\Gamma,f^k \big),
\end{align} 
where $\mathbb{F}^{k,i}:=(\mathcal{F}^{k,i}_t)_{t\geq 0}$ denotes the usual augmentation of the natural filtration of $\overline{X}^{k,i}$, for every $i\in\mathbb{N}$.
In the sequel, the following assumptions are satisfied by the data $\mathscr{D}^k$, for every $k \in \overline{\mathbb{N}}$, under a universal $\hat{\beta} > 0$:
\begin{enumerate} [label=\textup{\textbf{(B\arabic*)}}]
    \item\label{H1} The sequence 
    $\{\overline{X}^{k,i}\}_{i \in \mathbb{N}}$ consists of independent and identically distributed processes such that  
    \begin{align*}
    \overline{X}^{k,i}=(X^{k,i,\circ},X^{k,i,\natural})\in \mathcal{H}^2(\mathbb{F}^{k,i};\mathbb{R}^p) \times \mathcal{H}^{2,d}(\mathbb{F}^{k,i};\mathbb{R}^n)
    \text{ with }M_{\mu^{X^{k,i,\natural}}}[\Delta X^{k,i,\circ}|\widetilde{\mathcal{P}}^{\mathbb{F}^{k,i}}]
    = 0,
    \end{align*}
     for every $i\in \mathbb{N}$.
    Additionally, for every $i\in\mathbb{N}$, the process $\overline{X}^{k,i}$ has independent increments.\footnote{\label{footnote:Drop_index_i}The filtration $\mathbb{F}^{k,i}$ is associated to $\overline{X}^{k,i}$, for every $k,i\in\mathbb{N}$. 
    Hence, we will make use of the function $C^{k,i}{\overset{\triangle}{=}}C^{(\mathbb{F}^{k,i},\overline{X}^{k,i})}$, 
    the kernel $K^{k,i}{\overset{\triangle}{=}}K^{(\mathbb{F}^{k,i},\overline{X}^{k,i})}$ and 
    the function  
    $c^{k,i}{\overset{\triangle}{=}}c^{(\mathbb{F}^{k,i},\overline{X}^{k,i})}$, as defined in \eqref{def_C}, \eqref{def:Kernels} and \eqref{def_c}, respectively. 
    However, under the conditions imposed in \ref{H1}, 
    the aforementioned (deterministic) objects do not depend on $i\in\mathbb{N}$, \emph{i.e.}, they are common for all players associated to the data $\mathscr{D}^k$.
    Using this information, we will drop the index $i$ and we will denote the dependence only on the index associated to the standard data rendering thus the introduced notation of these objects in its simpler form $C^k$, $K^k$ and $c^k$.
    More detailed arguments about the aforementioned comments are provided in \cref{Remark 3.1}.\ref{Remark 3.1 ii}.} 

    \item\label{H2} The time horizon $T^k$ is deterministic. 
    The sequences of terminal conditions $\{\xi^{k,i}\}_{i \in \mathbb{N}}$ and $\{\xi^{k,i,N}\}_{i,N \in \mathbb{N}}$, for the $A^k$ defined in \ref{H5}, are such that the former consists of independent and identically distributed random variables and both of them consist of random variables which satisfy the following measurability and integrability conditions under $\hat{\beta}$: for every $i,N\in\mathbb{N}$  
    \begin{align*}
        \xi^{k,i}\in\mathbb{L}^2_{\hat{\beta}}(\mathcal{F}^{k,i}_{T^k}, A^{k};\mathbb{R}^d)
        \hspace{0.2em} \text{ and } \hspace{0.2em}
        \xi^{k,i,N}\in 
        \mathbb{L}^2_{\hat{\beta}}\big(\bigvee_{i = 1}^{N} \mathcal{F}^{k,i}_{T^k}, A^{k};\mathbb{R}^d\big),
        \text{ for every }  i\in\mathscr{N},
    \end{align*}
    \begin{align*}
        \big\| \xi^{k,i,N} - \xi^{k,i}\big\|_{\mathbb{L}^2_{\hat{\beta}}(\bigvee_{j = 1}^{N} \mathcal{F}^{k,j}_{T^k}, A^{k};\mathbb{R}^d)\hspace{0.2em}} 
        \xrightarrow[N \rightarrow \infty]{\hspace{0.2em}} 0 
    \end{align*}
    and
    \begin{align*}
        \frac{1}{N} \sum_{i = 1}^N 
        \big\|\xi^{k,i,N} - \xi^{k,i}\big\|^2_{\mathbb{L}^2_{\hat{\beta}}(\bigvee_{j = 1}^{N} \mathcal{F}^{k,j}_{T^k}, A^{k};\mathbb{R}^d)} \xrightarrow[N \rightarrow \infty]{} 0.
    \end{align*}
      
    \item\label{H3} The function $\Theta^k$ is deterministic, and the function $\Gamma^{k,i}\overset{\triangle}{=} \Gamma^{(\mathbb{F}^{k,i},\overline{X}^{k,i},\Theta^k)}$ is defined according to \cref{def_Gamma_function}, for every $i\in\mathbb{N}$.
    
    
    \item \label{H4} The generator $f^k: \mathbb{R}_+ \times \mathbb{R}^d \times \mathbb{R}^{d \times p} \times \mathbb{R}^d \times \mathscr{P}_2(\mathbb{R}^d) {}\longrightarrow \mathbb{R}^d$ is such that for any $(y,z,u,\mu) \in \mathbb{R}^d \times \mathbb{R}^{d \times p} \times \mathbb{R}^d \times \mathscr{P}_2(\mathbb{R}^d)$, the map 
    \begin{align*}
        t \longmapsto f^k(t,y,z,u,\mu) \hspace{0.2cm}\text{is}\hspace{0.2cm}\mathcal{B}(\mathbb{R}_+)\text{--measurable}
    \end{align*}
    and satisfies the following Lipschitz condition 
    \begin{align*}
    \begin{multlined}[0.9\textwidth]
    \big|f^k(t,y,z,u,\mu) - f^k(t,y',z',u',\mu')\big|^2\\
    \leq \hspace{0.1cm} r^k(t) \hspace{0.1cm} |y - y'|^2+ \hspace{0.1cm} \vartheta^{k,\circ}(t) \hspace{0.1cm} |z - z'|^2 
     + \hspace{0.1cm} \vartheta^{k,\natural}(t) \hspace{0.1cm} |u - u'|^2 + \vartheta^{k,*}(t) \hspace{0.1cm} W^2_{2,|\cdot|}\left(\mu,\mu'\right),
    \end{multlined}
    \end{align*}
     where  $
         (r^k,\vartheta^{k,\circ},\vartheta^{k,\natural},\vartheta^{k,*}): \left(\mathbb{R}_+, \mathcal{B}(\mathbb{R}_+)\right) \longrightarrow \left(\mathbb{R}^4_+,\mathcal{B}\left(\mathbb{R}^4_+\right)\right).
     $
    
    \item\label{H5}  Define $(\alpha^k)^2 := \max\{\sqrt{r^k},\vartheta^{k,\circ}, \vartheta^{k,\natural}, \sqrt{\vartheta^{k,*}}\}$. 
    Then, for the  c\`adl\`ag function
    \begin{align}\label{def:A_propag}
      A^{k}_{\cdot}
      := \int_{0}^{\cdot}(\alpha^k_s)^2\ud C^{k}_s
    \end{align} 
    there exists $\Phi^k \geq 0$ such that
    $
    \Delta A^{k}_t \leq \Phi^k,$ for every $t\in(0,T^k]$.
    
    \item \label{H6} The following holds 
       \begin{align}\label{equation 5.2}
       \mathbb{E}\bigg[\int_{0}^{T^k}\mathcal{E}\big(\hat{\beta} A^{k}\big)_{s-} \frac{|f^k(s,0,0,0,\delta_0)|^2}{(\alpha^k_s)^2}\ud C^{k}_s\bigg] < \infty,
       \end{align}
       where $\delta_0$ is the Dirac measure on the domain of the last argument concentrated at $0$, the neutral element of the addition. 

       
   \item \label{H:prop_contraction} 
   The following constant satisfies $3\widetilde{M}^{\Phi^k}(\hat{\beta}) < 1$, where 
   \begin{align*}
   \widetilde{M}^\Phi(\beta)
    &= \min_{\gamma \in (0,\beta)} \left\{ \frac{9}{\beta} +  8 \frac{(1 + \gamma\Phi)}{ \gamma} + \frac{2 + 9\beta}{\beta - \gamma}\hspace{0.1cm} \frac{(1 + \gamma \Phi)^2}{\gamma}\right\}
    \\
    &= \frac{ 2 \sqrt{\frac{2}{\beta} + 9}\sqrt{\frac{2}{\beta} + 17}  + \frac{4}{\beta} + 35}{\beta} + \left( 2 \sqrt{\frac{2}{\beta} + 9}\sqrt{\frac{2}{\beta} + 17}  + \frac{4}{\beta} + 26\right) \Phi.
   \end{align*}
\end{enumerate}
Using Assumptions \ref{H1}--\ref{H:prop_contraction}, the mean-field BSDE \eqref{mfBSDE}, for every $N\in\mathbb{N}$, as well as the McKean--Vlasov BSDE \eqref{MVBSDE}, for every $i\in\mathbb{N}$, admit a unique solution; for the former see
\cite[Theorem 4.13]{PST_BackPropagChaos}, and for the latter \cite[Theorem 4.8]{PST_BackPropagChaos}.
We remind the reader that the unique solution of the mean-field BSDE \eqref{mfBSDE} for $N-$players associated to the standard data $\mathscr{D}^k$ is denoted by $\textbf{S}^{k,N} : = (\textbf{Y}^{k,N},\textbf{Z}^{k,N},\textbf{U}^{k,N},\textbf{M}^{k,N})$, where $\textbf{Y}^{k,N} := (Y^{k,1,N},\dots,Y^{k,i,N},\dots,Y^{k,N,N})$.
Focusing only on the $i-$th player of the mean-field BSDE associated to the standard data $\mathscr{D}^k$, the solution will be denoted by $\mathscr{S}^{k,i,N}$, \emph{i.e.},
$\mathscr{S}^{k,i,N}:=(Y^{k,i,N},Z^{k,i,N},U^{k,i,N},M^{k,i,N})$ for $i\in\mathscr{N}$.
The unique solution of the McKean--Vlasov BSDE \eqref{MVBSDE} associated to the standard data $\mathscr{D}^k$ for the $i-$th player is indicated by $\mathscr{S}^{k,i}$, \emph{i.e.}, $\mathscr{S}^{k,i}:=(Y^{k,i},Z^{k,i},U^{k,i},M^{k,i})$.
Recalling the notational convention we adopted at the beginning of this section, there should be no confusion between $\mathscr{S}^{k,i,N}$ and $\mathscr{S}^{k,i}$.
\begin{remark}\label{Remark 3.1}
 \begin{enumerate}
     \item\label{Remark 3.1 i} In \ref{H2}, we used the $\sigma-$algebra $\bigvee_{i=1}^N\mathcal{F}_{T^k}^{k,i}$.
     Let us also define the filtration 
     \begin{align*}
         \mathbb{F}^{k,(1,\dots,N)} : = \bigvee_{i = 1}^{N} \mathbb{F}^{k,i}, \text{ \emph{i.e.}, } 
         \mathcal{F}^{k,(1,\dots,N)}_t: = \bigvee_{i = 1}^{N} \mathcal{F}^{k,i}_t
         \text{ for every }N\in\mathbb{N}, i\in\mathscr{N} \text{ and }t\in\mathbb{R}_+.
     \end{align*} 
     which satisfies the usual conditions, see  \citet[Theorem 1]{Wu1982}.
    For $i \in \mathscr{N}$ and $N\in\mathbb{N}$, 
    a direct consequence of the independence of filtrations is that every $\mathbb{F}^{k,i}-$martingale remains martingale under $\mathbb{F}^{k,(1,\dots,N)}$, \emph{i.e.}, the filtration $\mathbb{F}^{k,i}$ is immersed in the filtration $\mathbb{F}^{k,(1,\dots,N)}$.
    In particular, $X^{k,i,\natural} \in \mathcal{H}^{2,d}(\mathbb{F}^{k,(1,\dots,N)}; \mathbb{R}^n)$; see \cite[Corollary B.7]{PST_BackPropagChaos}.
     Additionally, from the assumption $M_{\mu^{{k,i,\natural}}}[\Delta X^{k,i,\circ}|\widetilde{\mathcal{P}}^{k,i}]
    = 0$ one can deduce that it is also true that $M_{\mu^{{k,i,\natural}}}[\Delta X^{k,i,\circ}|\widetilde{\mathcal{P}}^{\mathbb{F}^{k,(1,\dots,N)}}]= 0$; one can follow the exact same arguments as in the end of the proof of \cite[Lemma B.9]{PST_BackPropagChaos} .

    \item\label{Remark 3.1 ii}
The independence of the increments of the martingale $\overline{X}^{k,i}$ is equivalent to its associated triplet being deterministic, see \cite[Corollary 7.87]{medvegyev2007stochastic} or \cite[Chapter II, Theorem 4.15]{jacod2013limit}.
In particular, the aforementioned property implies that $C^{k,i}, K^{k,i}$ and $c^{k,i}$ are deterministic, see \eqref{def_C}, \eqref{def:Kernels} and \eqref{def_c}, respectively. 
In view of the fact that the elements of the sequence $\{ \overline{X}^{k,i}\}_{i\in\mathbb{N}}$ are identically distributed, we have that  $C^{k,1}=C^{k,i}$ for every $i\geq 2$, see \cite[Remark 4.14]{PST_BackPropagChaos} and, consequently from \eqref{def:Kernels} and \eqref{def_c}, $K^{k,1}=K^{k,i}$  and $c^{k,1}=c^{k,i}$ for every $i\geq 2$. 
Hence, omitting the index $i$ in the notation of $C^{k,i},K^{k,i}$ and $c^{k,i}$, which was introduced in \cref{footnote:Drop_index_i}, is justified.

\item\label{Remark 3.1 iii} Let us also clarify that the stochastic processes associated with the data $\mathscr{D}^k$ will be stopped at time $T^k$.
\end{enumerate}   
\end{remark}

Let us proceed  with a brief discussion about the notation involving filtrations, which will help in further simplifying the notation.
In the following, we assume that a set of standard data is fixed with associated index $k$, for $k\in\overline{\mathbb{N}}$.
Let us start with a subtle point,
namely the distinction between the use of the filtration associated to the $i-$th player and of that associated to the set of the first $N$ players; recall \ref{H1} and \cref{Remark 3.1} for the respective notation. 
Given the set of standard data and because of the immersion of the aforementioned smaller filtration into the larger one, see \cref{Remark 3.1}.\ref{Remark 3.1 i}, there is no harm in placing ourselves under the framework of any of the two when there are no measurability conflicts; see \emph{e.g.} \cite[Remark 2.1 and Section B.2]{PST_BackPropagChaos}.
If, however, there are sound reasons for a clarification, then this will be promptly provided.
Therefore, we will omit the symbol of a filtration, \emph{e.g.}, the predictable quadratic covariation will be simply denoted by $\langle \cdot,\cdot \rangle$, since the respective measurability will be easily concluded by the argument process. 
Having this simplification in mind, we will denote the predictable $\sigma-$algebra associated to the filtration $\mathbb{F}^{k,i}$ simply by $\mathcal{P}^{k,i}$, instead of $\mathcal{P}^{\mathbb{F}^{k,i}}$, and analogous simplification will be used for $\widetilde{\mathcal{P}}^{k,i}$.
Moreover, we simplify the notation associated to (integer-valued, random) measures. 
The integer-valued measure $\mu^{X^{k,\natural}}$ will be denoted by $\mu^{k,\natural}$, its compensator will be denoted by $\nu^{k,\natural}$ and the compensated integer--valued measure 
 will be denoted by $\widetilde{\mu}^{k,\natural}$.

Finally, we conclude this discussion about the notational simplifications with some comments related to the spaces introduced in \Cref{Norms and spaces}.
Given the set of standard data $\mathscr{D}^k$, for $k\in\overline{\mathbb{N}}$, a label $i$ of a player and the total number of players $N$, the main symbol of a space will be kept and the superscripts will be modified as follows:
the indices $k,i,N$ (if all necessary) will be affixed, succeeding the number $2$ and preceding any of the symbols $\circ, \natural, \perp$, if they are required; the dimensions of the state spaces will be omitted only when no confusion arises.
As an example, 
$\mathcal{H}^{2}_{\beta}(\mathbb{F}^{k,i},A^{k},\overline{X}^{\perp_{\mathbb{F}^{k,i}}};\mathbb{R}^{d})$
will be simply denoted by 
$\mathcal{H}^{2,k,i,\perp}_{\beta}(\mathbb{R}^d)$.

The rule for the norms makes the indices $k,i,N$ perform as subscripts which succeed the symbols $\star$ (if any) and $\beta$, and precede the value $t$, \emph{e.g.}, $ \Vert \cdot\Vert_{\star,\beta,k,i}$, $\tnorm{\cdot}_{k,i,t}$. \label{notation:short_norm}
In case the considered filtration is $\mathbb{F}^{k,(1,\dots,N)}$, the index that denotes the label of the player will be replaced by $(1,\dots,N)$, \emph{e.g.}, 
$\mathcal{H}^{2}_{\beta}(\mathbb{F}^{k,(1,\dots,N)},A^{k},\overline{X}^{\perp_{\mathbb{F}^{k,i}}};\mathbb{R}^{d})$
will be simply denoted by 
$\mathcal{H}^{2,k,(1,\dots,N),\perp}_{\beta}(\mathbb{R}^d)$.

We are now ready to provide the necessary conditions for the stability properties we are interested in proving.
Namely, the stability of McKean--Vlasov BSDEs, which will lead to the stability of the backward propagation of chaos.\footnote{Under conditions \ref{S1}--\ref{S10}, one cannot derive the stability of the mean-field BSDE \eqref{mfBSDE}, for any $N\in\mathbb{N}$. However, under modified assumptions, this can be proved as well; 
}
To this end, extending \citet{papapantoleon2023stability}, we are going to complement Assumptions \ref{H1}--\ref{H:prop_contraction} with the ones needed for the convergence of the data $\{\mathscr{D}^k\}_{k \in \mathbb{N}}$ and the convergence of the Lebesgue--Stieltjes integrals associated to the generators of the BSDEs. 

In the sequel, the following assumptions are considered:
\begin{enumerate} [label=\textup{\textbf{(S\arabic*)}}]
\item\label{S1} 
The martingale $X^{\infty,i,\circ}$ is continuous and 
the martingale $X^{\infty,i,\natural}$ is quasi-left-continuous, for every $i \in \mathbb{N}$,.
\item \label{S2} The following convergences hold, for every $i \in \mathbb{N}$, 
\begin{align*}
\overline{X}^{k,i} 
    \xrightarrow[k\to\infty]{\hspace{0.3em}(\textup{J}_1(\mathbb{R}^{p+n}),\mathbb{P})\hspace{0.3em}} 
    \overline{X}^{\infty,i}
    \hspace{0.3em}\text{ and }\hspace{0.3em}
\overline{X}^{k,i}_{\infty} 
    \xrightarrow[k\to\infty]{\hspace{0.3em}
    \mathbb{L}^2(\Omega,\mathcal{G},\mathbb{P};\mathbb{R}^{p+n})
    \hspace{0.3em}} 
    \overline{X}^{\infty,i}_{\infty}.
\end{align*}

\item The pair $\overline{X}^{\infty,i}$ possesses the $\mathbb{F}^{\infty,i}-$predictable representation property, for every $i \in \mathbb{N}$.
%
\item\label{S4} The following convergence holds, for every $i \in \mathbb{N}$,
\begin{align*}
     \big\|\xi^{k,i,N} - \xi^{k,i}\big\|^2_{\mathbb{L}^2(\Omega,\mathcal{G},\mathbb{P};\mathbb{R}^d)} \xrightarrow[(k,N) \to(\infty,\infty)]{|\cdot|} 0.
\end{align*}
Moreover,
\begin{align*}
    \frac{1}{N} \sum_{i = 1}^N \big\|\xi^{k,i,N} - \xi^{k,i}\big\|^2_{\mathbb{L}^2(\Omega,\mathcal{G},\mathbb{P};\mathbb{R}^d)} \xrightarrow[(k,N) \to(\infty,\infty)]{|\cdot|} 0.
\end{align*}


\item\label{S5} The following convergence holds, for every $i \in \mathbb{N}$,
\begin{align*}
    \xi^{k,i} \xrightarrow[k \rightarrow \infty]{\mathbb{L}^2(\Omega,\mathcal{G},\mathbb{P};\mathbb{R}^d)} \xi^{\infty,i}.
\end{align*}

\item\label{S6} The sequence of random variables
\begin{align*}
    \Bigg\{\int_{0}^{\infty}
    \frac{|f^k(s,0,0,0,\delta_0)|^2}{(\alpha^k)^2_s}\ud C^{k}_s\Bigg\}_{k \in \overline{\mathbb{N}}}
\end{align*}
is uniformly integrable, where $\delta_0$ is the Dirac measure on the domain of the last argument concentrated at $0$, the neutral element of the addition.

\item\label{S7} There exists $\overline{A} \in \mathbb{R}_+$ such that 
the sequence of real numbers $\{A^{k}_{\infty}\}_{k \in \overline{\mathbb{N}}}$ is bounded from $\overline{A}$; see \cref{Remark 3.1}.\ref{Remark 3.1 ii}.

\item\label{S8} The generators $\{f^k\}_{k \in \overline{\mathbb{N}}}$ possess, additionally to \ref{H4}, the following properties:
\begin{enumerate}
    \item For every $k \in \overline{\mathbb{N}},i \in \mathbb{N}, a \in \mathbb{D}^d, Z \in D^{\circ,d\times p}, U \in D^{\natural}, \mu \in \mathbb{D}(\mathscr{P}_2(\mathbb{R}^d),W_{2,|\cdot|})$, 
    %
    it holds that 
    \begin{align*}
f^k\big(\cdot,a_{\cdot},Z_{\cdot},
\Gamma^{k,i}
(U)_{\cdot}, \mu_{\cdot}\big)
\in \mathbb{D}^d.
    \end{align*}
    
    \item For every $i \in \mathbb{N}, Z \in D^{\circ,d\times p}, U \in D^{\natural}$ and a sequence of $\mathbb{R}^d-$valued stochastic processes $\{a^k\}_{k \in \overline{\mathbb{N}}}$ 
    such that 
    $\mathbb{E}\left[\sup_{t \in \mathbb{R}_+}|a^k_t|^2\right] < \infty$ for every $k\in \overline{\mathbb{N}}$, if $a^k \xrightarrow[k \rightarrow \infty]{\textup{J}_1(\mathbb{R}^d,|\cdot|)}a^{\infty},\hspace{0.1cm} \mathbb{P}-a.s.$, 
    then
\begin{align*}
f^k\big(\cdot,a_{\cdot}^k,Z_{\cdot},
\Gamma^{k,i}
(U)_{\cdot}, \mathcal{L}(a^k_{\cdot})\big)
\xrightarrow[k \rightarrow \infty]{\textup{J}_1(\mathbb{R}^d,|\cdot|)} 
f^{\infty}\big(\cdot,a_{\cdot}^{\infty},Z_{\cdot},
\Gamma^{\infty,i}
(U)_{\cdot}, \mathcal{L}(a^{\infty}_{\cdot})\big)
\quad\mathbb{P}-a.s. 
\end{align*}
    Moreover, if $\sup_{k \in \overline{\mathbb{N}}}\left\{\|a^k(\omega)\|_{\infty}\right\}_{k \in \overline{\mathbb{N}}} < \infty, \hspace{0.1cm} \mathbb{P}-a.s.$, then
\begin{align*}
    \sup_{k \in \overline{\mathbb{N}}}
    \left\{
    \left\|
    f^k\big(\cdot,a_{\cdot}^k,Z_{\cdot},
    \Gamma^{k,i}
    (U)_{\cdot}, \mathcal{L}(a^k_{\cdot})\big)
    \right\|_{\infty}\right\}_{k \in \overline{\mathbb{N}}} < \infty, \hspace{0.1cm} \mathbb{P}-a.s.
\end{align*}
\end{enumerate}
\item \label{S9}
\begin{enumerate}
    \item\label{S9.i} The sequence $\{\Phi^k\}_{k \in \overline{\mathbb{N}}}$ satisfies $\Phi^k \xrightarrow[k \rightarrow \infty]{|\cdot|} \Phi^{\infty} := 0$.

    \item\label{S9.ii} $\widetilde{M}^0(\hat{\beta}) \overset{\triangle}{=} \frac{ 2 \sqrt{\frac{2}{\hat{\beta}} + 9}\sqrt{\frac{2}{\hat{\beta}} + 17}  + \frac{4}{\hat{\beta}} + 35}{\hat{\beta}} < \frac{1}{4}$.
\end{enumerate}


\item\label{S10} The time horizon $T^{\infty}$ is finite and $T^k \xrightarrow[k \rightarrow \infty]{|\cdot|} T^{\infty}$.
\end{enumerate}

Let us add some remarks regarding the set of Assumptions \ref{S1}--\ref{S10}.
In view of Assumption \ref{H1}, the following are true for every $i\in\mathbb{N}$.
In \ref{S1}, we have assumed that the process $\overline{X}^{\infty,i}$ is quasi-left-continuous. 
Hence, the filtration $\mathbb{F}^{\infty,i}$ is also quasi-left-continuous, because this is the one generated by $\overline{X}^{\infty,i}$.
Actually, since $\overline{X}^{\infty,i}$ has independent increments, being quasi-left-continuous is equivalent to having no fixed times of discontinuity, see \citet[Section 1]{Wang_RemarksIndependentIncrem}.
Next, the alerted reader may observe that we have not (explicitly) assumed the weak convergence of the filtrations $\{\mathbb{F}^{k,i}\}_{k\in\overline{\mathbb{N}}}$, \emph{i.e.}, we did not assume the analogon of \citet[Assumption (S4)]{papapantoleon2023stability}.
This apparent omission can be justified because the combination of \ref{H1} and \ref{S2} yields that 
\begin{align}\label{weak_conv_Filtrations}
\mathbb{F}^{k,i}\xrightarrow[k\to\infty]{\textup{w}}\mathbb{F}^{\infty,i};
\end{align}
see \citet[Proposition 2]{coquet2001weak}.

Having justified the validity of \eqref{weak_conv_Filtrations}, for every $i\in\mathbb{N}$, we may proceed to discuss another apparent omission, namely the lack of assumptions on the convergence of the sequences 
$\{C^k_{\cdot}\}_{k\in\overline{\mathbb{N}}}$ and
$\{C^k_{\infty}\}_{k\in\overline{\mathbb{N}}}$
, \emph{i.e.}, we did not assume the analogon to \citet[Assumption (S9)]{papapantoleon2023stability}.
Recalling \cref{rem:crucial_remarks}.\ref{rem:C_properties}, the function $C^k$ is defined as the trace of $\langle \overline{X}^k \rangle$, for every $k\in\overline{\mathbb{N}}.$ 
The combination of \ref{S2} and \eqref{weak_conv_Filtrations} yields, via \citet[Corollary 12]{memin2003stability}, the convergence 
\begin{align*}
    \langle \overline{X}^{k,i} \rangle 
    \xrightarrow[k\to\infty]{\hspace{0.3em}\textup{J}_1(\mathbb{R}^{p+n})\hspace{0.3em}}
    \langle \overline{X}^{\infty,i} \rangle,
\end{align*}
which, in particular, implies, recalling \cref{rem:crucial_remarks}.\ref{rem:C_properties}, that
\begin{align*}
    C^k=\textup{Tr}\big[\langle \overline{X}^{k,i} \rangle\big] 
    \xrightarrow[k\to\infty]{\hspace{0.3em}\textup{J}_1(\mathbb{R}^{p})\hspace{0.3em}}
    \textup{Tr}\big[\langle \overline{X}^{\infty,i} \rangle\big] = C^{\infty}.
\end{align*}
The continuity of $C^{\infty}$, because of the quasi-left-continuity of $\overline{X}^{\infty,i}$ for every $i\in\mathbb{N}$, allows the above convergence to hold under the locally uniform convergence.
In view of \ref{S10}, we then have
\begin{align*}
    C^k_{\infty}=C^k_{T^{k}} 
    \xrightarrow[k\to\infty]{\hspace{0.3em}|\cdot|\hspace{0.3em}}
    C^{\infty}_{T^{\infty}}= C^{\infty}_{\infty},
\end{align*}
because the limit time horizon $T^{\infty}$ is finite.

At this point, let us present an equivalent form of Assumption \ref{S4}. 
More precisely, in view of \ref{H2}, \ref{S4} is equivalent to
\begin{gather*}
    \sup_{k \in \mathbb{N}} \left\{ \big\|\xi^{k,i,N} - \xi^{k,i}\big\|^2_{\mathbb{L}^2(\Omega,\mathcal{G},\mathbb{P};\mathbb{R}^d)}\right\} \xrightarrow[N \rightarrow \infty]{|\cdot|} 0,
    \shortintertext{and}
    \sup_{k \in \mathbb{N}} \left\{ \frac{1}{N} \sum_{i = 1}^N \big\|\xi^{k,i,N} - \xi^{k,i}\big\|^2_{\mathbb{L}^2(\Omega,\mathcal{G},\mathbb{P};\mathbb{R}^d)}\right\} \xrightarrow[N \rightarrow \infty]{|\cdot|} 0.
\end{gather*}

Finally, let us turn our attention to Assumption \ref{S9}.\ref{S9.ii}. 
We want to point out that the left-hand side of the inequality corresponds to the summand of $\widetilde{M}^{\Phi}(\hat{\beta})$, which is independent of $\Phi$, \emph{i.e.}, 
it coincides with $\widetilde{M}^0(\hat{\beta})$; see \ref{H:prop_contraction}.
Combining the two parts of Assumption \ref{S9}, we get that finally, for all $k\in\mathbb{N}$, it holds that
\begin{align}\label{ineq:Mtilde_smaller_quarter}
    \widetilde{M}^{\Phi^k}(\hat{\beta}) <\frac{1}{4}.
\end{align}
In other words, \ref{S9} strengthens \ref{H:prop_contraction}.


\section{Stability properties}
\label{sec:StabilityProperties}

In this section, we will present the main result of the current work, namely the stability of the backward propagation of chaos. 
We have already described in the introduction the approach we intend to follow.
In order to prepare the main result, we will also provide a stability result for McKean--Vlasov BSDEs. The stability of McKean--Vlasov BSDEs is a direct generalization of analogous results for (standard) BSDEs, see \textit{e.g.} \citet{briand2002robustness} and \citet{papapantoleon2019stability}, and means that given a convergent sequence of standard data for a McKean--Vlasov BSDE, each associated to each own filtration, then the associated (unique) solutions are also convergent.
The stability of backward propagation of chaos is a novel result, that states that given a sequence of mean-field BSDEs that satisfy the backward propagation of chaos property, then the limiting McKean--Vlasov BSDEs are convergent.

Let us point out that, in the framework described so far, we cannot establish the stability of mean-field BSDEs for any $N\in\mathbb{N}$; for this reason, there is no down arrow appearing between $\mathscr{S}^{k,i,N}$ and $\mathscr{S}^{\infty,i,N}$ in the second to last row of the \cref{table:propagation scheme}. 
However, this stability is not necessary in order to prove the main result, \textit{i.e.} the stability of the backward propagation of chaos.
The lack of stability in the aforementioned case is mainly because we have not imposed any condition on the convergence of the respective terminal conditions.    
Once such a condition is imposed, \emph{i.e.}, for $N\in\mathbb{N}$
\begin{align*}
    \xi^{k,i,N} \xrightarrow[k \rightarrow \infty]{\mathbb{L}^2(\Omega,\mathcal{G},\mathbb{P},\mathbb{R}^d)} \xi^{\infty,i,N},
\end{align*}
accompanied with a suitable extension of the continuity property for the generators, then the stability of mean-field BSDEs for $N$ players can also be proved.
These conditions appear later as \ref{S5'} and \ref{S8'}. 

At this point, let us recall the versions of the Moore--Osgood theorem that we are going to use; see \citet[Chapter VI, Sections 336–337]{hobson1907theory}.
\begin{itemize}
    \item[] Given a double indexed sequence $\{s_{k,q}\}_{(k,q)\in \mathbb{N}\times\mathbb{N}}$ of real numbers we have that 
    \begin{align*}
    \lim_{k\rightarrow \infty}\lim_{q\rightarrow \infty} s_{k,q} = \lim_{q\rightarrow \infty}\lim_{k\rightarrow \infty} s_{k,q} = \lim_{(k,q)\rightarrow(\infty,\infty)} s_{k,q} \in \mathbb{R}
    \end{align*}
    if any of the following conditions hold:
    \begin{enumerate}
        \item The $\lim_{q\rightarrow\infty}\lim_{k\rightarrow \infty}s_{k,q}$ exists, and the convergences $\lim_{k\rightarrow \infty}s_{k,q}$ are uniform with respect to $q \in \mathbb{N}$.
        \item The $\lim_{q\rightarrow\infty}s_{k,q}, \lim_{k\rightarrow \infty}s_{k,q}$ finally exist for every $k,q \in \mathbb{N}$, and one type of the convergences is uniform with respect to the index of the other.
    \end{enumerate}
    Note that, although in the above statement the sequence $\{s_{k,q}\}_{(k,q)\in \mathbb{N}\times \mathbb{N}}$ takes values in $\mathbb{R}$, that is only for simplicity.
    In practice, we could have an arbitrary metric space in place of $\mathbb{R}$.
\end{itemize}

The union of Assumptions \ref{H1}--\ref{H:prop_contraction} and \ref{S1}--\ref{S10} leads to the desired stability property for the McKean--Vlasov BSDE associated to the $i-$th player.
The next theorem serves as the precise statement for this result.

\begin{theorem}[Stability of McKean--Vlasov BSDEs]\label{McKean--Vlasov Stability}
Consider a sequence of data $\{\mathscr{D}^k\}_{k\in\overline{\mathbb{N}}}$ as described in \eqref{def:seq_standard_data}, which satisfy \ref{H1}--\ref{H:prop_contraction} for every $k\in\overline{\mathbb{N}}$, and assume that \ref{S1}--\ref{S10} are in force.
Then, the following are true, for every $i\in\mathbb{N}$:
\begin{align}\label{MVStab:con.1} \begin{multlined}[0.9\textwidth]
\Big(Y^{k,i},Z^{k,i}\cdot X^{k,i,\circ} + U^{k,i} \star \widetilde{\mu}^{{k,i,\natural}},M^{k,i}\Big)
\xrightarrow[k \rightarrow \infty]{\hspace{0.3em}(\textup{J}_1,\mathbb{L}^2)\hspace{0.3em}}
\Big(Y^{\infty,i},Z^{\infty,i}\cdot X^{\infty,i,\circ} + U^{\infty,i} \star \widetilde{\mu}^{{\infty,i,\natural}},0\Big),
\end{multlined} \end{align}
\begin{align}\label{MVStab:con.2}\begin{multlined}[0.9\textwidth]
\Big(
    [Y^{k,i}],[Z^{k,i}\cdot X^{k,i,\circ} + U^{k,i} \star \widetilde{\mu}^{{k,i,\natural}}],
    [M^{k,i}],[Y^{k,i},X^{k,i,\circ}],[Y^{k,i},X^{k,i,\natural}],[Y^{k,i},M^{k,i}]\Big) \\
\xrightarrow[k \rightarrow \infty]{\hspace{0.3em}(\textup{J}_1,\mathbb{L}^1)\hspace{0.3em}} 
\Big(
    [Y^{\infty,i}],[Z^{\infty,i}\cdot X^{\infty,i,\circ} + U^{\infty,i} \star \widetilde{\mu}^{{\infty,i,\natural}}],
    0,[Y^{\infty,i},X^{\infty,i,\circ}],[Y^{\infty,i},X^{\infty,1,\natural}],0\Big)
\end{multlined} \end{align}
and 
\begin{align}\label{MVStab:con.3}\begin{multlined}[0.9\textwidth]
\Big(
    \langle Y^{k,i} \rangle,
    \langle Z^{k,i}\cdot X^{k,i,\circ} + U^{k,i} \star \widetilde{\mu}^{{k,i,\natural}}\rangle,
    \langle M^{k,i}\rangle, 
    \langle Y^{k,i}, X^{k,i,\circ} \rangle, 
    \langle Y^{k,i}, X^{k,i,\natural} \rangle, 
    \langle Y^{k,i}, M^{k,i} \rangle 
    \Big)\\
\xrightarrow[k \rightarrow \infty]{\hspace{0.3em}(\textup{J}_1,\mathbb{L}^1)\hspace{0.3em}} 
\Big(
    \langle Y^{\infty,i} \rangle,
    \langle Z^{\infty,i}\cdot X^{\infty,i,\circ} + U^{\infty,i} \star \widetilde{\mu}^{{\infty,i,\natural}}\rangle,
    0,
    \langle Y^{\infty,i}, X^{\infty,i,\circ} \rangle, 
    \langle Y^{\infty,i}, X^{\infty,i,\natural} \rangle, 
    0\Big),
\end{multlined} \end{align}
where in \eqref{MVStab:con.1} the state space is $\mathbb{R}^{d\times 3}$, while in the other two cases the state space is $\mathbb{R}^{d\times (4d+p+n)}$.  
\end{theorem}

\begin{proof}
We will treat the case $i=1$, for notational simplicity.
We will follow the same strategy as in \citet{papapantoleon2023stability}, which is described schematically in \cref{table:McKean--Vlasovscheme}. 
In other words, we will apply the Moore--Osgood theorem on the doubly-indexed sequence 
$\left\{\mathscr{S}^{k,1,(q)}\right\}_{(k,q)\in \mathbb{N} \times \mathbb{N}}$ of the Picard approximation scheme, where $\mathscr{S}^{k,1,(q)}$ denotes the representation obtained at the $q-$th step of the Picard iteration under the data $\mathscr{D}^k$.

\begin{table}[ht]
{\centering 
\begin{tabular}{c|ccccccc}
$\mathscr{D}^1$ 	&{$\mathscr{S}^{1,1,(0)}$} 			&{$\mathscr{S}^{1,1,(1)}$} 			&{$\mathscr{S}^{1,1,(2)}$}		
 					&{$\cdots$}							&{$\mathscr{S}^{1,1,(q)}$}			&{$\xrightarrow{\ q\to\infty\ }$} 	&{$\mathscr{S}^{1,1}$}\\[0.3cm]
$\mathscr{D}^2$ 	&{$\mathscr{S}^{2,1,(0)}$} 			&{$\mathscr{S}^{2,1,(1)}$} 			&{$\mathscr{S}^{2,1,(2)}$}
 					&{$\cdots$}							&{$\mathscr{S}^{2,1,(q)}$}			&{$\xrightarrow{\ q\to\infty\ }$}	&{$\mathscr{S}^{2,1}$}\\[0.3cm]
$\mathscr{D}^3$ 	&{$\mathscr{S}^{3,1,(0)}$} 			&{$\mathscr{S}^{3,1,(1)}$} 			&{$\mathscr{S}^{3,1,(2)}$}
					&{$\cdots$}							&{$\mathscr{S}^{3,1,(q)}$}			&{$\xrightarrow{\ q\to\infty\ }$}	&{$\mathscr{S}^{3,1}$}\\[0.1cm]
$\vdots$			&{$\vdots$} 						&{$\vdots$} 						&{$\vdots$} 						& 
					&{$\vdots$}							& 												&{$\vdots$} \\[0.2cm]
$\mathscr{D}^k$ 	&{$\mathscr{S}^{k,1,(0)}$} 			&{$\mathscr{S}^{k,1,(1)}$} 			&{$\mathscr{S}^{k,1,(2)}$}			&{$\cdots$}
					&{$\mathscr{S}^{k,1,(q)}$}			&{$\xrightarrow{\ q\to\infty\ }$}	&{$\mathscr{S}^{k,1}$}\\
$\big\downarrow$	&\rotatebox[origin=c]{270}{$\dashrightarrow{}$} 	&\rotatebox[origin=c]{270}{$\dashrightarrow{}$}  	&\rotatebox[origin=c]{270}{$\dashrightarrow{}$}  	&  		
					&\rotatebox[origin=c]{270}{$\dashrightarrow{}$}  	&									& \rotatebox[origin=c]{270}{\Large${\rightsquigarrow}$} \\[0.2cm]
$\mathscr{D}^\infty$&{$\mathscr{S}^{\infty,1,(0)}$} 		&{$\mathscr{S}^{\infty,(1)}$} 		&{$\mathscr{S}^{\infty,1,(2)}$}		&{$\cdots$}
					&{$\mathscr{S}^{\infty,1,(q)}$}		&{$\xrightarrow{\ q\to\infty\ }$}	&{$\mathscr{S}^{\infty,1}$}\\[0.2cm]
\end{tabular}\\}
\caption{The doubly-indexed Picard scheme for the stability of McKean--Vlasov BSDEs.}
\label{table:McKean--Vlasovscheme}
\end{table}

Using \cref{prop:uniform_Picard_MVSDE}, we obtain the uniform (over $k\in\overline{\mathbb{N}}$) horizontal convergence; in \cref{table:McKean--Vlasovscheme} this corresponds to the horizontal, right arrows denoting the convergence of the Picard schemes as $q\to\infty$.
In order to apply the Moore--Osgood theorem, it suffices to prove that the convergence indicated by the dashed, vertical down arrows in \cref{table:McKean--Vlasovscheme} is indeed true, for every $q \in \mathbb{N}$, as $k\to\infty$. 
In order to avoid the redundant repetition of arguments and results presented in \cite{papapantoleon2023stability}, we will describe only the necessary modifications wherever needed; these are presented in \cref{lem:MVBSDE_Picard_Induction} in the next section.
\end{proof}

%
%
%
%
%
%
%
%
%
%
%
%

The main theorem of the present work follows. 

\begin{theorem}[Stability of backward propagation of chaos]\label{Prop. of Chaos Stability}
Consider a sequence of data $\{\mathscr{D}^k\}_{k\in\overline{\mathbb{N}}}$ as described in \eqref{def:seq_standard_data}, which satisfy \ref{H1}--\ref{H:prop_contraction} for every $k\in\overline{\mathbb{N}}$, and assume that \ref{S1}--\ref{S10} are in force.
Then, the following are true for every $i\in\mathbb{N}$:
\begin{align}\label{conv:Stab_Prop_1}
\begin{multlined}[0.9\textwidth]
      \left(Y^{k,i,N}, Z^{k,i,N} \cdot X^{k,i,\circ} + U^{k,i,N} \star \widetilde{\mu}^{{k,i,\natural}},M^{k,i,N}\right)  \\
\xrightarrow[(k,N)\rightarrow (\infty,\infty)]
{\hspace{0.3em}(\textup{J}_1,\mathbb{L}^2)\hspace{0.3em{}}} 
      \left(Y^{\infty,i}, Z^{\infty,i} \cdot X^{\infty,i,\circ} + U^{\infty,i} \star \widetilde{\mu}^{{\infty,i,\natural}},0\right),
  \end{multlined}
\end{align}
\begin{align}\label{conv:Stab_Prop_2}\begin{multlined}[0.9\textwidth]
\Big(
    [Y^{k,i,N}],[Z^{k,i,N}\cdot X^{k,i,\circ} + U^{k,i,N} \star \widetilde{\mu}^{{k,i,\natural}}],
    [M^{k,i,N}],
    [Y^{k,i,N},X^{k,i,\circ}],[Y^{k,i,N},X^{k,i,\natural}],[Y^{k,i,N},M^{k,i,N}]\Big) \\
\xrightarrow[(k,N)\rightarrow (\infty,\infty)]{\hspace{0.3em}(\textup{J}_1,\mathbb{L}^1)\hspace{0.3em}}
\Big(
    [Y^{\infty,i}],[Z^{\infty,i}\cdot X^{\infty,i,\circ} + U^{\infty,i} \star \widetilde{\mu}^{{\infty,i,\natural}}],
    0,[Y^{\infty,i},X^{\infty,i,\circ}],[Y^{\infty,i},X^{\infty,1,\natural}],0\Big)
\end{multlined} \end{align}
and 
\begin{align}\label{conv:Stab_Prop_3}\begin{multlined}[0.9\textwidth]
\Big(
    \langle Y^{k,i,N} \rangle,
    \langle Z^{k,i,N}\cdot X^{k,i,\circ} + U^{k,i,N} \star \widetilde{\mu}^{{k,i,\natural}}\rangle,
    \langle M^{k,i,N}\rangle, 
    \langle Y^{k,i,N}, X^{k,i,\circ} \rangle, 
    \langle Y^{k,i,N}, X^{k,i,\natural} \rangle, 
    \langle Y^{k,i,N}, M^{k,i,N} \rangle 
    \Big)\\
\xrightarrow[k \rightarrow \infty]{\hspace{0.3em}(\textup{J}_1,\mathbb{L}^1)\hspace{0.3em}} 
\Big(
    \langle Y^{\infty,i} \rangle,
    \langle Z^{\infty,i}\cdot X^{\infty,i,\circ} + U^{\infty,i} \star \widetilde{\mu}^{{\infty,i,\natural}}\rangle,
    0,
    \langle Y^{\infty,i}, X^{\infty,i,\circ} \rangle, 
    \langle Y^{\infty,i}, X^{\infty,i,\natural} \rangle, 
    0\Big),
\end{multlined} \end{align}
where in \eqref{conv:Stab_Prop_1} the state space is $\mathbb{R}^{d\times 3}$, while in the other two the state space is $\mathbb{R}^{d\times (4d+p+n)}$.  
\end{theorem}

\begin{proof}
Let us consider the case $i=1$ again, for notational simplicity.
The main argument will be (again) the application of the Moore--Osgood theorem for a doubly-indexed sequence, namely $\{\mathscr{S}^{k,1,N}\}_{(k,N)\in {\mathbb{N}}\times \mathbb{N}}$.
Analogously to the proof of the previous theorem, we first need to prove the uniform (over $k$) backward propagation of chaos, which is provided by \cref{Unif. Prop. Chaos}.
The existence of the iterated limit 
\begin{align*}
    \lim_{k\to\infty}\lim_{N\to\infty} \mathscr{S}^{k,1,N} = \mathscr{S}^{\infty,1}
\end{align*}
is guaranteed by \citet[Theorem 6.3]{PST_BackPropagChaos} and \cref{McKean--Vlasov Stability}. 
Then, using \cref{Corollary 5.7}, we can conclude with the same arguments as in \cite[Sections 3.3.2, 3.3.3]{papapantoleon2023stability}.
\end{proof}
As an immediate corollary of the above result we can get the stability of backward propagation of chaos expressed in distributional form.


\begin{corollary}
Let $(k,i,N) \in \overline{\mathbb{N}}\times \mathbb{N} \times \mathbb{N}$ and define
\begin{gather*}
    \theta^{k,i,N} := \textup{Law}\Big(
    \big(Y^{k,i,N}, Z^{k,i,N} \cdot X^{k,i,\circ} + U^{k,i,N} \star\widetilde{\mu}^{{k,i,\natural}},M^{k,i,N}\big)
    \Big)
\shortintertext{and}
    \theta^{\infty,i} :=\textup{Law}\Big(
    \big(Y^{\infty,i}, Z^{\infty,i} \cdot X^{\infty,i,\circ} + U^{\infty,i} \star\widetilde{\mu}^{{\infty,i,\natural}},0\big)
    \Big).
\end{gather*}
Then, for every $i \in \mathbb{N}$, we have
\begin{align}
W^2_{2,\rho_{\textup{J}_1^{d \times 3}},}\big(\theta^{k,i,N},\theta^{\infty,i}\big) \xrightarrow[(k,N)\rightarrow (\infty,\infty)]{|\cdot|} 0.
\end{align}
\end{corollary}

\begin{proof}
Let us define, for notational simplicity, 
\begin{align*}
    Q^{k,i,N}:= (Y^{k,i,N}, Z^{k,i,N} \cdot X^{k,i,\circ} + U^{k,i,N} \star \widetilde{\mu}^{{k,i,\natural}},M^{k,i,N})
\end{align*}
and 
\begin{align*}
    Q^{k,i}:= (Y^{k,i}, Z^{k,i} \cdot X^{k,i,\circ} 
        + U^{k,i} \star\widetilde{\mu}^{{k,i,\natural}},M^{k,i}),
\end{align*}
for every $(k,i,N) \in \overline{\mathbb{N}}\times \mathbb{N} \times \mathbb{N}$.
Then, from the definition of the Wasserstein distance, we get
    \begin{align*} 
    W^2_{2,\rho_{\textup{J}_1^{d \times3}}}\big(\theta^{k,i,N},\theta^{\infty,i}\big) 
    \leq \int_{\mathbb{D}^{d\times 3}\times \mathbb{D}^{d\times 3}} \rho_{\textup{J}_1^{d \times 3}}^2(x,z)\pi(\ud x,\ud z)
    \leq \mathbb{E}\Big[\rho^2_{\textup{J}_1^{d \times 3}}\big(Q^{k,i,N},Q^{\infty,i}\big)\Big],
    \end{align*}
where we chose $\pi$ to be the image measure on $\mathbb{D}^{d \times 3} \times \mathbb{D}^{d \times 3}$ produced by the measurable function
\begin{align*}
(Q^{k,i,N}, Q^{\infty,i}): \Omega \longrightarrow \mathbb{D}^{d \times 3} \times \mathbb{D}^{d\times 3}.
\end{align*}
Hence, the result follows directly from convergence \eqref{conv:Stab_Prop_1}, see \cref{Prop. of Chaos Stability}. 
\end{proof}

As pointed out earlier, we cannot deduce the stability of mean-field BSDE systems under Assumptions \ref{S1}--\ref{S10} from the previous theorem, for any $N\in\mathbb{N}$; instead, we can only deduce the following result
\begin{align*}
    \lim_{N\to\infty}\limsup_{k\to\infty} \big\| \mathscr{S}^{k,1,N} - \mathscr{S}^{\infty,i,N} \big\|_{\mathcal{S}^2} = 0,
\end{align*}
which can be interpreted as an \emph{asymptotic} stability of mean-field BSDE systems. 
Let us conclude this section with a theorem that provides the stability of mean-field BSDE systems.
To this end, let us fix the number of players $N\in\mathbb{N}$ until the end of this section. 
We need to modify Assumptions \ref{S5} and \ref{S8} as follows:
%
%
\begin{enumerate}[label=\textup{\textbf{(S\arabic*${}^{\prime}$)}}]

\setcounter{enumi}{4}
\item\label{S5'} For every $i \in \mathscr{N}$, we have
\begin{align*}
    \xi^{k,i,N} \xrightarrow[k \rightarrow \infty]{\mathbb{L}^2(\Omega,\mathcal{G},\mathbb{P},\mathbb{R}^d)} \xi^{\infty,i,N}.
\end{align*}
\setcounter{enumi}{7}
\item\label{S8'}  The generators $\{f^k\}_{k \in \overline{\mathbb{N}}}$ possess the following properties:
\begin{enumerate}
    \item For every $k \in \overline{\mathbb{N}}$,
    $i \in \mathscr{N}, a \in \mathbb{D}(\mathbb{R}^d), Z \in D^{\circ,d\times p}, U \in D^{\natural}, \mu \in \mathbb{D}(\mathscr{P}_2(\mathbb{R}^d),W_{2,|\cdot|})$, it holds that 
    \begin{align*}
f^k(\cdot,a_\cdot,Z_\cdot,\Gamma^{k,(1,\dots,N)}(U)_\cdot, \mu_\cdot) \in \mathbb{D}(\mathbb{R}^d).
    \end{align*}
    
    \item For every 
    $Z \in D^{\circ,d\times p}$,
    $ U \in D^{\natural}$ and a sequence of $\mathbb{R}^{d\times N}-$valued c\`adl\`ag stochastic processes $\{\textbf{a}^{k}\}_{k\in\overline{\mathbb{N}}}$, where $\textbf{a}^{k} := (a^{k,1},\dots,a^{k,N})$ 
    such that $\mathbb{E}\big[\sup_{t \in \mathbb{R}_+}|\textbf{a}^{k}_t|^2\big] < \infty$ for every $k\in \overline{\mathbb{N}}$, if $\textbf{a}^{k} \xrightarrow[k \rightarrow \infty]{\textup{J}_1(\mathbb{R}^{d\times N})}\textbf{a}^{\infty},\hspace{0.1cm} \mathbb{P}-a.s.$, 
    then, for every $i\in\mathscr{N}$, it holds    $\mathbb{P}-a.s.$ 
    \begin{align*}
    \hspace{5em}
    \begin{multlined}[c][0.85\textwidth]
f^k\big(\cdot,a_\cdot^{k,i},Z_\cdot,\Gamma^{k,(1,\dots,N)}(U)_\cdot, L^N(\textbf{a}^k_\cdot)\big) \\
    \xrightarrow[k \rightarrow \infty]{\hspace{0.3em}\textup{J}_1(\mathbb{R}^d)\hspace{0.3em}} 
f^{\infty}\big(\cdot,a_\cdot^{\infty,i},Z_\cdot,
    \Gamma^{(\mathbb{F}^{\infty,(1,\dots,N)},\overline{X}^{\infty,i},\Theta^{\infty})}(U)_\cdot, L^N(\textbf{a}^{\infty}_\cdot)\big).
    \end{multlined}  
    \end{align*}
Furthermore, if $\sup_{k \in \overline{\mathbb{N}}}\left\{\|a^{k,i}(\omega)\|_{\infty}\right\}_{k \in \overline{\mathbb{N}}} < \infty, \hspace{0.1cm} \mathbb{P}-a.s.$, then
    \begin{align*}
        \sup_{k \in \overline{\mathbb{N}}}\left\{\left\|\left(f^k(t,a_t^{k,i},Z_t,\Gamma^{k,(1,\dots,N)}(U)_t, L^N(\textbf{a}^k_t))\right)_{t\in \mathbb{R}_+}\right\|_{\infty}\right\}_{k \in \overline{\mathbb{N}}} < \infty, \hspace{0.1cm} \mathbb{P}-a.s.
    \end{align*}
\end{enumerate}
\end{enumerate}

\begin{remark}
\begin{enumerate}
    \item\label{remark 4.4 i} 
    The filtration $\mathbb{F}^{k,(1,\dots,N)}$ can be seen as the usual augmentation of the natural filtration of the square integrable martingale $\overline{X}^{k,(1,\dots,N)} := (X^{k,1,\circ},X^{k,1,\natural},\dots,X^{k,N,\circ},X^{k,N,\natural})$. 
    Then, $\overline{X}^{k,(1,\dots,N)}$ has independent increments with respect to $\mathbb{F}^{k,(1,\dots,N)}$, thus, using \citet[Proposition 2]{coquet2001weak} we get that 
 \begin{align*}
         \mathbb{F}^{k,(1,\dots,N)}\xrightarrow[k \rightarrow \infty]{\textup{w}} \mathbb{F}^{\infty,(1,\dots,N)},
    \end{align*}\
    \textit{i.e.} the filtrations converge weakly.
    \item Using \ref{remark 4.4 i} above, \ref{S1} and \citet[Section 1]{Wang_RemarksIndependentIncrem}, we have that for every $N \in \mathbb{N}$ the filtrations $\mathbb{F}^{\infty,(1,\dots,N)}$
    are quasi-left-continuous.
\end{enumerate}
\end{remark}
The next theorem provides the stability of mean-field BSDEs for $N$ players.


\begin{theorem}[Stability of mean-field systems of BSDEs]\label{Mean-field system stability}
Consider a sequence of data $\{\mathscr{D}^k\}_{k\in\overline{\mathbb{N}}}$ as described in \eqref{def:seq_standard_data}, which satisfy \ref{H1}--\ref{H:prop_contraction} for every $k\in\overline{\mathbb{N}}$, and assume that \ref{S1}--\ref{S4},\ref{S5'},\ref{S6}, \ref{S7}, \ref{S8'}, \ref{S9},\ref{S10} are in force, and also fix $N \in \mathbb{N}$.
Then, for every $i \in \mathscr{N}$, we have  
\begin{align}\label{conv:Stab_mf_1} \begin{multlined}[0.9\textwidth]
    \Big(Y^{k,i,N},Z^{k,i,N}\cdot X^{k,i,\circ} + U^{k,i,N} \star \widetilde{\mu}^{{k,i,\natural}},M^{k,i,N}\Big)\\  
    \xrightarrow[k \rightarrow \infty]{\left(\textup{J}_1,\mathbb{L}^2\right)}
    \Big(Y^{\infty,i,N},Z^{\infty,i,N}\cdot X^{\infty,i,\circ} + U^{\infty,i,N} \star \widetilde{\mu}^{{\infty,i,\natural}},0\Big),
\end{multlined} \end{align}
\begin{align}\label{conv:Stab_mf_2} \begin{multlined}[0.9\textwidth]
\Big([Y^{k,i,N}],[Z^{k,i,N}\cdot X^{k,i,\circ} + U^{k,i,N} \star \widetilde{\mu}^{{k,i,\natural}}],
[M^{k,i,N}],\\
[Y^{k,i,N},X^{k,i,\circ}],[Y^{k,i,N},X^{k,i,\natural}],[Y^{k,i,N},M^{k,i,N}]\Big) \\
\xrightarrow[k \rightarrow \infty]{\left(\textup{J}_1,\mathbb{L}^1\right)} 
\Big([Y^{\infty,i,N}],[Z^{\infty,i,N}\cdot X^{\infty,i,\circ} + U^{\infty,i,N} \star \widetilde{\mu}^{{\infty,i,\natural}}],0,\\
[Y^{\infty,i,N},X^{\infty,i,\circ}],[Y^{\infty,i,N},X^{\infty,i,\natural}],0\Big)
\end{multlined} 
\end{align}  
and 
\begin{align}\label{conv:Stab_mf_3} \begin{multlined}[0.9\textwidth]
\Big(\langle Y^{k,i,N} \rangle,
\langle Z^{k,i,N}\cdot X^{k,i,\circ} 
    + U^{k,i,N} \star\widetilde{\mu}^{{k,i,\natural}} \rangle,
\langle M^{k,i,N} \rangle,\\
\langle Y^{k,i,N},X^{k,i,\circ} \rangle,
\langle Y^{k,i,N},X^{k,i,\natural} \rangle,
\langle Y^{k,i,N},M^{k,i,N} \rangle\Big) \\
\xrightarrow[k \rightarrow \infty]{\left(\textup{J}_1,\mathbb{L}^1\right)} 
\Big(\langle Y^{\infty,i,N} \rangle,
\langle Z^{\infty,i,N}\cdot X^{\infty,i,\circ} + U^{\infty,i,N} \star \widetilde{\mu}^{{\infty,i,\natural}} \rangle,0,\\
\langle Y^{\infty,i,N},X^{\infty,i,\circ} \rangle,
\langle Y^{\infty,i,N},X^{\infty,i,\natural} \rangle,0\Big),
\end{multlined} 
\end{align}  
where in \eqref{conv:Stab_mf_1} the state space is $\mathbb{R}^{d\times 3}$, while in the other two the state space is $\mathbb{R}^{d\times (4d+p+n)}$.  
\end{theorem}

\begin{proof}
The proof is completely analogous to the proof of \cref{McKean--Vlasov Stability}.
The missing details to complete this proof appear in \cref{subsec:aux_mfBSDE_Stability}. 
\end{proof}


\section{Proofs of the theorems presented in Section \ref{sec:StabilityProperties}}
\label{sec:proofs}

In the following subsections, we will present certain preparatory and auxiliary results required for the proofs of \cref{McKean--Vlasov Stability,Prop. of Chaos Stability,Mean-field system stability}.
Each subsection will be devoted to a specific theorem for the convenience of the reader.
At this point, we have to extend the convention we made about the role of the indices $k,i,N$ by introducing another integer, namely $q$, which may also take the value $0$, and whose role is to denote the step of the Picard iteration.
The newly introduced index will always succeed the other indices and will be placed within parentheses.
As an example, $\mathscr{S}^{k,i,N,(q)}$ denotes the Picard representation obtained in the $q-$th iteration under the data $\mathscr{D}^k$ for the $i-$th player participating in an $N$-player game.
The interpretation is obvious in case fewer indices appear.


\subsection{Auxiliary results for Theorem \ref{McKean--Vlasov Stability}: Stability of McKean--Vlasov BSDEs}
\label{subsec:aux_MVBSDE_Stability}

In this subsection, the framework of \cref{McKean--Vlasov Stability} is adopted.
The next proposition is an extension of the uniform \textit{a priori} estimates of \citet[Corollary 10]{briand2002robustness} and \citet[Proposition 3.2]{papapantoleon2023stability} to McKean--Vlasov BSDEs driven by general semimartingales, and provides the horizontal convergences in \cref{table:McKean--Vlasovscheme}.

\begin{proposition}[Uniform \emph{a priori} estimates for McKean--Vlasov BSDEs]\label{prop:uniform_Picard_MVSDE}
Let $i\in\mathbb{N}$.
Let us associate to the standard data $\mathscr{D}^k$ the sequence of Picard iterations $\{\mathscr{S}^{k,i,(q)}\}_{q \in \mathbb{N}\cup\{0\}}$, for every $k \in \overline{\mathbb{N}}$, where $\mathscr{S}^{k,i,(0)}$ is the zero element of $\mathscr{S}^{2,k}_{\hat{\beta}} \times \mathbb{H}^{2,k,\circ}_{\hat{\beta}} \times \mathbb{H}^{2,k,\natural}_{\hat{\beta}} \times \mathcal{H}^{2,k,\perp_{\mathbb{F}^{k,i}}}_{\hat{\beta}}$. 
Then, there exists $k_{\star,0}$, independent of $i$, such that\footnote{Recall the convention we made on p. \pageref{notation:short_norm} for the notation of the norms.}
\begin{align*}
    \lim_{q \rightarrow \infty} \sup_{k \geq k_{\star,0}} \left\{ \big\|\mathscr{S}^{k,i,(q)} -  \mathscr{S}^{k,i}\big\|^{2}_{\star,\hat{\beta},k,i}\right\} = 0.
\end{align*}%
Additionally, we have that $\sup_{k \geq k_{\star,0}}\left\{\|  \mathscr{S}^{k,i}\|^{2}_{\star,\hat{\beta},k,i}\right\} < \infty$.
\end{proposition}

\begin{proof}
Let us fix $i\in\mathbb{N}$.
Using \ref{S9}.\ref{S9.ii}, we can select $k_{\star,0}$ such that for every $k \geq k_{\star,0}$ we have 
$ 4 \widetilde{M}^{\Phi^k}(\hat{\beta}) < 1$. 
Then, from the triangle inequality we have for every $k \in \overline{\mathbb{N}}$ and $q \in \mathbb{N}$ that
\begin{align*}
    \big\|\mathscr{S}^{k,i,(q)} -  \mathscr{S}^{k,i}\big\|^{2}_{\star,\hat{\beta},k,i} 
    &\leq \sum_{m = 0}^{\infty} 2^{m + 1}
    \big\|\mathscr{S}^{k,i,(m + 1 + q)} -  \mathscr{S}^{k,i,(m + q)}\big\|^{2}_{\star,\hat{\beta},k,i}\\
    &\leq \sum_{m = 0}^{\infty} 2^{m + 1}
    \big(2 \widetilde{M}^{\Phi^k}(\hat{\beta})\big)^{m + q} \big\|\mathscr{S}^{k,i,(1)} -  \mathscr{S}^{k,i,(0)}\big\|^{2}_{\star,\hat{\beta},k,i}\\
    &= 2\big( 2\widetilde{M}^{\Phi^k}(\hat{\beta})\big)^q
    \big\|\mathscr{S}^{k,i,(1)}\big\|^{2}_{\star,\hat{\beta},k,i}
    \sum_{m = 0}^{\infty} \big(4 \widetilde{M}^{\Phi^k}(\hat{\beta})\big)^{m}  \\
    &= 2\frac{\left(2\widetilde{M}^{\Phi^k}(\hat{\beta})\right)^q}{1 - 4\widetilde{M}^{\Phi^k}(\hat{\beta})} \hspace{0.1cm}\|\mathscr{S}^{k,i,(1)}\|^{2}_{\star,\hat{\beta},k,i}.
\end{align*}
The right-hand side is finite. 
Indeed, we have from \citet[Lemma 3.1]{PST_BackPropagChaos} that
\begin{align*}
\big\|\mathscr{S}^{k,i,(1)}\big\|^{2}_{\star,\hat{\beta},k,i} 
    &\overset{\phantom{\ref{S7}}}{\leq} 
    \Big(26 + \frac{2}{\hat{\beta}}+ (9\hat{\beta} + 2)\Phi^k \Big) 
    \big\|\xi^{k,i}\big\|^2_{\mathbb{L}^{2,k,i}_{\hat{\beta}}(\mathbb{R}^d)}
    + \widetilde{M}^{\Phi^k}(\hat{\beta})
    \Big\| \frac{f^k(\cdot,0,0,0,\delta_0)}{\alpha^k_\cdot}\Big\|^{2}_{\mathbb{H}^{2,k,i}_{\hat{\beta}}(\mathbb{R}^d)}\\
    &\overset{\ref{S7}}{\leq} 
    \ue^{\hat{\beta}\overline{A}}\Big(26 + \frac{2}{\hat{\beta}}+ (9\hat{\beta} + 2)\Phi^k \Big) 
    \big\|\xi^{k,i}\big\|^2_{\mathbb{L}^{2,k,i}_{0}(\mathbb{R}^d)}
    + \ue^{\hat{\beta}\overline{A}}\widetilde{M}^{\Phi^k}(\hat{\beta})
    \Big\| \frac{f^k(\cdot,0,0,0,\delta_0)}{\alpha^k_\cdot}\Big\|^{2}_{\mathbb{H}^{2,k,i}_{0}(\mathbb{R}^d)},
\end{align*}
which in particular implies
\begin{align*}
\sup_{k\ge k_{\star,0}} \big\|\mathscr{S}^{k,i,(1)}\big\|^{2}_{\star,\hat{\beta},k,i} <\infty 
\end{align*}
in view of  \ref{S5}, \ref{S6} and \ref{S9}.
\end{proof}



\begin{notation}\label{Notation5.2}
In order to simplify the notation, let us define the following quantities, for $i\in\mathbb{N}$, $t \in \mathbb{R}_+\cup\{\infty\}$, $k \in \overline{\mathbb{N}}$ and $q \in \mathbb{N}\cup \{0\}$:
    \begin{itemize}
        \item $\displaystyle L^{k,i,(q)}_t := \int^{t}_{0}f^k\big(s,Y^{k,i,(q)}_s,Z^{k,i,(q)}_s c^{k}_s,\Gamma^{(\mathbb{F}^{k,i},\overline{X}^{k,i},\Theta^k)}(U^{k,i,(q)})_s,\mathcal{L}(Y^{k,i,(q)}_s)\big) \, \ud C^{k}_s$,
\vspace{0.5cm}
        \item $\displaystyle L^{k,i}_t := \int^{t}_{0}f^k\big(s,Y^{k,i}_s,Z^{k,i}_s c^{k}_s,\Gamma^{(\mathbb{F}^{k,i},\overline{X}^{k,i},\Theta^k)}(U^{k,i})_s,\mathcal{L}(Y^{k,i}_s)\big) \, \ud C^{k}_s$,\hspace{0.5em} and
        \item $\displaystyle G^{k,i,(q)} := \int^{\infty}_{0}\frac{\Big|f^k\big(s,Y^{k,i,(q)}_s,Z^{k,i,(q)}_s c^{k}_s,\Gamma^{(\mathbb{F}^{k,i},\overline{X}^{k,i},\Theta^k)}(U^{k,i,(q)})_s,\mathcal{L}(Y^{k,i,(q)}_s)\big)\Big|^2}{(\alpha^k)^2} \, \ud C^{k}_s$.%
    \end{itemize}%
\end{notation}%

Analogously to the arguments presented in \citet[Sections 3.3.1, 3.3.2 and 3.3.3]{papapantoleon2023stability}, we can conclude that the convergences \eqref{MVStab:con.1}, \eqref{MVStab:con.2} and \eqref{MVStab:con.3} follow by combining \cref{prop:uniform_Picard_MVSDE} with the next result, \textit{i.e.} \cref{lem:MVBSDE_Picard_Induction}.
In order to further clarify the argument, let us make a series of notes. 
\begin{enumerate}
\item The first step is to prove convergence \eqref{MVStab:con.1}, and to do this we just need to prove the vertical convergences depicted in \cref{table:McKean--Vlasovscheme}. 
    These are proved by induction on $q$, and the method consists of splitting $Y^{k,i,(q)}$ in the martingale and the finite variation parts. 
    Then, for the convergence of the martingale parts, we rely on the stability results for martingale representations provided in \citet{papapantoleon2019stability}, while the convergence of the finite variation parts follows from \cref{prop:uniform_Picard_MVSDE} and \cref{lem:MVBSDE_Picard_Induction}. 
\item The second step is to show convergence \eqref{MVStab:con.2}, and in that respect we again use the Moore--Osgood theorem and the splitting to the  martingale and the finite variation parts.
    The convergence of the martingale parts follows from \citet[Theorem VI.6.26]{jacod2013limit} after one checks that the so called P-UT property is satisfied finally uniformly in $k$. 
    This follows from \cref{prop:uniform_Picard_MVSDE} and \cite[Corollary VI.6.30]{jacod2013limit}. 
    Then, in order to obtain the joint convergence for the finite variation parts, the Moore--Osgood theorem is used once again. 
    The (uniform) horizontal convergence follows from \cite[Theorem VI.6.26]{jacod2013limit} after the P-UT property is verified to hold finally uniformly in $k$. 
    Now, for this purpose \cite[Comment VI.6.6]{jacod2013limit} is utilized along with \cref{prop:uniform_Picard_MVSDE} to prove the finally uniform boundedness of the total variation with respect to the $\mathbb{L}^2-$norm. 
    The vertical convergences follow from \cite[Corollary 3.10]{papapantoleon2019stability}, and the general result follows by the polarization identity.
\item The last step is to show convergence \eqref{MVStab:con.3}.
    Here, we work directly with the solutions, and not with the elements of the Picard scheme. 
    We treat first the martingale parts and then the parts related with the (special) semimartingales $Y^{k,i}$. 
    In order to prove the convergence for the martingale parts, we use \cite[Corollary 3.10]{papapantoleon2019stability}, while for the sequences related to $Y^{k,i}$ we first note that their quadratic variations are also special semimartingales. 
    Then, from \citet[Theorem 3.2.1]{lenglart1980presentation} and using the de La Vallèe-Poussin criterion, we can get the uniform integrability of the trace of the angle bracket sequences at infinity from the corresponding result for the quadratic variations. 
    Hence, the sequence associated to the angle brackets is also tight in $\mathbb{R}^{d}$. 
    Therefore, we can apply \citet[Theorem 11]{memin2003stability} which, along with standard arguments and the polarization identity, allows us to conclude.
\end{enumerate}

\begin{lemma}\label{lem:MVBSDE_Picard_Induction}
Let $q \in \mathbb{N}\cup\{0\}$, then we have
\begin{enumerate}
    \item\label{lem:MVBSDE_Picard_Induction i.} $L^{k,i,(q)}_{\infty} \xrightarrow[k \rightarrow \infty]{\mathbb{L}^2(\Omega,\mathcal{G},\mathbb{P},\mathbb{R}^d)} L^{\infty,i,(q)}_{\infty}$; 
    \vspace{0.5cm}
    \item $L^{k,i,(q)}_{\cdot} \xrightarrow[k \rightarrow \infty]{\left(\textup{J}_1(\mathbb{R}^d),\mathbb{L}^2\right)} L^{\infty,i,(q)}_{\cdot}$;
\vspace{0.5cm}
    \item\label{lem:MVBSDE_Picard_Induction iii.} The sequence $\left\{G^{k,i,(q)}\right\}_{k \in \overline{\mathbb{N}}}$ is uniformly integrable.
\end{enumerate}
\end{lemma}

\begin{proof} 
The above results are proved by induction. 
The first step of the inductions follows exactly \cite[Section 3.4]{papapantoleon2023stability}, except that in the proof of \cite[Lemma 3.12]{papapantoleon2023stability}, when we apply the Lipschitz condition of the generators we would get for every $k \in \overline{\mathbb{N}}$ that the following hold $\mathbb{P}-a.s.$\footnote{We use the inequalities $\frac{r^k}{(\alpha^k)^2}, \frac{\vartheta^{k,*}}{(\alpha^k)^2}\leq (\alpha^k)^2$ and $\frac{\vartheta^{k,\circ}}{(\alpha^k)^2},\frac{\vartheta^{k,\natural}}{(\alpha^k)^2} \leq 1$.}
\begin{align*}
&G^{k,i,(1)} 
    = \int^{\infty}_{0}\frac{\Big|f^k\big(s,Y^{k,i,(1)}_s,Z^{k,i,(1)}_s c^{k}_s,\Gamma^{(\mathbb{F}^{k,i},\overline{X}^{k,i},\Theta^k)}(U^{k,i,(1)})_s,\mathcal{L}(Y^{k,i,(1)}_s)\big)\Big|^2}{(\alpha^k)^2} \, \ud C^{k}_s\\
&
\begin{multlined}[\textwidth]
    \leq \int^{\infty}_{0}(a^k_s)^2|Y^{k,i,(1)}_s|^2 + \|Z^{k,i,(1)}_s c^{k}_s\|^2 + 2 \tnorm{U^{k,i,(1)}_s}_{k,i,s}^2\, \ud C^{k}_s\\
+ \int^{\infty}_{0}
(a^k_s)^2 \mathbb{E}\big[|Y^{k,i,(1)}_s|^2\big] 
+\frac{\left|f^k\left(s,0,0,0,\delta_0\right)\right|^2}{(\alpha^k)^2} \, \ud C^{k}_s
\end{multlined}\\
&
\begin{multlined}[\textwidth]
\leq 2 \int^{\infty}_{0}|Y^{k,i,(1)}_s|^2\, \ud A^{k}_s 
+ 2\mathbb{E}\bigg[\int^{\infty}_{0}|Y^{k,i,(1)}_s|^2\, \ud A^{k}_s\bigg] \\
+ 2\int^{\infty}_{0}\, \ud \text{Tr}\Big[
    \big\langle Z^{k,i,(1)}\cdot X^{k,i,\circ} + U^{k,i,(1)}\star \widetilde{\mu}^{{k,i,\natural}} \big\rangle_s\Big]
+ 2G^{k,i,(0)}
\end{multlined}\\
&
\begin{multlined}[\textwidth]
\leq 4 \overline{A}
\sup_{s \in\mathbb{R}_+}
\big|\mathbb{E}\big[\xi^{k,i}\big|\mathcal{F}^{k,i}_s\big]\big|^2 
+ 4 \overline{A}\sup_{s\in\mathbb{R}_+}
\bigg|\mathbb{E}\bigg[\int_{s}^{\infty}f^k(t,0,0,0,\delta_0) \ud C^{k}_t\Big|\mathcal{F}^{k,i}_s\bigg]\bigg|^2
+ 2\hspace{0.1cm}\text{Tr}\left[\langle \widetilde{M}^{k,i,(0)} \rangle_{\infty}\right] \\
+ 2\hspace{0.1cm} G^{k,i,(0)} 
+ 4\overline{A}\hspace{0.1cm}\mathbb{E}\bigg[
\sup_{s \in \mathbb{R}_+}
\big|\mathbb{E}\big[\xi^{k,i}\big|\mathcal{F}^{k,i}_s\big]\big|^2
+ \sup_{s\in\mathbb{R}_+}\Big|\mathbb{E}\Big[\int_{s}^{\infty}f^k(t,0,0,0,\delta_0) \ud C^{k}_t\Big|\mathcal{F}^{k,i}_s\Big]\Big|^2
\bigg],    
\end{multlined}
\end{align*}
where we have introduced the notation $\widetilde{M}^{k,i,(0)}_{\cdot}$ for the martingale
\begin{align*}
    \widetilde{M}^{k,i,(0)}_{\cdot}
    &:=
    Y^{k,i,(1)}_{\cdot} + \int_{0}^{\cdot}f^k\left(s,Y^{k,i,(0)}_s,Z^{k,i,(0)}_s c^{k}_s,\Gamma^{(\mathbb{F}^{k,i},\overline{X}^{k,i},\Theta^k)}(U^{k,i,(0)})_s,\mathcal{L}(Y^{k,i,(0)}_s)\right) \ud C^{k}_s\\
    &= Z^{k,i,(1)} \cdot X^{k,i,\circ}_\cdot + U^{k,i,(1)}\star \widetilde{\mu}^{k,i,\natural}_{\cdot} .
\end{align*}
Hence, to show that $\{G^{k,i,1}\}_{k \in \overline{\mathbb{N}}}$ is uniformly integrable, we will show that each of the summands in the last inequality belongs to a uniformly integrable sequence. 
Let us begin with the sequence
$$\left\{\text{Tr}\left[\langle \widetilde{M}^{k,i,(0)} \rangle_{\infty}\right]\right\}_{k \in \overline{\mathbb{N}}}.$$
The induction hypothesis yields that
$$
L^{k,i,(0)}_{\infty} \xrightarrow[k \rightarrow \infty]{\mathbb{L}^2(\Omega,\mathcal{G},\mathbb{P},\mathbb{R}^d)} L^{\infty,i,(0)}_{\infty},
$$
and because for every $k\in \overline{\mathbb{N}}$ we have $Y^{k,i,(1)}_{\infty} = \xi^{k,i}$, from \ref{S5} we get
$$
\widetilde{M}^{k,i,(0)}_{\infty} \xrightarrow[k \rightarrow \infty]{\mathbb{L}^2(\Omega,\mathcal{G},\mathbb{P},\mathbb{R}^d)} \widetilde{M}^{\infty,i,(0)}_{\infty}.
$$
Then, from \cite[Theorem 2.16]{papapantoleon2019stability} we have the convergence
$$
\text{Tr}\left[\langle \widetilde{M}^{k,i,(0)} \rangle_{\infty}\right] \xrightarrow[k \rightarrow \infty]{\mathbb{L}^1(\Omega,\mathcal{G},\mathbb{P},\mathbb{R}^d)} \text{Tr}\left[\langle \widetilde{M}^{{ \infty},i,(0)} \rangle_{\infty}\right],
$$
from which follows the desired uniform integrability property.
As for the first, fourth and fifth terms in the above inequality, our main tool will be \cite[Lemma A.17]{papapantoleon2023stability}. 
In order to apply this lemma, it would be enough that the sequences $$\{|\xi^{k,i}|^2\}_{k \in \overline{\mathbb{N}}} \hspace{0.2cm}\text{and}\hspace{0.2cm}\{G^{k,i,(0)}\}_{k \in \overline{\mathbb{N}}}$$ 
are uniformly integrable, which holds from \ref{S5} and \ref{S6}. 
The claim about the remaining two terms holds true because
\begin{align*}
&\sup_{s\in\mathbb{R}_+}
\bigg|\mathbb{E}\bigg[\int_{s}^{\infty}f^k(t,0,0,0,\delta_0) \ud C^{k}_t\Big|\mathcal{F}^{k,i}_s\bigg]\bigg|^2 \\
&\hspace{1em}= \sup_{s\in\mathbb{R}_+}
\bigg|\mathbb{E}\bigg[\int_{0}^{\infty}f^k(t,0,0,0,\delta_0) \ud C^{k}_t\Big|\mathcal{F}^{k,i}_s\bigg] - \int_{0}^{s}f^k(t,0,0,0,\delta_0) \ud C^{k}_t\bigg|^2\\
&\hspace{1em}\leq 2 \sup_{s\in\mathbb{R}_+}
\bigg|\mathbb{E}\bigg[\int_{0}^{\infty}f^k(t,0,0,0,\delta_0) \ud C^{k}_t\Big|\mathcal{F}^{k,i}_s\bigg]\Bigg|^2 + 2 \sup_{s \in \mathbb{R}_+}\Bigg|\int_{0}^{s}f^k(t,0,0,0,\delta_0) \ud C^{k}_t\bigg|^2,
\end{align*}
hence, using the Jensen and the Cauchy--Schwarz inequalities in conjunction with \ref{S7}, we conclude the required result.

Finally, we proceed with the proof that the $q-$th step of the induction is valid. 
Following \cite[Sections 3.5.1, 3.5.2, 3.5.3, 3.5.4, 3.5.5] {papapantoleon2023stability} we can conclude with exactly the same line of arguments, noting only that
in \cite[Proposition 3.19]{papapantoleon2023stability} we utilize our 
    assumption \ref{S8}.
\end{proof}


\subsection{Auxiliary results for Theorem \ref{Prop. of Chaos Stability}: Stability of backward propagation of chaos}\label{subsec:aux_PropagChaos_Stability}

In this subsection, the framework of \cref{Prop. of Chaos Stability} is adopted.
The assumptions here are the same as for \cref{McKean--Vlasov Stability}, hence we can use the results derived in \cref{subsec:aux_MVBSDE_Stability}.

\begin{lemma}\label{lemma 2.2}
Let $i\in\mathbb{N}$.
The random variables 
$\big\{\sup_{s \in \mathbb{R}_+}|Y^{k,i}_s|^2\big\}_{k \in \mathbb{N}}$ are uniformly integrable.
\end{lemma}
\begin{proof}
Let $i\in\mathbb{N}$.
Initially, we will show that the random variables 
$\big\{\sup_{s \in \mathbb{R}_+}|Y^{k,i,(q)}_s|^2\big\}_{k \in \mathbb{N}}$ are uniformly integrable,  for every $q \in \mathbb{N}$. 
This will be established by induction. 
The first step is obvious; let us recall that, using \cref{Remark 3.1}.\ref{Remark 3.1 iii} for every $k \in \mathbb{N}$, we have $C^{{k}}_t = C^{{k}}_{T^k}$ for $t \in [T^k,\infty)$. 
Using Picard's scheme, we know that for every $q \in \mathbb{N}$ and $t \in \mathbb{R}_+$ we have
\begin{align*}
     Y^{k,i,(q)}_t 
     &= \mathbb{E}\Big[\xi^{k,i} + \int^{\infty}_{t}f^k\left(s,Y^{k,i,(q-1)}_s,Z^{k,i,(q-1)}_s  c^{k}_s,\Gamma^{k,i}(U^{k,i,(q-1)}_s)_s,\mathcal{L}(Y^{k,i,(q-1)}_s)\right) \, \ud C^{k}_s\Big|\mathcal{F}^{k,i}_t\Big]\\
     &\begin{multlined}[0.9\textwidth]
         = \mathbb{E}\Big[\xi^{k,i} + \int^{\infty}_{0}f^k\left(s,Y^{k,i,(q-1)}_s,Z^{k,i,(q-1)}_s  c^{k}_s,\Gamma^{k,i}(U^{k,i,(q-1)}_s)_s,\mathcal{L}(Y^{k,i,(q-1)}_s)\right) \, \ud C^{k}_s\Big|\mathcal{F}^{k,i}_t\Big]\\
     - \int^{t}_{0}f^k\left(s,Y^{k,i,(q-1)}_s,Z^{k,i,(q-1)}_s  c^{k}_s,\Gamma^{k,i}(U^{k,i,(q-1)}_s)_s,\mathcal{L}(Y^{k,i,(q-1)}_s)\right) \, \ud C^{k}_s.
     \end{multlined}
\end{align*}
Hence we get,
\begin{multline}
\label{ineq 5.1}
\sup_{t\in \mathbb{R}_+}|Y^{k,i,(q)}_s|^2 \leq 4 \sup_{t\in \mathbb{R}_+}\big|\mathbb{E}[\xi^{k,i}|\mathcal{F}^{k,i}_t]\big|^2 \\
  +4 \sup_{t\in \mathbb{R}_+}\bigg|\mathbb{E}\bigg[ \int^{\infty}_{0}f^k\left(s,Y^{k,i,(q-1)}_s,Z^{k,i,(q-1)}_s  c^{k}_s,\Gamma^{k,i}(U^{k,i,(q-1)}_s)_s,\mathcal{L}(Y^{k,i,(q-1)}_s)\right) \, \ud C^{k}_s\bigg|\mathcal{F}^{k,i}_t\bigg]\bigg|^2\\
   +2 \sup_{t\in \mathbb{R}_+} 
   \bigg|\int^{t}_{0}f^k\left(s,Y^{k,i,(q-1)}_s,Z^{k,i,(q-1)}_s  c^{k}_s,\Gamma^{k,i}(U^{k,i,(q-1)}_s)_s,\mathcal{L}(Y^{k,i,(q-1)}_s)\right) \, \ud C^{k}_s\bigg|^2.
\end{multline}
Now, from the Cauchy--Schwarz inequality, and recalling \ref{S7}, \emph{i.e.}, $\sup_{k \in \mathbb{N}}\left\{A^{k}_{\infty}\right\} \leq \overline{A} \in \mathbb{R}_+$, 
we have for every $t \in \mathbb{R}_+ \cup \{\infty\}$
 \begin{multline*}
     \left|\int^{t}_{0}f^k\left(s,Y^{k,i,(q-1)}_s,Z^{k,i,(q-1)}_s  c^{k}_s,\Gamma^{k,i}(U^{k,i,(q-1)}_s)_s,\mathcal{L}(Y^{k,i,(q-1)}_s)\right) \, \ud C^{k}_s\right|^2\\
     \leq \overline{A} \int^{\infty}_{0}\frac{\left|f^k\left(s,Y^{k,i,(q-1)}_s,Z^{k,i,(q-1)}_s  c^{k}_s,\Gamma^{k,i}(U^{k,i,(q-1)}_s)_s,\mathcal{L}(Y^{k,i,(q-1)}_s)\right)\right|^2}{(\alpha^k_s)^2} \, \ud C^{k}_s.
 \end{multline*}
Moreover, using \cref{lem:MVBSDE_Picard_Induction}.\ref{lem:MVBSDE_Picard_Induction iii.}, \ref{S5} and \cite[Lemma A.17]{papapantoleon2023stability}, we get that the right hand side of inequality \eqref{ineq 5.1} consists of elements of sequences of uniformly integrable random variables.
Hence, our first claim is proved.

Therefore, for every $q \in \mathbb{N}$ and $\varepsilon > 0$ there exists $\delta(\varepsilon,q) > 0$ such that for every $S \in \mathcal{F}$ with $\mathbb{P}(S) \leq \delta(\varepsilon,q)$ we have
\begin{align*}
    \sup_{k \in \mathbb{N}}\mathbb{E}\big[\sup_{s \in \mathbb{R}_+}|Y^{k,i,(q)}_s|^2 \mathds{1}_S\big] \leq \frac{\varepsilon}{4}.
\end{align*}
Using \cref{prop:uniform_Picard_MVSDE}, there exists a $k_{\star,0}$ such that we can choose $q({\varepsilon}) \in \mathbb{N}$ large enough such that 
\begin{align*}
    \sup_{k \geq k_{\star,0}}\mathbb{E}\big[\sup_{s \in \mathbb{R}_+}|Y^{k,i,(q({\varepsilon}))}_s -Y^{k,i}_s|^2\big] \leq \frac{\varepsilon}{4}.
\end{align*} 
Then, we have that for every $S \in \mathcal{F}$ with $\mathbb{P}(S) \leq \delta(\varepsilon,{q({\varepsilon})})$ we get
\begin{align*}
\mathbb{E}\big[\sup_{s \in \mathbb{R}_+}|Y^{k,i}_s|^2 \mathds{1}_S\big] &= \mathbb{E}\big[\sup_{s \in \mathbb{R}_+}|Y^{k,i}_s - Y^{k,i,(q({\varepsilon}))}_s + Y^{k,i,(q({\varepsilon}))}_s|^2 \mathds{1}_S\big] \\
     &\leq 2 \mathbb{E} \big[\sup_{s \in \mathbb{R}_+}|Y^{k,i}_s - Y^{k,i,(q({\varepsilon}))}_s |^2\big] 
     + 2\mathbb{E}\big[\sup_{s \in \mathbb{R}_+}|Y^{k,i,(q({\varepsilon}))}_s|^2 \mathds{1}_S\big]
     \leq \varepsilon,
\end{align*}
for every $k \in \mathbb{N}$. 
Therefore, using the result above, the fact that $\sup_{k \in \mathbb{N}}\mathbb{E}\big[\sup_{s \in \mathbb{R}_+}|Y^{k,i}_s|^2\big] < \infty$, and recalling \cref{prop:uniform_Picard_MVSDE}, we can conclude.
\end{proof}
    
\begin{lemma}\label{lemma:convergence_tildeYkn}
Let $\widetilde{\mathbf{Y}}^{k,N}$ be the vector of the $Y$--component of the solutions of the first $N$ McKean--Vlasov BSDEs under the data $\mathscr{D}^k$, for $k\in\overline{\mathbb{N}}$ and $N\in\mathbb{N}$.
Then,
\begin{align*}
    \lim_{N\to\infty} \sup_{k\in\overline{\mathbb{N}},s\in\mathbb{R}_+}
        \mathbb{E}\Big[ W_{2,|\cdot|}^2\Big(    L^N\big({\widetilde{\textup{\textbf{Y}}}}^{k,N}_s\big), \mathcal{L}\big( Y^{k,1}_s\big) \Big)\Big] = 0.
\end{align*}
\end{lemma}

\begin{proof}
Define, for every $m\in\mathbb{N}$,
\begin{align*}
    M_2 \big( Y^{k,1}_s,m \big) 
    := \int_{B_m}|x|^2\, \textup{d} \mathcal{L}\big( Y^{k,1}_s\big)(x),
\end{align*} 
where $B_m := (-2^{m},2^m]^d \setminus (-2^{m - 1},2^{m - 1}]^d$ and also $B_0 := (-1,1]^d$. 
Let $\ell\in \mathbb{N}$, then for every $k \in \mathbb{N}$ and $s \in \mathbb{R}_+$ we have
\begin{align}
\sum_{m = \ell}^{\infty} M_2\big(Y^{k,1}_s,m\big) 
&= \int_{\bigcup_{m = \ell}^{\infty} B_m}|x|^2 \, \textup{d} \mathcal{L}\big( Y^{k,1}_s\big)(x)
    = \int_{\mathbb{R}^d}|x|^2 \mathds{1}_{[2^{\ell-1},\infty)}(x) 
    \, \textup{d} \mathcal{L}\big( Y^{k,1}_s\big)(x) \nonumber\\
&\leq \int_{\mathbb{R}^d}|x|^2 \mathds{1}_{[2^{2(\ell-1)},\infty)}(|x|^2) \, \textup{d} \mathcal{L}\big( Y^{k,1}_s\big)(x) 
    = \mathbb{E}\Big[\big|Y^{k,1}_s\big|^2 
        \mathds{1}_{\big\{|Y^{k,1}_s|^2 \geq 2^{2(\ell -1)}\big\}}
        \Big] \nonumber\\
&\leq \mathbb{E}\bigg[\sup_{s \in \mathbb{R}_+}\big\{|Y^{k,1}_s|^2\big\} \mathds{1}_{\big\{\sup_{s \in \mathbb{R}_+}|Y^{k,1}_s|^2 \geq 2^{2(\ell-1)}\big\}}\bigg] \label{2.1}.
\end{align}
Using \cref{lemma 2.2}, the random variables 
$\big\{\sup_{s \in \mathbb{R}_+}|Y^{k,1}_s|^2\big\}_{k \in \mathbb{N}}$ are uniformly integrable. 
Hence, using this information on inequality \eqref{2.1}, for every positive number $\varepsilon_1$, we can choose $\ell_1({\varepsilon_1}) \in \mathbb{N}$, which is universal with respect to $k \in \mathbb{N}$ and $s \in \mathbb{R}_+$, such that 
\begin{align}\label{aux_ineq:M2m}
    \sum_{m = \ell_1({\varepsilon_1}) + 1}^{\infty}M_2(Y^{k,1}_s,m) < \varepsilon_1.
\end{align}
Moreover, for every positive $\varepsilon_2$, there exists $\ell_2(\varepsilon_2)\in\mathbb{N}$ such that 
\begin{align}\label{aux_ineq:Series-2l}
\sum_{l = \ell_2({\varepsilon_2})+ 1}^{\infty} 2^{-2l} < \varepsilon_2.
\end{align}
Let $\mathcal{R}_0 := \sup_{k \in \mathbb{N}}\mathbb{E}\big[\sup_{s \in \mathbb{R}_+}|Y^{k,1}_s|^2\big] < \infty$. 
Then, trivially, for every $k \in \mathbb{N}$ and $s \in \mathbb{R}_+$ we have $M_2(Y^{k,1}_s,m) \leq \mathcal{R}_0.$
Using \citet[Lemma 6 and Inequality (4) on p. 716]{fournier2015rate} 
 we get
\begin{align*}
\mathbb{E}\Big[W_{2,|\cdot|}^2\Big(L^N\big({\widetilde{\textup{\textbf{Y}}}}^{k,N}_s\big),\mathcal{L}\big( Y^{k,1}_s\big)\Big)\Big]
\leq C_{d} 
\sum_{m = 0}^{\infty} 2^{2m}
\sum_{l = 0}^{\infty}2^{-2l}\min\bigg\{2\mathcal{L}\big( Y^{k,1}_s\big)(B_m),2^{\frac{dl}{2}}\Big(\frac{1}{N}\mathcal{L}\big( Y^{k,1}_s\big)(B_m)\Big)^{\frac{1}{2}}\bigg\},
\end{align*}
where $C_{d}$ is a constant that depends on the dimension. 
Trivially, for every $k \in \mathbb{N}$ and $s \in \mathbb{R}_+$, we have for $m=0$
\begin{align*}
    \mathcal{L}\big(Y^{k,1}_s\big) (B_0) 
    \leq 1
\end{align*}
and for $m \in \mathbb{N}$ 
\begin{align}\label{bound:M2m}
   \mathcal{L}\big( Y^{k,1}_s\big)(B_m) 
   = \int_{B_m}1\,\textup{d} \mathcal{L}\big( Y^{k,1}_s\big)(x)
   \leq \int_{B_m}2^{-2(m-1)}|x|^2 \,\textup{d} \mathcal{L}\big( Y^{k,1}_s\big)(x) = 2^{-2(m-1)}M_2(Y^{k,1}_s,m).
\end{align}  

Let us now fix $\varepsilon >0$ and define 
\begin{align}\label{def:varepsilon_1}
    \varepsilon_1 := \frac{\varepsilon}{33C_{d}}.
\end{align}
A natural number $\ell_1(\varepsilon_1)$ corresponds to the defined $\varepsilon_1$, such that \eqref{aux_ineq:M2m} holds.
Next, to
\begin{align}\label{def:varepsilon_2}
    \varepsilon_2 := \frac{\varepsilon}{\Big(6 + 24 \ell_1(\varepsilon_1) \mathcal{R}_0\Big)C_{d}}
\end{align}
we can associate a natural number $\ell_2(\varepsilon_2)$ such that \eqref{aux_ineq:Series-2l} holds.
Finally, we define 
\begin{align}\label{def:N_varepsilon}
N(\varepsilon) := \left \lfloor\frac{9 C_{d}^2 2^{d\ell_2(\varepsilon_2) + 2}\Big(2^{\ell_1(\varepsilon_1) + 1} \ell_1(\varepsilon_1) {\mathcal{R}_0}^{\frac{1}{2}} + 1\Big)^2}{\varepsilon^2}\right\rfloor + 1.\footnotemark 
\end{align}%
\footnotetext{We denote by $\lfloor x \rfloor$ the integer part of the real number $x$.}%
Collecting the information presented above, we can deduce that for the fixed $\varepsilon > 0$ we have (we omit the dependencies for the sake of a readable notation), for every $ N \geq N({\varepsilon})$, universal with respect to $k \in \mathbb{N}$ and $s \in \mathbb{R}_+$
\begin{align*}
&\mathbb{E}\left[
    W_{2,|\cdot|}^2
    \Big(L^N\big({\widetilde{\textup{\textbf{Y}}}}^{k,N}_s\big),
    \mathcal{L}\big( Y^{k,1}_s\big)\Big)\right]\\
&\begin{multlined}[0.9\textwidth]
    \hspace{0.2cm}\overset{\eqref{bound:M2m}}{\leq} 
    C_{d} \sum_{l=0}^{\infty}2^{-2l} \min\big\{2,2^\frac{dl}{2}N^{-\frac{1}{2}}\big\}\\
    + C_{d} \sum_{m = 1}^{\infty} 2^{2m}
    \sum_{l = 0}^{\infty}2^{-2l}
    \min\bigg\{2^{-2(m-1) + 1}
    M_2\big(Y^{k,1}_s,m\big),2^{\frac{dl}{2}}\Big(\frac{2^{-2{(m-1)}}}{N}M_2\big(Y^{k,1}_s,m\big)\Big)^{\frac{1}{2}}\bigg\}
\end{multlined}\\
&\begin{multlined}[0.9\textwidth]
\hspace{0.2cm}
\overset{\phantom{\eqref{bound:M2m}}}{\leq} 
C_{d} \Big[2
    \sum_{l=\ell_2 + 1}^{\infty}2^{-2l}
    +\frac{4}{3} 2^{\frac{d\ell_2}{2}} N^{-\frac{1}{2}} 
    \Big]
    + 11 C_{d}
    \sum_{m = \ell_1 + 1}^{\infty}
    M_2\big(Y^{k,1}_s,m\big)\\  
   + C_{d} 
   \sum_{m = 1}^{\ell_1}
   \bigg[8\mathcal{R}_0
   \sum_{l = \ell_2+ 1}^{\infty} 2^{-2l}
   + 2^{\ell_1}\frac{4}{3} 2^{\frac{d\ell_2}{2}+1}
   \Big(\frac{\mathcal{R}_0}{N}\Big)^{\frac{1}{2}}\bigg]
   \end{multlined}\\
&\begin{multlined}[0.9\textwidth]\hspace{0.2cm} 
\overset{\phantom{\eqref{bound:M2m}}}{\leq} 
C_{d} \frac{4}{3}2^{\frac{d\ell_2}{2}} N^{-\frac{1}{2}} 
    + 2 C_{d} \varepsilon_2
    + 11 C_{d} \varepsilon_1 
   + 8 C_{d} \ell_1\mathcal{R}_0\varepsilon_2 + C_{d} \ell_1 \frac{4}{3} 2^{\frac{d\ell_2}{2}+1+\ell_1}\Big(\frac{\mathcal{R}_0}{N}\Big)^{\frac{1}{2}}
   \end{multlined}\\
   &\begin{multlined}[0.9\textwidth]\hspace{0.2cm}
    \overset{\phantom{\eqref{bound:M2m}}}{\leq} 
   11 C_{d} \varepsilon_1 
   + \Big(2 + 8 \ell_1 \mathcal{R}_0\Big) C_{d} \varepsilon_2
   +  \frac{C_{d}2^{\frac{d\ell_2}{2} + 1}\Big(2^{\ell_1 + 1} \ell_1 {\mathcal{R}_0}^{\frac{1}{2}} + 1\Big)}{N^{\frac{1}{2}}}
   \end{multlined} 
    \overset{\phantom{\eqref{bound:M2m}}}{\leq} 
   \varepsilon,
\end{align*}
where the last inequality follows from \eqref{def:varepsilon_1}, \eqref{def:varepsilon_2} and \eqref{def:N_varepsilon}.    
\end{proof}

\begin{remark}
The approach we used in the proof of \Cref{lemma:convergence_tildeYkn} was inspired by the approach used in \citet{fournier2015rate}. 
There, the authors used advanced integrability assumptions in order to control the tails of the distributions and get rates of convergence for the different cases. 
In our setting we had to work with sharp square integrability; this is possible by noticing that we can bound $ \mathcal{L}\big( Y^{k,1}_s\big)(B_m)$ from the quantities $2^{-2(m-1)}M_2(Y^{k,1}_s,m)$, and thus make them controllable. 
Then, we proceeded in a similar fashion as the authors of \cite{fournier2015rate} did, while ensuring that our bounds are universal with respect to $k$ by virtue of \cref{lemma 2.2}.
\end{remark}

\begin{theorem}[Uniform propagation of chaos for the system]\label{Unif. Prop. Chaos of the system}
Assume that the setting of \cref{Prop. of Chaos Stability} is in force. 
Then, the following holds
\begin{align*}
    \lim_{N \rightarrow \infty} \sup_{k \in \mathbb{N}}\left\{
   \frac{1}{N}\sum_{i = 1}^{N} \Big\|\left(Y^{k,i,N} - Y^{k,i},Z^{k,i,N} - Z^{k,i},U^{k,i,N} - U^{k,i},M^{k,i,N} - M^{k,i}\right) \Big\|^{2}_{\star,\hat{\beta},{k,(1,\dots,N)}}\right\} = 0.
\end{align*}
\end{theorem}

\begin{proof}
Using the proof of \citet[Theorem 6.2]{PST_BackPropagChaos}, we have that for every $k,N \in \mathbb{N}$ it holds
\begin{multline*}
   \frac{1}{N}\sum_{i = 1}^{N} \Big\|\left(Y^{k,i,N} - Y^{k,i},Z^{k,i,N} - Z^{k,i},U^{k,i,N} - U^{k,i},M^{k,i,N} - M^{k,i}\right) \Big\|^{2}_{\star,\hat{\beta},{k,(1,\dots,N)}} \\\leq \frac{\Big(26 + \frac{2}{\hat{\beta}}+ (9\hat{\beta} + 2)\Phi^k \Big)}{1 - 3\widetilde{M}^{\Phi^k}(\hat{\beta})} \frac{1}{N} \sum_{i = 1}^N 
\big\|\xi^{k,i,N} - \xi^{k,i}\big\|^2_{\mathbb{L}^{2,k,(1,\dots,N)}_{\hat{\beta}}(\mathbb{R}^d)}\\ +  \frac{2 \widetilde{M}^{\Phi^k}(\hat{\beta})}{1 - 3 \widetilde{M}^{\Phi^k}(\hat{\beta})} \frac{1}{\hat{\beta}}
\int_{0}^{\infty}
    \mathbb{E}\Big[
    W_{2,|\cdot|}^2\Big(
    L^N\big({\widetilde{\textup{\textbf{Y}}}}^{k,N}_s\big),
    \mathcal{L}\big( Y^{k,1}_s\big)
    \Big)\Big] \,\ud 
    \mathcal{E}\big(\hat{\beta} A^{k}\big)_{s},
\end{multline*}
where we have denoted by $\widetilde{\textup{\textbf{Y}}}^{k,N}$ the vector of the $Y$--component of the solutions of the first $N$ McKean--Vlasov BSDEs under the data $\mathscr{D}^k$.
The first summand of the right-hand side of the above inequality converges in view of \ref{S4} and \ref{S9}. 
The convergence of the second summand comes from \cref{lemma:convergence_tildeYkn} in conjunction with \ref{S7}, which implies $\sup_{k \in \mathbb{N}}\left\{\mathcal{E}\big(\hat{\beta} A^{k}\big)_{\infty}\right\} \leq e^{\hat{\beta}\overline{A}}$.
These combined with \ref{S4} and \ref{S9} allow us to conclude.
\end{proof}

\begin{theorem}[Uniform propagation of chaos]\label{Unif. Prop. Chaos}
Assume that the setting of \cref{Prop. of Chaos Stability} is in force. 
Then, for every $i \in \mathbb{N}$, the following holds
\begin{align*}
    \lim_{N \rightarrow \infty} \sup_{k \in \mathbb{N}}
    \Big\|\Big(Y^{k,i,N} - Y^{k,i},Z^{k,i,N} - Z^{k,i},U^{k,i,N} - U^{k,i},M^{k,i,N} - M^{k,i}\Big) \Big\|^{2}_{\star,\hat{\beta},{k,(1,\dots,N)}} = 0.
\end{align*}
\end{theorem}

\begin{proof}
Let $i\in\mathbb{N}$.
Using again the proof of \cite[Theorem 6.3]{PST_BackPropagChaos}, we get that
\begin{multline*}
\Big\|\big(Y^{k,i,N} - Y^{k,i},Z^{k,i,N} - Z^{k,i},U^{k,i,N} - U^{k,i},M^{k,i,N} - M^{k,i}\big) \Big\|^{2}_{\star,\hat{\beta},k,(1,\dots,N)} \nonumber\\
    \leq  \frac{\Big(26 + \frac{2}{\hat{\beta}}+ (9\hat{\beta} + 2)\Phi^k \Big)}{1-2\widetilde{M}^{\Phi^k}(\hat{\beta})}
    \big\|\xi^{k,i,N} - \xi^{k,i}\big\|^2_{\mathbb{L}^{2,k,(1,\dots,N)}_{\hat{\beta}}(\mathbb{R}^d)}\\
    + \frac{2 \widetilde{M}^{\Phi^k}(\hat{\beta})}{1 - 2 \widetilde{M}^{\Phi^k}(\hat{\beta})}
     \frac{1}{N}\sum_{i = 1}^{N} \Big\|\left(Y^{k,i,N} - Y^{k,i},Z^{k,i,N} - Z^{k,i},U^{k,i,N} - U^{k,i},M^{k,i,N} - M^{k,i}\right) \Big\|^{2}_{\star,\hat{\beta},{k,(1,\dots,N)}}\\
    + \frac{2 \widetilde{M}^{\Phi^k}(\hat{\beta})}{1 - 2 \widetilde{M}^{\Phi^k}(\hat{\beta})}
    \frac{1}{\hat{\beta}}
    \int_{0}^{\infty}
    \mathbb{E}\Big[
    W_{2,|\cdot|}^2\Big(
    L^N\big({\widetilde{\textup{\textbf{Y}}}}^{k,N}_s\big),
    \mathcal{L}\big( Y^{k,1}_s\big)
    \Big)\Big] \,\ud 
    \mathcal{E}\big(\hat{\beta} A^{k}\big)_{s}.
\end{multline*}
Then, from \ref{S4}, \ref{S9}, \cref{Unif. Prop. Chaos of the system} and \cref{lemma:convergence_tildeYkn} we can conclude.
\end{proof}

\begin{corollary}\label{Corollary 5.7}
Define, for every $k\in\overline{\mathbb{N}}$ and  $i,N \in \mathbb{N}$,
\begin{align*}
      L^{k,i,N}_\cdot := \int^{\cdot}_{0}f^k\left(s,Y^{k,i,N}_s,Z^{k,i,N}_s c^{k}_s,\Gamma^{k,(1,\dots,N)}(U^{k,i,N})_s,L^N(\textup{\textbf{Y}}^{k,N}_s)\right) \, \ud C^{k}_s.
\end{align*}
Then, the following are true:\footnote{The reader may also recall \ref{Notation5.2}.}
\begin{align}
  \label{Corollary 5.7. i}
  \sup_{k \in \mathbb{N}}
  \mathbb{E}\Big[\big|\textrm{\emph{Var}}\big[ L^{k,i,N} - L^{k,i}\big]_{\infty}
  \big|^2\Big] \xrightarrow[N \rightarrow \infty]{|\cdot|} 0,
\end{align}
where the total variation is calculated per coordinate,
\begin{gather}
    \label{Corollary 5.7. 00}
    \sup_{k\in \mathbb{N}}
    \mathbb{E}\Big[\big\|[L^{k,i,N}]_{\cdot} - [L^{k,i}]_{\cdot}\big\|_{\infty}\Big]
    \xrightarrow[N \rightarrow \infty]{|\cdot|}0,\\
    \label{Corollary 5.7 0}
    [L^{k,i,N}] \xrightarrow[(k,N) \rightarrow (\infty,\infty)]{\left(\textup{J}_1(\mathbb{R}^{d \times d}),  \mathbb{L}^1\right)} 0
\shortintertext{and}
    \label{Corollary 5.7. ii}
    Z^{k,i,N} \cdot X^{k,i,\circ}_{\infty} + U^{k,i,N} \star \widetilde{\mu}^{{k,i,\natural}}_{\infty} + M^{k,i,N}_{\infty} \xrightarrow[(k,N)\rightarrow (\infty,\infty)]{\mathbb{L}^2(\Omega,\mathcal{G},\mathbb{P},\mathbb{R}^d)} Z^{\infty,i} \cdot X^{\infty,i,\circ}_{\infty} + U^{\infty,i} \star \widetilde{\mu}^{{\infty,i,\natural}}_{\infty}.
\end{gather}
Finally, the sequence of random variables $\Big\{\big|Z^{k,i,N} \cdot X^{k,i,\circ}_{\infty} + U^{k,i,N} \star \widetilde{\mu}^{{k,i,\natural}}_{\infty} + M^{k,i,N}_{\infty}\big|^2\Big\}_{k \in \overline{\mathbb{N}}, N \in \mathbb{N}}$ is uniformly integrable.
\end{corollary}

\begin{proof} 
In order to prove \eqref{Corollary 5.7. i}, we start by reminding the facts that 
\begin{align*}
 \frac{r^k}{(\alpha^k)^2}\leq (\alpha^k)^2, \quad
 \vartheta^{k,\circ} \leq (\alpha^k)^2, \quad 
 \vartheta^{k,\natural} \leq (\alpha^k)^2 
    \quad \text{ and }\quad 
 \frac{\vartheta^{k,*}}{(\alpha^k)^2} \leq (\alpha^k)^2.
\end{align*}
Define for every $s \in \mathbb{R}_+$
\begin{align*}
  \delta f^{k,N}_s:= f^k\left(s,Y^{k,i,N}_s,Z^{k,i,N}_s c^{k}_s,\Gamma^{k,i}(U^{k,i,N})_s,L^N(\textbf{Y}^{k,N}_s)\right) - f^k\left(s,Y^{k,i}_s,Z^{k,i}_s c^{k}_s,\Gamma^{k,i}(U^{k,i})_s,\mathcal{L}(Y^{k,i}_s)\right).
\end{align*}
Using the Cauchy--Schwarz inequality, \ref{H4} and \ref{S7}, we get for every $k \in \mathbb{N}$ that
\begin{align*}
    &\mathbb{E}\Big[\Big|\textrm{Var}\big[ L^{k,i,N}_{\cdot} - L^{k,i}_{\cdot}\big]\Big|^2\Big] 
    = \mathbb{E}\Big[\Big|\textup{Var}\Big[\int_{0}^{\cdot}\delta f^{k,N}_s \ud C^k_s\Big]\Big|^2\Big] 
    \leq \overline{A}\hspace{0.1cm}\mathbb{E}\left[\int_{0}^{\infty}\mathcal{E}\big(\hat{\beta} A^{k}\big)_{s-}\frac{|\delta f^{k,N}_s|^2}{(a^k_s)^2} \ud C^k_s\right]\\
    &\begin{multlined}[0.9\textwidth]
    \hspace{1em}\leq \overline{A}\hspace{0.1cm}\Bigg( \|\alpha^k (Y^{k,i,N} - Y^{k,i}) \|^2_{\mathbb{H}_{\hat{\beta}}^{2,k,(1,\dots,N)}(\mathbb{R}^d)} 
        + \|Z^{k,i,N} - Z^{k,i}\|^2_{\mathbb{H}^{2,k,(1,\dots,N),\circ}_{\hat{\beta}}(\mathbb{R}^{d \times p})}\\ + \|U^{k,i,N} - U^{k,i}\|^2_{\mathbb{H}^{2,k,(1,\dots,N),\natural}_{\hat{\beta}}(\mathbb{R}^d)} + \frac{1}{\hat{\beta}}\hspace{0.1cm}\mathbb{E}\left[\int_{0}^{\infty}W^2_{2,|\cdot|}(L^N(\textbf{Y}^{k,N}_s),\mathcal{L}(Y^{k,i}_s))\ud 
    \mathcal{E}\big(\hat{\beta} A^{k}\big)_{s}\right]\Bigg).
    \end{multlined}
\end{align*}
In view of \cref{Unif. Prop. Chaos of the system} and \ref{Unif. Prop. Chaos}, we only need to prove the convergence of the last summand on the right-hand side of the last inequality.
Using the triangle inequality for the Wasserstein distance and Tonelli's theorem, we have
\begin{multline*}
    \mathbb{E}\left[\int_{0}^{\infty}W^2_{2,|\cdot|}(L^N(\textbf{Y}^{k,N}_s),\mathcal{L}(Y^{k,i}_s))\ud 
    \mathcal{E}\big(\hat{\beta} A^{k}\big)_{s} \right]\\ 
    \leq 2\mathbb{E}\left[\int_{0}^{\infty}W^2_{2,|\cdot|}(L^N(\textbf{Y}^{k,N}_s),L^N(\widetilde{\textup{\textbf{Y}}}^{k,N}_s))
    \ud \mathcal{E}\big(\hat{\beta} A^{k}\big)_{s}\right]\\ 
    + 2 \int_{0}^{\infty}
    \mathbb{E}\Big[
    W_{2,|\cdot|}^2\Big(
    L^N\big({\widetilde{\textup{\textbf{Y}}}}^{k,N}_s\big),
    \mathcal{L}\big( Y^{k,1}_s\big)
    \Big)\Big] \,\ud 
    \mathcal{E}\big(\hat{\beta} A^{k}\big)_{s}\footnotemark.
\end{multline*}
\footnotetext{Note that $ \mathcal{L}\big( Y^{k,1}_s\big) = \mathcal{L}\big( Y^{k,i}_s\big)$ for every $s \in \mathbb{R}_+$ and $k \in \overline{\mathbb{N}},i \in \mathbb{N}$.}%
On the right-hand side of the above inequality, the convergence of the first summand follows from \cref{Unif. Prop. Chaos of the system} in conjunction with \ref{S7}, once we observe that by \eqref{empiricalineq} holds
\begin{gather*}
    W^2_{2,|\cdot|}\Big(L^N\big(\textbf{Y}^{k,N}_s\big),L^N\big(\widetilde{\textup{\textbf{Y}}}^{k,N}_s\big)\Big) 
        \leq \frac{1}{N}\sum_{i = 1}^{N}\big|Y^{k,i,N}_s - Y^{k,i}_s \big|^2.
\end{gather*}
The convergence of the second summand follows from \cref{lemma:convergence_tildeYkn} in conjunction with \ref{S7}.

In order to show \eqref{Corollary 5.7. 00}, we argue as follows. 
Let $j\in\{1,\dots,d\}$, then we will adjoin the upper index $[j]$ in order to denote the $j-$th element of a $d-$dimensional vector.
\begin{align*}
    \textup{Tr}\big[ [L^{k,i,N} - L^{k,i}]_{\infty}\big]
    &= \sum_{j=1}^d \sum_{t\ge 0}\big(\Delta (L^{k,i,N,[j]} - L^{k,i,[j]})\big)^2
    \leq \sum_{j=1}^d \Big(\sum_{t\ge 0}\Delta |L^{k,i,N,[j]} - L^{k,i,[j]}|\Big)^2\\
    &\leq \big| \textup{Var}[L^{k,i,N} - L^{k,i}]_{\infty}\big|^2.
\end{align*}
Now, the pathwise identity
\begin{align*}
    [L^{k,i,N}] - [L^{k,i}] = [L^{k,i,N} - L^{k,i}] + 2 [L^{k,i,N} - L^{k,i}, L^{k,i}]
\end{align*}
implies for every $t\in\mathbb{R}_+$ (we understand the following elementwise)
\begin{align*}
    \textup{Var}\big([L^{k,i,N}]_t - [L^{k,i}]_t\big) 
    &= \textup{Var}\big([L^{k,i,N} - L^{k,i}]_t + 2 [L^{k,i,N} - L^{k,i}, L^{k,i}]_t\big)\\
    &\leq
    \textup{Var}\big([L^{k,i,N} - L^{k,i}]_{\infty}\big) + 2 \textup{Var}\big([L^{k,i,N} - L^{k,i}, L^{k,i}]_{\infty}\big).
\end{align*}
Then, the desired convergence \eqref{Corollary 5.7. 00} is derived by the convergence \eqref{Corollary 5.7. i}, the Kunita--Watanabe inequality, see \citet[Theorem 6.33]{he2019semimartingale}, and the $\mathbb{L}^1-$boundedness of $\{ [L^{k,i}]_{\infty} \}_{k\in\overline{\mathbb{N}}}$.

The convergence \eqref{Corollary 5.7 0} is immediate from the convergence \eqref{Corollary 5.7. 00} if the next limit holds 
\begin{align*}
    [L^{k,i}] \xrightarrow[k \rightarrow \infty]{\left(\textup{J}_1(\mathbb{R}^{d \times d}),  \mathbb{L}^1\right)} 0.
\end{align*}
However, this follows by using for the McKean--Vlasov case the exact same reasoning as in the analogous results of \citet[Section 3.3.2]{papapantoleon2023stability}, \emph{i.e.},
\begin{align*}
[L^{k,i,(q)}] \xrightarrow[q \rightarrow \infty]{\left(\|\cdot\|_{\infty},  \mathbb{L}^1\right)} [L^{k,i}]
 \hspace{0.5cm}\text{and}\hspace{0.5cm}  [L^{k,i,(q)}] \xrightarrow[(k,q) \rightarrow (\infty,\infty)]{\left(\textup{J}_1(\mathbb{R}^{d \times d}),  \mathbb{L}^1\right)} 0. 
\end{align*}
Finally, we want to prove convergence \eqref{Corollary 5.7. ii}.
To this end, we have from \ref{S5}, \cref{lem:MVBSDE_Picard_Induction}.\ref{lem:MVBSDE_Picard_Induction i.} and \cref{prop:uniform_Picard_MVSDE} by an application of the Moore--Osgood theorem that
\begin{align}\label{eq 5.5}
    Z^{k,i} \cdot X^{k,i,\circ}_{\infty} + U^{k,i} \star \widetilde{\mu}^{{k,i,\natural}}_{\infty} + M^{k,i}_{\infty} \xrightarrow[k\rightarrow \infty]{\mathbb{L}^2(\Omega,\mathcal{G},\mathbb{P},\mathbb{R}^d)} Z^{\infty,i} \cdot X^{\infty,i,\circ}_{\infty} + U^{\infty,i} \star \widetilde{\mu}^{{\infty,i,\natural}}_{\infty}.
\end{align}
Applying again the Moore--Osgood theorem,  we conclude from \eqref{eq 5.5} and \cref{Unif. Prop. Chaos}.
The desired uniform integrability property is immediate from Vitalli's theorem, see \citet[Theorem 4.5.4]{bogachev2007measure} which can be adapted for the $\mathbb{L}^2-$case.
\end{proof}





\subsection{Auxiliary results for Theorem \ref{Mean-field system stability}: Stability of mean-field BSDE systems}\label{subsec:aux_mfBSDE_Stability}

In this subsection, the framework of \cref{Mean-field system stability} is adopted. 
We present below, for the convenience of the reader, a table that outlines the method of proof for the stability of mean-field BSDEs with $N$ players. 
The approach is exactly the same as in the case of the stability of McKean--Vlasov BSDEs, see \cref{table:McKean--Vlasovscheme} and the discussion thereafter, therefore we avoid the repetition by omitting the respective description.

\begin{table}[ht]
{\centering 
\begin{tabular}{c|ccccccc}
$\mathscr{D}^1$ 	&{$\textbf{S}^{1,N,(0)}$} 			&{$\textbf{S}^{1,N,(1)}$} 			&{$\textbf{S}^{1,N,(2)}$}		
 					&{$\cdots$}							&{$\textbf{S}^{1,N,(q)}$}			&{$\xrightarrow{\ q\to\infty\ }$} 	&{$\textbf{S}^{1,N}$}\\[0.3cm]
$\mathscr{D}^2$ 	&{$\textbf{S}^{2,N,(0)}$} 			&{$\textbf{S}^{2,N,(1)}$} 			&{$\textbf{S}^{2,N,(2)}$}		
 					&{$\cdots$}							&{$\textbf{S}^{2,N,(q)}$}			&{$\xrightarrow{\ q\to\infty\ }$} 	&{$\textbf{S}^{2,N}$}\\[0.3cm]
$\mathscr{D}^3$ 	&{$\textbf{S}^{3,N,(0)}$} 			&{$\textbf{S}^{3,N,(1)}$} 			&{$\textbf{S}^{3,N,(2)}$}		
 					&{$\cdots$}							&{$\textbf{S}^{3,N,(q)}$}			&{$\xrightarrow{\ q\to\infty\ }$} 	&{$\textbf{S}^{3,N}$}\\[0.1cm]
$\vdots$			&{$\vdots$} 						&{$\vdots$} 						&{$\vdots$} 						& 
					&{$\vdots$}							& 												&{$\vdots$} \\[0.2cm]
$\mathscr{D}^k$ 	&{$\textbf{S}^{k,N,(0)}$} 			&{$\textbf{S}^{k,N,(1)}$} 			&{$\textbf{S}^{k,N,(2)}$}			&{$\cdots$}
					&{$\textbf{S}^{k,N,(q)}$}			&{$\xrightarrow{\ q\to\infty\ }$}	&{$\textbf{S}^{k,N}$}\\
$\big\downarrow$	
&\rotatebox[origin=c]{270}{$\dashrightarrow{}$}	&\rotatebox[origin=c]{270}{$\dashrightarrow{}$}  	&\rotatebox[origin=c]{270}{$\dashrightarrow{}$}  	&  		
					& \rotatebox[origin=c]{270}{$\dashrightarrow{}$} 	&									& \rotatebox[origin=c]{270}{\Large${\rightsquigarrow}$} \\[0.2cm]
$\mathscr{D}^\infty$&{$\textbf{S}^{\infty,N,(0)}$} 		&{$\textbf{S}^{\infty,N,(1)}$} 		&{$\textbf{S}^{\infty,N,(2)}$}		&{$\cdots$}
					&{$\textbf{S}^{\infty,N,(q)}$}		&{$\xrightarrow{\ q\to\infty\ }$}	&{$\textbf{S}^{\infty,N}$}\\[0.2cm]
\end{tabular}\\}
\caption{The doubly-indexed Picard scheme for the stability of mean-field systems of BSDEs.}
\label{table:mean-fieldscheme}
\end{table}

\begin{proposition}[Uniform \emph{a priori} mean-field BSDE system estimates]
Let $k \in \overline{\mathbb{N}}$ and $N \in \mathbb{N}$.
Then, we associate to the standard data $\mathscr{D}^k$ the sequence of Picard iterations $\{\textbf{S}^{k,N,(q)}\}_{q \in \mathbb{N}\cup\{0\}}$, where $\textbf{S}^{k,N,(0)}$ is the zero element of 
\begin{align*}
    \prod_{i = 1}^{N}\mathscr{S}^{2,k,(1,\dots,N)}_{\hat{\beta}} 
\times \prod_{i = 1}^{N}\mathbb{H}^{2,k,(1,\dots,N),\circ}_{\hat{\beta}} 
\times \prod_{i = 1}^{N}\mathbb{H}^{2,k,(1,\dots,N),\natural}_{\hat{\beta}} 
\times \prod_{i = 1}^{N}\mathcal{H}^{2,k,(1,\dots,N),\perp}_{\hat{\beta}}.
\end{align*} 
There exists $k_{\star,0}$ such that 
\begin{align*}
    \lim_{q \rightarrow \infty} \sup_{k \geq k_{\star,0}}
    \big\|\textbf{S}^{k,N,(q)} -  \textbf{S}^{k,N}\big\|^{2}_{\star,k,(1,\dots,N),\hat{\beta}} = 0.
\end{align*}
Additionally, we also have $\sup_{k \geq k_{\star,0}}\big\|  \textbf{S}^{k,N}\|^{2}_{\star,k,(1,\dots,N),\hat{\beta}} < \infty$.
\end{proposition}

\begin{proof}
As the reader can confirm, one can follow, \emph{mutatis mutandis}, the same arguments as in the proof of \cref{prop:uniform_Picard_MVSDE}, which deals with the stability of McKean--Vlasov BSDEs.
\end{proof}

In order to prove \cref{Mean-field system stability}, we will follow the same strategy as in the proof of \cref{McKean--Vlasov Stability}, but working now for the $N-$player system of BSDEs.
Hence, we need to take care of the appropriate modifications. 

\begin{notation}\label{Definition 5.9}
For every $i \in \mathscr{N}$, $t \in \mathbb{R}_+\cup\{\infty\}$, $k \in \overline{\mathbb{N}}$ and $q \in \mathbb{N}\cup \{0\}$ we define:
    \begin{itemize}
        \item $\displaystyle L^{k,i,N,(q)}_t := \int^{t}_{0}f^k\big(s,Y^{k,i,N,(q)}_s,Z^{k,i,N,(q)}_s c^{k}_s,\Gamma^{k,(1,\dots,N)}(U^{k,i,N,(q)})_s,L^N(\textbf{Y}^{k,N,(q)}_s)\big) \, \ud C^{k}_s$,
\vspace{0.5cm}
        \item $\displaystyle L^{k,i,N}_t := \int^{t}_{0}f^k\big(s,Y^{k,i,N}_s,Z^{k,i,N}_s c^{k}_s,\Gamma^{k,(1,\dots,N)}(U^{k,i,N})_s,L^N(\textbf{Y}^{k,N}_s)\big) \, \ud C^{k}_s$, 
        and
        \vspace{0.5cm}
        \item $\displaystyle G^{k,i,N,(q)} := \int^{\infty}_{0}\frac{\Big|f^k\big(s,Y^{k,i,N,(q)}_s,Z^{k,i,N,(q)}_s c^{k}_s,\Gamma^{k,(1,\dots,N)}(U^{k,i,N,(q)})_s,L^N(\textbf{Y}^{k,N,(q)}_s)\big)\Big|^2}{(\alpha^k)^2} \, \ud C^{k}_s$.
    \end{itemize}
\end{notation}

Then, as before, using \citet[Sections 3.3.1, 3.3.2 and 3.3.3]{papapantoleon2023stability} we conclude that the convergences of \cref{Mean-field system stability} are equivalent to the next result.

\begin{lemma}\label{lem:MVBSDE_Picard_Induction 2}
Let $i\in \mathscr{N}$ and $q \in \mathbb{N}\cup\{0\}$, then we have
\begin{enumerate}
    \item $L^{k,i,N,(q)}_{\infty} \xrightarrow[k \rightarrow \infty]{\mathbb{L}^2(\Omega,\mathcal{G},\mathbb{P},\mathbb{R}^d)} L^{\infty,i,N,(q)}_{\infty}$, 
    \vspace{0.5cm}
    \item $L^{k,i,N,(q)}_{\cdot} \xrightarrow[k \rightarrow \infty]{\left(\textup{J}_1(\mathbb{R}^d),\mathbb{L}^2\right)} L^{\infty,i,N,(q)}_{\cdot}$,
\vspace{0.5cm}
    \item\label{lem:MVBSDE_Picard_Induction iii. 2} the sequence $\left\{G^{k,i,N,(q)}\right\}_{k \in \overline{\mathbb{N}}}$ is uniformly integrable.
\end{enumerate}
\end{lemma}

The above result is proved by induction, with exactly the same arguments as \cref{lem:MVBSDE_Picard_Induction}, but note that in each step of the induction we treat all $i \in \mathscr{N}$ simultaneously. 
As an example, we will provide the computations needed in the proof of \cite[Lemma 3.12.]{papapantoleon2023stability}, as we did at \cref{lem:MVBSDE_Picard_Induction} above. 
Hence, for $i \in \mathscr{N}$ and $k \in \overline{\mathbb{N}}$, we have $\mathbb{P}-a.s.$
\begin{align*}
    G^{k,i,N,(1)} 
    &= \int^{\infty}_{0}\frac{\Big|f^k\big(s,Y^{k,i,N,(1)}_s,Z^{k,i,N,(1)}_s c^{k}_s,\Gamma^{k,(1,\dots,N)}(U^{k,i,N,(1)})_s,L^N(\textbf{Y}^{k,N,(1)}_s)\big)\Big|^2}{(\alpha^k)^2} \, \ud C^{k}_s\\
    &\hspace{-2em}\begin{multlined}[0.95\textwidth]
        \leq \int^{\infty}_{0}(a^k_s)^2|Y^{k,i,N,(1)}_s|^2 + \|Z^{k,i,N,(1)}_s c^{k}_s\|^2 + 2 \Big(\tnorm{U^{k,i,N,(1)}_s(\cdot)}_s^{(\mathbb{F}^{k,(1,\dots,N)},\overline{X}^{k,i})}\Big)^2\, \ud C^{k}_s\\ +\int^{\infty}_{0} (a^k_s)^2 \frac{1}{N}\sum_{m =1}^{N}|Y^{k,m,N,(1)}_s|^2
    +\frac{\left|f^k\left(s,0,0,0,\delta_0\right)\right|^2}{(\alpha^k)^2} \, \ud C^{k}_s
    \end{multlined}\\
    &\hspace{-2em}\begin{multlined}[0.95\textwidth]
    \leq 2 \int^{\infty}_{0}|Y^{k,i,N,(1)}_s|^2 + \frac{1}{N}\sum_{m =1}^{N}|Y^{k,m,N,(1)}_s|^2\, \ud A^{k}_s  \\
    +2 \int^{\infty}_{0}\, \ud \text{Tr}\left[\left\langle Z^{k,i,N,(1)}\cdot X^{k,i,\circ} + U^{k,i,N,(1)}\star \widetilde{\mu}^{{k,i,\natural}} \right\rangle_s\right]
    + 2G^{k,i,N,(0)}
    \end{multlined}\\
    &\hspace{-2em}\begin{multlined}[0.95\textwidth]
        \leq 4 \overline{A}\sup_{s \in \mathbb{R}_+}\left\{\left|\mathbb{E}\left[\xi^{k,i,N}\Big|\mathcal{F}^{k,(1,\dots,N)}_s\right]\right|^2\right\} + \frac{4 \overline{A}}{N}\sum_{m = 1}^{N} \sup_{s \in \mathbb{R}_+}\left\{\left|\mathbb{E}\left[\xi^{k,m,N}\Big|\mathcal{F}^{k,(1,\dots,N)}_s\right]\right|^2\right\} \\ 
    + 8 \overline{A}\sup_{s\in\mathbb{R}_+}\left\{\left|\mathbb{E}\left[\int_{s}^{\infty}f^k(t,0,0,0,\delta_0) \ud C^{k}_t\Bigg|\mathcal{F}^{k,(1,\dots,N)}_s\right]\right|^2\right\} + 2\hspace{0.1cm}\text{Tr}\left[\langle \widetilde{M}^{k,i,N,(0)} \rangle_{\infty}\right] 
    + 2\hspace{0.1cm} G^{k,i,N,(0)},
    \end{multlined}
\end{align*}
with $\widetilde{M}^{k,i,N,(0)}_{\cdot}$ defined to be the martingale
\begin{align*}
    Y^{k,i,N,(1)}_{\cdot} + \int_{0}^{\cdot}f^k\left(s,Y^{k,i,N,(0)}_s,Z^{k,i,N,(0)}_s c^{k}_s,\Gamma^{k,(1,\dots,N)}(U^{k,i,N,(0)})_s,L^N(\textbf{Y}^{k,N,(0)}_s)\right) \ud C^{k}_s.
\end{align*}
Hence, from the information provided, for every $i \in \mathscr{N}$, we can conclude as before in the McKean--Vlasov case.


\bibliographystyle{abbrvnat}
\bibliography{References}


\end{document}